\documentclass[11pt]{article}
\usepackage[utf8]{inputenc}
\usepackage[T1]{fontenc}    

\usepackage{ae,aecompl}     

\usepackage{latexsym}
\usepackage{color, graphicx, rotating, a4}
\usepackage{amsmath, amsfonts,amssymb, amsthm, mathtools, mathrsfs}
 \usepackage{hyperref}

\newcommand{\bR}{\mathbf{R}}

\newcommand{\bE}{\mathbf{E}}
\newcommand{\bP}{\mathbf{P}}
\newcommand{\bN}{\mathbf{N}}
\newcommand{\1}{\mathbf{1}}

\newcommand{\cT}{\mathcal{T}}

\newcommand{\cM}{\mathcal{M}}
\newcommand{\cP}{\mathcal{P}}

\newcommand{\cC}{\mathcal{C}}
\newcommand{\cD}{\mathcal{D}}

\newcommand{\cO}{\mathcal{O}}
\newcommand{\ep}{\epsilon}
\newcommand{\vep}{\varepsilon}

\newcommand{\wth}{\widetilde{h}}

\newcommand{\cX}{\mathcal{X}}
\newcommand{\cMs}{\cM_{\cT_{L}}}

\newtheorem{proposition}{Proposition}[section]

\newtheorem{lemma}[proposition]{Lemma}
\newtheorem{theorem}[proposition]{Theorem}

\theoremstyle{definition}

\newtheorem{definition}[proposition]{Definition}
\newtheorem{remark}[proposition]{Remark}
\newtheorem{example}[proposition]{Example}
\usepackage[normalem]{ulem}
\numberwithin{equation}{section}

\usepackage{colortbl}
\definecolor{darkgrey}{rgb}{0.75,0.75,0.75}

\newcommand\unnumberedfootnote[1]{ %
  \let\temp=\thefootnote %
  \renewcommand{\thefootnote}{}%
  \footnote{#1}%
  \let\thefootnote=\temp%
  \addtocounter{footnote}{-1}}

\makeatletter
\renewcommand{\@fnsymbol}[1]{\ensuremath{%
    \ifcase#1\or 1\or 2\or 3\or
    \mathsection\or \mathparagraph\or \|\or 1\or
    2\or 3 \else\@ctrerr\fi}}
\makeatother

\begin{document}
\title{\Large A spatial measure-valued model for chemical reaction
  networks in heterogeneous systems} \author{Lea
  Popovic\thanks{Department of Mathematics and Statistics Concordia
    University, Montreal QC H3G 1M8, Canada, e-mail:
    lea.popovic@concordia.ca}, Amandine V\'eber\thanks{MAP5, CNRS, Universit\'e de Paris, 45 rue des Saints P\`eres, 75006 Paris, France, e-mail: amandine.veber@parisdescartes.fr}}
    \thispagestyle{empty}

 \date{\today}
\maketitle
\unnumberedfootnote{\emph{AMS 2000 subject classification.} 60J27,
  60F17, 92C45, 80A30}

\unnumberedfootnote{\emph{Keywords and phrases.} Measure valued process, chemical reaction
  network, multi-scale process, scaling limits, reaction-diffusion process, piecewise deterministic Markov process}

\vspace{-1.75cm}\begin{abstract}
  \noindent
  We propose a novel measure valued process which models the behaviour of chemical reaction networks in spatially heterogeneous systems.
  It models reaction dynamics between different molecular species and continuous movement of molecules in space.
  Reactions rates at a spatial location are proportional to the mass of different species present locally and to a location specific chemical rate, which may be a function of the local or global species mass as well.
 We obtain asymptotic limits for the process, with appropriate rescaling depending on the abundance of different molecular types.
 In particular, when the mass of some species in the scaling limit is discrete while the mass of the others is continuous, we obtain a new type of spatial random evolution process. This process can be shown, in some situations, to correspond to a measure-valued piecewise deterministic Markov process in which the discrete mass of the process evolves stochastically, and the continuous mass evolves in a deterministic way between consecutive jump times of the discrete part.
\end{abstract}

\tableofcontents

\section{Introduction}\label{section:intro}

The goal of our work is to establish a mathematical framework for the dynamics of molecules of different types interacting and moving in continuous heterogeneous space. The model represents interactions that depend on both spatial location and amounts of other species, and the framework allows us to obtain results for scaling limits in scenarios of different species abundances. Our work is motivated by models for intracellular mechanisms in terms of biochemical reaction networks, but can easily be used in many other applications.

Spatial location of biochemical species within a cell plays a pivotal role in the dynamics of many key intracellular mechanisms. For example, protein movement between the nucleus and cytoplasm affects cellular responses (proteins must be present in
the nucleus to regulate their target genes).  Signalling proteins need to shuttle from the plasma membrane to cytoplasm and the nucleus to turn genes on or off and ultimately induce a response (spatial movement and organization is paramount to  signal transduction processes).

Modelling frameworks combining spatial dynamics and biochemical reaction networks have been made at different levels of detail using deterministic and stochastic objects. The reaction-diffusion framework uses partial differential equations to model concentrations of species where  reactions and movement produce deterministic changes continuously in both time and space.  This framework is not appropriate for reaction dynamics which rely on changes due to molecular types in low abundances which may be localized, since it assumes the same rate of motion for all species and the same concentration scaling for all species.

The compartment framework (also referred to as the reaction-diffusion master equation, RDME) counts the number of molecules of different species in subdivided partitions of space, and in each subdivision (compartment) reactions and movement produce stochastic changes based on a Markov chain whose rates depend on the species counts in the compartment. Reactions are allowed only between species within the same compartment, and movement can occur only between different compartments. This framework is useful for speeding up simulations of stochastic dynamics (see \cite{SmartCell, MesoRD} for examples of variants  of algorithms using RDME). The problem with this framework is that it assumes homogeneity of reaction dynamics within each compartment hence their size cannot be too large, and since reactions are allowed only within a compartment their size can also not be too small (isolating molecules). This makes the choice of partition size a challenge in many situations, in particular if the orders of abundances of different species vary.

The Brownian dynamics framework keeps track of individual molecules where each moves by an independent Brownian diffusion and participates in reactions if molecules of other source types are at close enough binding distance from it. This framework gives a more detailed account of the stochastic behaviour. Consequently, simulations in this framework are more intensive (see \cite{Smoldyn, VirtualCell} for examples of algorithms using Brownian dynamics (BD), in reaction models), and often an RDME approximation is used to simulate it (see \cite{EO14} for references on multi-scale approximation simulation approaches). The BD framework can also be used to analyze the effects on the motion of molecules due to macromolecular crowding or confined geometries \cite{BC13}. See the surveys \cite{GS08, SG18} for a discussion of appropriateness and shortcomings of different reaction modelling frameworks, as well as for examples of specific intracellular mechanisms whose size and abundance scalings are varied and for which a careful choice of framework is essential. See also the survey \cite{ECM07} of stochastic algorithms for reaction network dynamics with a spatial component.

This individual based model of BD motion and reaction dynamics is naturally related to the deterministic models described by reaction-diffusion partial integro-differential and reaction-diffusion partial differential equations (in the same way that the non-spatial Markov chain model for reaction dynamics is related to the deterministic reaction dynamics described by ordinary differential equations). Some very recent works (\cite{IMS20a, IMS20b, LLN20}) have shown that for systems involving only bimolecular reactions one can obtain rigorous approximation limits connecting the BD and RD-PIDEs/RD-PDEs, under the assumptions of a single uniform scaling of all molecules and of all reaction rates. Although one can capture the behaviour of many interesting examples (see \cite{IMS20b} for one on the trade-off between the effects of diffusion speed and interaction range), these recent results are still limited by the number of molecules involved in each reaction (and by the restriction of conservation of mass in the entropy approach of \cite{LLN20}). In addition to allowing for more general reactions, our representation of the Brownian dynamics also allows one to establish rigorous approximation results  for varying scenarios of abundance and reaction rate scalings. The feature of flexibility to different scalings is particularly important in intracellular settings.

Current experimental evidence shows that stochasticity of reaction mechanisms in a cellular environment plays a key role in many intracellular systems  (see \cite{ELSS02, P04,RvO08} for gene expression experiments and a discussion of other examples). The modelling of chemical reaction networks for cellular processes requires different scalings than usual, both in terms of orders of magnitude for the abundances of different species, and in terms of orders of magnitude for reaction rates. For this reason, general multi-scale models have been developed for rigorous mathematical analysis of multi-scale stochastic reaction networks, firstly in the context of non-spatial and well-mixed (spatially homogeneous) systems (see \cite{BKPR06, KK13, KKP14, CW16, P19, Ro19} for functional law of large numbers, central limit theorem and large deviation results for such models, and \cite{BCM14, McSP14} for additional scenarios with time-scale separations). Many applications use a model reduction based on some form of these results. 

Spatial heterogeneity and movement of different species within it may also require different orders of magnitude for their speeds (for example small secondary messengers are faster than large complex proteins). The movement of some molecules can even be zero if a species is localized in a part of the space (for example, genes in nucleus, members of signaling pathways via anchor or scaffolding, see \cite {BSS04,LGDN09}).
Analytic consequences of these have been investigated within the context of compartment models (see \cite{PP15, PP15b} for  reduction limits in multi-scale compartment models), in particular reduction or increase in nonlinearity in responses can be the result of spatial heterogeneity. Many important features in applications, such as cellular adaptive ability and  signal  processing, rely on shaping nonlinearity in this way (see e.g. \cite{HFWS13} for the role that movement and localization can play in generating different cellular states).

Since taking into account stochasticity may lead to qualitatively different physiological predictions, our task is to provide a representation within which one can model movement and reactions at the individual molecule level in a stochastic and heterogeneous manner. The  main criterion for it is to enable macroscopic approximations of its behaviour under a diverse set of abundance and rate scalings. Using the most appropriate scalings of different molecular types, on a case by case basis, such limits will capture all of the essential reaction dynamics as well as the essential stochasticity, and provide a rigorous model reduction technique.

The applications of such models can extend beyond the molecular intracellular dynamics to processes at the organism and population levels. In recent years, tractable continuous space stochastic models of movement and interaction of different types of individuals in a heterogeneous environment have been developed in evolution and ecology (see \cite{BM15,BEV10, ChM07,Fo18,Le16} and references therein). In these models, the distribution of the population in type and space is represented by a random measure evolving by birth and death processes whose rates are linear or at most pairwise (due to competition for resources, or fecundity selection with two types of individuals). Our framework includes such models, and extends them to more general rates as we allow any finite number of source and product species. Technically, in our framework Gronwall-type arguments are no longer sufficient to control the total interaction rates, and obtain bounds on appropriate moments of the total or local mass at any given time. Consequently, we need to assume appropriate moment conditions are satisfied (and discuss how they can be proved by other kinds of methods, such as coupling or comparison arguments).

Given our representation of the individual-based interaction and movement framework, we first show the set of assumptions on its dynamics under which this measure-valued Markov process is well defined (Theorem~\ref{thm:existence}). We then turn to a multi-scale analysis of the process in the particular case where some species diffuse in space and others are localized (\emph{i.e.}, can only be found at one spatial location), assuming that there are two different orders of species abundances. Letting our scaling parameter $N$ tend to infinity, we derive the asymptotic limit in terms of the solution to a given martingale problem (Theorem~\ref{thm:largeNbis}). Under additional assumptions, we argue that this process is in fact a measure-valued piecewise deterministic process (PDMP), in which the unscaled mass (\emph{i.e.}, counts) of species in lower abundances follows a Markov jump process, while the scaled mass (\emph{i.e.}, concentration) of species in higher abundances in between these jumps evolves deterministically according to a flow that can be described as the weak solution to a system of partial integro-differential equations  (Proposition~\ref{prop: PDMP}). In the special case where there are no low abundance species, we recover the expected law of large number result in which the limiting process is deterministic and only given in terms of a system of nonlinear partial integro-differential equations (similar to reaction-diffusion equation results in \cite{IMS20a, LLN20}, see Section~\ref{section:RD scaling}). In this case, we also focus on the qualitative effects of the localization of some molecular species, and investigate the assumptions needed to ensure regularity of limits for the subset of moving species (Proposition~\ref{thm:densityN}). 

The original class of PDMPs (as defined by \cite{Davies}) encompasses strong Markov processes with two types of coordinates, taking values in a subset of $\mathbb N^d\times \mathbb R^c$, in which the deterministic dynamics of the continuous coordinates and the stochastic Markov dynamics of the discrete coordinates of the processes are fully coupled. The stochastic dynamics is prescribed by jump rates and the deterministic dynamics is prescribed by a continuous flow. The appearance of infinite-dimensional PDMPs in the literature has been recent, and mostly motivated by biological applications (predominantly neuroscience models), with the continuous component taking values in a separable Hilbert space (see \cite{A08, BR11,CdSJ20,GT12,RTW12} for a possibly complete set of results to date). Since our approach derives the infinite-dimensional PDMP as a limit of a sequence of measure-valued Markov processes, some useful properties of the process are inherited from the pre-limiting sequence, and could potentially be used in devising simulation algorithms.

The rest of the paper is organised as follows. We introduce our model for chemical reaction networks in heterogeneous space in Section~\ref{ss:model}. We provide an algorithmic construction of the corresponding stochastic measure-valued process which we show is well-defined in Theorem~\ref{thm:existence} of Section~\ref{sec:construction}, and prove in Appendix~\ref{s: proof Th1}. We then turn to a multi-scale analysis of this process in the case when some molecular species diffuse in space while others are localized, and among the set of localized species, some remain in low abundances (as our scaling parameter $N$ tends to infinity) while the others have abundances of order $\mathcal{O}(N)$. The multi-scale limit is  Theorem~\ref{thm:largeNbis} of Section~\ref{section:magnitudes} which we prove in  Section~\ref{section:proofTH}. We complement our convergence results by showing regularity properties of the limit in Proposition~\ref{thm:densityN} of Section~\ref{section:RD scaling}, when there are no low abundance species (and therefore the limit is deterministic and characterized by a set of partial integro-differential equations). For convenience, all the notation for the different types of reactions and of molecular species are summarised in Table~\ref{table:notation} in Section~\ref{section:magnitudes}.

\section[A spatial measure-valued reaction process]{A Spatial Measure-Valued Reaction Process -- Definition and Multi-scale Analysis}\label{section:model}
\subsection{The model}\label{ss:model}
Let $\cT$ be a finite set of molecular species, and let $E$ be a compact subset of $\bR^d$ of possible spatial locations common for all species, with a smooth boundary and a non-empty interior (the specific assumptions on $E$ will depend on the type of motion that will model the movement of molecules). We define the underlying space for the process to be
\[\cP=\cT \times E, \]
each molecule corresponding to a point $p=(x,y)\in\cP$. The measure-valued process describing the composition of molecular species in space is given, at any time $t\geq 0$, by the counting measure (with a finite indexing set ${\cal I}_t$)
\begin{equation}\label{def Mt}
M_t=\sum_{i\in {\cal I}_t} \delta_{(x_i,y_i)}
\end{equation}
in the space $\cM_p$ of all finite point measures on $\cP$. For any $M\in \cM_p$ and any function $f:\cP\rightarrow \bR$, we set $\langle M,f\rangle := \sum_{i} f(x_i,y_i)$. Also, for a counting measure $M$ and $n\in \bN$, we write 
\begin{equation}\label{def sampling}
M^{\otimes \downarrow  n}(dp_1,\ldots,dp_n) := M (dp_1)(M - \delta_{p_1})(dp_2)\cdots (M-\delta_{p_1}-\cdots - \delta_{p_{n-1}})(dp_n)
\end{equation}
for the measure describing the sampling of $n$ points without replacement according to $M$. By convention, $M^{\otimes \downarrow n}$ is the null measure when $n=0$ or $n>\langle M,1\rangle$.

Let the reaction network consist of a finite set of reactions $R$, where (for convenience of notation in our model) a reaction $r\in R$
is of the form
\[A^r_{1}+\cdots+ A^r_{k_r} \mapsto B^r_{1}+\cdots+ B^r_{k'_r}\]
with source reactants of types $A^r_i$ and product reactants of type $B^r_i$, both of which are allowed to repeat. The usual notation for reaction networks by stoichiometric vectors and a stoichiometric  matrix can be determined as
\begin{equation}\label{stoichiometry}
\nu_{r,x}=\sum_{j=1}^{k_r}\1_{x}(A^r_j), \quad \nu_{r,x}'=\sum_{j=1}^{k_r'}\1_{x}(B^r_j), \quad S=[\nu_{r,x}'-\nu_{r,x}].
\end{equation}

In our model, we shall consider two kinds of reactions. In a \emph{non-localized} reaction $r$, close-by molecules react and are consumed or created continuously in space (with a rate dependent on the region in space where the reaction takes place, this rate being potentially equal to zero in some parts of $E$ -- see below). In contrast, a \emph{localized} reaction takes place at a given point in space, usually where some of the species involved are attached (\emph{e.g.}, the nucleus membrane, or the extra-cellular membrane of a cell). We shall denote the set of non-localized reactions by $R_{NL}$ and the set of localized reactions by $R_{L}$, so that $R=R_{NL}\cup R_{L}$ (a disjoint union). For simplicity we shall not formally consider reactions happening continuously in some part of $E$ and in a localized way in other parts of $E$, since they may be encoded as the sum of several purely continuous or purely localized reactions in our framework (note however that this generalisation would simply consist in taking a measure $\varrho_r(d\bar y)$ of a mixed form in \eqref{def varrho}).

\medskip
\noindent{\bf Non-localized reactions.} The dynamics of a non-localized reaction $r\in R_{NL}$ at any location $\bar y\in E$ is specified by a spatially-dependent chemical reaction factor ${\bar h}_r(\bar y, M)$ and a mass-action-kinetics reaction function that is based on availability of source reactants in a neighbourhood of $\bar y$
as determined by a proximity (probability) kernel $\Gamma_\ep$ centred at $\bar y$ and with support in the ball $B(0,\ep)\subset \bR^d$ for some $\ep>0$ (for simplicity, we take the same kernel $\Gamma_\ep$ for all reactions, but this can be easily generalised to reaction-dependent kernels). Specifically, given a set of $k_r$ source reactants $(p_i=(x_i,y_i))_{i=1,\dots,k_r}$ of the appropriate types, the rate of the reaction $r$ which produces molecules of type $B_1^r,\ldots, B^r_{k_r'}$ at location $\bar y\in E$ is
\begin{equation} \label{eq:lambda}
\lambda_r(\bar y,M;p_1,\ldots,p_{k_r}):={\bar h}_r(\bar y, M)\, \Bigg(\prod_{i=1}^{k_r}\1_{A_i^r}(x_i)\Gamma_\ep(y_i-\bar y)\Bigg).
\end{equation}
The overall rate of reaction~$r$ at location $\bar y\in E$ is then obtained by sampling the $k_r$ source reactants, without replacement, from the current state $M_t$ of the measure-valued process:
\begin{align}
&\int_{\cP^{k_r}}M_t^{\otimes \downarrow k_r}(dp_1,\ldots, dp_{k_r})\, \lambda_r(\bar y,M_t;p_1,\ldots,p_{k_r}) \nonumber\\
& \qquad = \, {\bar h}_r(\bar y, M_t)\int_{\cP^{k_r}} M_t^{\otimes \downarrow k_r}(dp_1, \ldots, dp_{k_r})\Bigg(\prod_{i=1}^{k_r}\big(\1_{A^r_i}(x_i)
\Gamma_\ep(y_i-\bar y)\big)\Bigg).\label{local rate}
\end{align}
The dependence of the chemical reaction factor ${\bar h}_r(\bar y, M_t)$ on the current state of the global species composition is included in order to allow the presence of chemical reaction rates that are not simply of ``mass-action" form (see \cite{CW16} for examples of where elimination of extremely fast intermediate subnetworks appear in the reaction factors of the reduced network reactions). Such reactions allow the mass of some species, which is unchanged by that particular reaction, to affect the reaction rate indirectly through a role of a promoter or inhibitor.  This dependence on $M$ is of the form
\begin{equation}\label{form of rate}
{\bar h}_r(\bar y, M)=h_r(\bar y, \langle M,\Psi_{r,\bar y}\rangle),
\end{equation}
for some nonnegative functions $h_r$, $\Psi_{r,\bar y}$. For instance, we may want to take $\Psi_{r,\bar y}(x,y)= \1_{B(\bar y,\ep)}(y)$ when the reaction rate is affected only by nearby mass. The assumptions we make on $h_r$ and $\Psi_{r,\bar y}$ are detailed in Assumption~(A1) below (in particular, for technical reasons we shall need to replace the indicator function $\1_{B(\bar y,\ep)}(y)$ by a continuous approximation to it in the above example).

\medskip
\noindent{\bf Localized reactions.} Suppose that reaction $r\in R_{L}$ occurs at a single location ${\bar y}_r\in E$. We use the same kernel $\Gamma_\ep$ to check the availability of source reactants in the neighbourhood of ${\bar y}_r$. Hence, equation~\eqref{eq:lambda} with ${\bar y}={\bar y}_r$ still describes the rate at which a given set of $k_r$ molecules $(p_i=(x_i,y_i))_{i=1,\dots,k_r}$ of the appropriate types react at ${\bar y}_r$ and equation~\eqref{local rate} now describes the total rate at which reaction $r$ occurs (in contrast with non-localized reactions, whose local rates have to be integrated over $\bar y \in E$ to obtain their global reaction rates).

\begin{remark} Taking $\Gamma_\ep$ and $h_r$ to be constant over the range of possible values for their arguments, we recover the mass-action form for reaction rates. Molecules still have spatial locations but the latter play no role in the dynamics of the reactions.  We can also recover compartment dynamics by making all reactions localized at centres of the compartments and by taking $\Gamma_\ep$ and $h_r$ to be constant (with $\ep=$ half the compartment size).
\end{remark}

To unify the notation, let us write $\ell_E$ for Lebesgue measure on $E$ and let us define the measures
\begin{equation}\label{def varrho}
\varrho_r(d\bar y)= \ell_E(d\bar{y}) \quad \hbox{if }r\in R_{NL} \qquad \hbox{and} \qquad \varrho_r(d\bar y)= \delta_{{\bar y}_r}(d\bar{y}) \quad \hbox{if }r\in R_L.
\end{equation}
Hence, in both the localized and non-localized case, the total rate at which reaction $r$ occurs when the current state of the system is described by the counting measure $M$ is given by
\begin{equation}\label{rate Lambda}
\Lambda_r(M):= \int_{E\times \cP^{k_r}}\varrho_r(d\bar y)M^{\otimes \downarrow k_r}(dp_1,\ldots,dp_{k_r})\,\lambda_r(\bar y,M;p_1,\ldots,p_{k_r}).
\end{equation}

\medskip

We shall rigorously construct the measure-valued process $(M_t)_{t\geq 0}$ in an algorithmic way in Section~\ref{sec:construction}. For now we just keep introducing the key ingredients to describe its dynamics.
The weak topology on the space $\cM$ of all finite measures on the compact space $\cP$ is determined by $\langle M,f\rangle:=\int f dM$ over a sufficiently large class of functions $f$ in $\cC(\cP)$, the space of all continuous (hence bounded) functions on $\cP$. Note that for a point measure $M\in \cM_p$, we recover $\langle M,f\rangle= \sum_i f(x_i,y_i)$. 

Most of our analysis of the process $(M_t)_{t\geq 0}$ will rely on the martingale problem it satisfies. Hence, let us introduce the different objects we shall need to formulate this martingale problem. First, the operator describing the change due to reactions in $R$, acting on test functions of the form
\begin{equation}\label{eq:Ff}
F_f=F(\langle \cdot,f\rangle),
\end{equation}
with $F\in \cC_b(\bR)$ (\emph{i.e.}, continuous and bounded on $\bR$) and $f\in \cC(\cP)$, will be given by $\sum_{r\in R} G_r$, where for all $r\in R$ we have
\begin{align}
G_rF_f(M)=& \int_{E\times \cP^{k_r}} \varrho_r(d\bar y) M^{\otimes \downarrow k_r}(dp_{1},\ldots,dp_{k_r}) \lambda_r(\bar y,M;p_1,\ldots,p_{k_r}) \label{eq:Gr1}\\
 & \qquad  \qquad \bigg[F\bigg(\langle M,f\rangle-\sum_{i=1}^{k_r}f(x_i,y_i)+\sum_{i=1}^{k_r'}f(B^r_i,\bar y)\bigg)-F_f(M)\bigg]. \nonumber
\end{align}

Recalling the definition of $\lambda_r$ given in \eqref{eq:lambda}, we see that a reaction $r$ at location $\bar y$ occurs if the needed types $A^r_i$ of reactants sampled from the measure $M$ exist in sufficient numbers in the neighbourhood of $\bar y$ described by the kernel $\Gamma_\ep$. When it occurs, it removes the source reactants from $M$ and produces molecules of types $B^r_i$, all at location $\bar y$. Note that in case $k_r=0$ (creation of product molecules from an external source), the rate is determined by $\bar h_r$ (which may still be a function of the measure $M$).

\begin{example}\label{example1} Suppose the network consists of one $(r=1)$ localized and three $(r=2,3,4)$ non-localized reactions on two molecular species $(S,S')$:
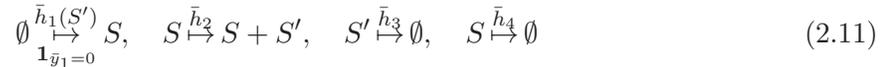
\begin{equation}\label{eq:exampleloc}
 \emptyset\mathop{\mapsto}\limits^{{\bar h}_1(S')} _{\1_{\bar y_1=0} }S,\quad S\mathop{\mapsto}\limits^{{\bar h}_2} S+S', \quad S'\mathop{\mapsto}\limits^{{\bar h}_3}\emptyset,\quad S\mathop{\mapsto}\limits^{{\bar h}_4} \emptyset
\end{equation}
This is a simplified version of the transcription-translation mechanism of a protein: here $S$ is the mRNA and $S'$ is the protein, the creation of mRNA occurs only in the nucleus at $\bar y_1=0\in E$ and is given by the transcription rate ${\bar h}_1(\bar y_1,M)=h_1(\bar y_1,\langle M,\Psi_{S',\ep}\rangle)$, where $\Psi_{S',\ep}$ is a continuous approximation to $\1_{\{S'\}\times B(0,\ep)}$. In the unregulated case the function $\bar h_1=h_1(0)$ is constant, while in the self-regulated case $\bar h_1(\bar{y}_1,M)=h_1(0,a)$ is a function of the mass $a=\langle M,\Psi_{S',\ep}\rangle$ of produced protein that diffuses back to the neighbourhood of the nucleus. To ensure $\bar h_1$ satisfies our Assumption (A1) (see below) we ask that $h_1$ should be Lipschitz in $a$ and uniformly bounded on \mbox{compact sets for $a$}.  For example, we can take $h_1(\bar y_1, a)=c_1/(1+(c_2a)^k)$ if the mechanism is repressed by $S'$, or $h_1(\bar y_1, a)=(1+c_1a^k)/(c_2^k+a^k)$ if the mechanism is activated by $S'$ (see \cite{
Mackey16}), for some $k\ge 1$ (also referred to as Hill function coefficient). For reactions $r=2,3,4$, the reaction rate factor $\bar h_r$ is taken to depend on the spatial coordinate but not on the mass coordinate and we ask that each should be uniformly bounded over $\bar y\in E$.
The operators encoding these reactions are
\begin{align}
G_1F_f(M)&=  h_1(\bar y_1, \langle M, \Psi_{S',\ep}\rangle)\Big[F\big(\langle M,f\rangle+f(S,\bar y_1)\big)-F_f(M)\Big]\label{eq:Gexample}\\
G_2F_f(M)&= \int_{E\times \cP} \ell_E(d\bar y) M(dp) \, \1_{S}(x) \Gamma_\ep(y-\bar y) h_2(\bar y)
\Big[F\big(\langle M,f\rangle+f(S',\bar y)\big)-F_f(M)\Big]\nonumber\\
G_3F_f(M)&= \int_{E\times \cP} \ell_E(d\bar y) M(dp) \,\1_{S'}(x) \Gamma_\ep(y-\bar y) h_3(\bar y) \Big[F\big(\langle M,f\rangle-f(S',y)\big)-F_f(M)\Big] \nonumber \\
G_4F_f(M)&= \int_{E \times \cP} \ell_E(d\bar y)M(dp) \, \1_{S}(x) \Gamma_\ep(y-\bar y) h_4(\bar y) \Big[F\big(\langle M,f\rangle-f(S,y)\big)-F_f(M)\Big]. \nonumber
 \end{align}
 Note that in the second reaction, for simplicity we chose to consider that the species $S$ involved was not at all modified by the reaction. Another option, in line with the description of the model given above, would have been to consider that the source reactant $(S,y)$ appearing in the expression for $G_2F_f(M)$ should be withdrawn and replaced by some new molecule $(S,\bar y)$ at location $\bar y$. This choice of formulation is left to the modeller. This simple example will be used later to illustrate the effects of reaction localization, and multi-scaling of abundance of molecular types, see Examples~\ref{ex:PDElimit} and \ref{ex:PDMPlimit} in Section~\ref{section:magnitudes}.
 \end{example}

The change due to the movement of molecules in the spatial domain $E$ will be described by an operator $\cD$ that we can take to be fairly general, provided that the martingale problem associated with $\cD$ defined on a large enough class of functions is well-posed and that it is bounded by some power of the total mass function (see Assumption (A0) and Remark~\ref{rmk: moment bound} below). In particular, we may be interested in situations where the local concentration in molecules of some species influences the propensity of other species to visit or avoid the corresponding region of space. A much simpler example, on which we shall concentrate in the multi-scale analysis expounded in Section~\ref{section:magnitudes}, is to suppose that molecules of type $x$ move in $E$ independently of each other and of molecules of the other types, following a diffusion with locally bounded Lipschitz drift coefficient $b_x :E\rightarrow \bR^d$ and locally bounded Lipschitz dispersion matrix $\Sigma_x : E \rightarrow \bR^{d\times d}$ such that the diffusion matrix $\Sigma_x^2:= \Sigma_x(\Sigma_x)^{\bf t}$ is uniformly elliptic, this diffusion being normally reflected at the boundary of $E$. When the interior $\mathring{E}$ of $E$ is non-empty, bounded and either convex, or smooth ($\mathcal C^3$), both with a piecewise-smooth boundary and only a finite number of convex corners, this reflected movement is well posed (\cite{LionsSznitman84}, \cite{Tanaka79} give solutions to stochastic differential equations with reflection; \cite{KangRamanan17} equates them to solutions of submartingle problems). In this particular case, for every $x\in \cT$ and every sufficiently regular function $f$, we define the functions $b_x \cdot \nabla_yf$, $\Sigma_x^2 \circ \Delta_y f$, and $\Sigma_x^2 \circ \big( (\nabla_yf) (\nabla_y f)^{\bf t}\big)$ by
\begin{align}
b_x \cdot \nabla_yf & : (x',y')\mapsto \1_{x}(x')\sum_{i=1}^d b_x(y')_i\frac{\partial f}{y_i}(x',y') \label{derivatives} \\
\Sigma_x^2 \circ \Delta_y f &: (x',y') \mapsto \frac{\1_x(x')}{2} \sum_{i=1}^d \sum_{j=1}^d \Sigma_x^2(y')_{ij}\frac{\partial^2f}{\partial y_i\partial y_j}(x',y')  \nonumber \\
\Sigma_x^2 \circ \big( (\nabla_yf) (\nabla_y f)^{\bf t}\big) &: (x',y')\mapsto \frac{\1_x(x')}{2}\sum_{i=1}^d\sum_{j=1}^d \Sigma_x^2(y')_{ij} \frac{\partial f}{\partial y_i}(x',y')\frac{\partial f}{\partial y_j}(x',y'), \nonumber
\end{align}
where $z_i$ denotes the $i$-th coordinate of the vector $z$ and $Z_{ij}$ denotes the $(i,j)$-coordinate of the matrix $Z$ (and the notation $\circ$ is inspired by the Hadamard product of matrices). Still in this particular example, we consider test functions of the form $F_f=F(\langle \cdot,f\rangle)$ with
\begin{itemize}
\item[$(a)$] $F\in \cC_b^2(\bR)$, \emph{i.e.}, bounded and of classe $\cC^2$ on $\bR$, and
\item[$(b)$] $f\in \cC^{0,2}(\cT\times E)$ (\emph{i.e.}, measurable in the first coordinate, and of class $\cC^2$ in the second coordinate) satisfying $\nabla_yf(x',y')\cdot n(y')=0$ for all $(x',y')\in \cT\times \partial E$, where $n(y')$ denotes the outward normal to the boundary of $E$ at $y'\in \partial E$ and $ \cdot$ denotes scalar product in $\bR^d$,
\end{itemize}
and the operator $\cD$ applied to such a test function can be written (see Theorem~3.1 in \cite{RR90}) 
\begin{align}
\cD F_f(M)=& \ F'(\langle M,f\rangle) \sum_{x\in \cT}\langle M, b_x \cdot \nabla_y f+\Sigma_x^2\circ \Delta_y f\rangle   \label{eq:cD} \\
& + \ F''(\langle M,f\rangle)\sum_{x\in \cT} \big\langle M,\Sigma_x^2 \circ \big( (\nabla_yf) (\nabla_y f)^{\bf t}\big) \big\rangle.\nonumber
\end{align}

Coming back to the general case for the movement of species and summing up the above, the operator that will serve as a basis for the martingale problem describing the overall dynamics of the measure-valued Markov process $(M_t)_{t\ge 0}$ is the following:
\begin{equation}\label{eq:L}
LF_f(M)= \sum_{r\in R}G_rF_f(M)+ \cD F_f(M).
\end{equation}
For ease of reference, let us give a name to the martingale problem associated with $L$.
\begin{definition}\label{def: MP}
Let $\mathbf{F}$ be the set of test functions defined in Assumption~(A0) below. We say that an $\cM_p$-valued process $(M_t)_{t\geq 0}$ satisfies the martingale problem $\mathrm{MP}(L)$ if for every function $F_f\in \mathbf{F}$, the process
$$
\bigg(F_f(M_t)-F_f(M_0)- \int_0^t ds\, LF_f(M_s)\bigg)_{t\geq 0}
$$
is a martingale (for the natural filtration associated with $(M_t)_{t\geq 0}$).
\end{definition}

We make the following \underline{\bf Assumptions} on the operators $\cD$ and $G_r$. By convention, for every localized reaction $r\in R_{L}$ we set $h_r(y,\cdot)\equiv 0$ for all $y\neq {\bar y}_r$. We write $D_{\cM_p}[0,\infty)$ for the space of all c\`adl\`ag paths with values in $\cM_p$.
\begin{enumerate}
 \item[(A0)] There exists a set $\mathbf{F}$ of functions of the form~(\ref{eq:Ff}), dense in $\cC(\cM)$ for the topology of uniform convergence over compact sets, satisfying 
 \begin{itemize}
 \item[$(i)$] For every $F_f\in \mathbf{F}$, there exists a constant $c_{F,f}>0$ such that
     $$
     |\cD F_f(M)|\leq c_{F,f}\langle M,1\rangle \qquad \hbox{for every }M\in \cM_p.
     $$ 
\item[$(ii)$] The martingale problem associated to $\cD$ (with domain $\mathbf{F}$) has a unique solution in $D_{\cM_p}[0,\infty)$ for any initial distribution belonging to the set of probability measures on $\cM_p$. Furthermore, this solution has the Markov property and satisfies that the total number of atoms and their first coordinates $x_i\in \cT$ are left unchanged by the dynamics (in other words, only the spatial locations in $E$ of the atoms evolve in time).
\end{itemize}
 \item[(A1)] For each $r\in R$, the reaction factor $h_r:E\times \bR_+\rightarrow \bR_+$ is uniformly bounded over compact subsets of $E\times\bR_+$: for every $\ell\geq 0$,
 \[ \sup_{\bar y \in E}\, \sup_{a\in [0,\ell]}\ h_r(\bar y,a)=\|h_r\|_{\infty,\ell}<\infty.\]
It is also Lipschitz in the second coordinate, with Lipschitz constant $L_r$ independent of the first coordinate $\bar y$.
Finally, for every $\bar y\in E$ the function $\Psi_{r,\bar y}:\cP\rightarrow \bR_+$ is continuous and
$$
\sup_{\bar y\in E}\sup_{p\in \cP}\Psi_{r,\bar y}(p)=\|\Psi_r\|_\infty<\infty.
$$
 \item[(A2)] For some fixed $\ep>0$, the function $\Gamma_\ep\ge 0$ is a continuous probability density with support contained in the closed ball $B(0,\ep)\subset \bR^d$: in particular,
 \[\sup_{y\in \bR^d} \Gamma_\ep(y)=\|\Gamma_\ep\|_\infty<\infty, \quad \int_{\bR^d} \Gamma_\ep(y)dy=1.\]
\end{enumerate}

\begin{remark}\label{rmk:F sufficient}
Assumption~(A0) is satisfied in our previous example of independent inhomogeneous diffusions if we restrict our attention to $F\in \cC_b^2(\bR)$ with bounded first and second derivatives, since the set of all $f\in \cC^{0,2}(\cT\times E)$ with vanishing normal derivative on $\partial E$ is dense in $\cC(\cP)$ for the supremum norm -- see Remark~1.1 in \cite{ChM07}. 
\end{remark}

\begin{remark}\label{rmk: moment bound}
$(a)$ We may generalise the bound stated in Assumption~(A0)-$(i)$ into the existence of $K\in \bN$ such that
\begin{equation}
|\cD F_f(M)|\leq c_{F,f}\langle M,1\rangle^K \qquad \hbox{for every }M\in \cM_p,
\end{equation}
for instance if we wanted to include some density-dependence in the movement of species. For our existence result, Theorem~\ref{thm:existence}, to hold true, we would then have to replace Assumption~(A3) stated in the theorem by the stronger condition that we can control the supremum over any finite time interval of the $(K \vee (1+\max_r k_r))$-th moment of the total mass of the process. See Remark~\ref{rmk:justif K} at the end of Appendix~\ref{s: proof Th1}.

\noindent $(b)$ We may also relax the assumption that the solution to the martingale problem associated to $\cD$ should have c\`adl\`ag paths (at the expense of $(M_t)_{t\geq 0}$ itself not having c\`adl\`ag paths), but since it is a natural assumption in view of the applications and since our multi-scale analysis of the particular case of diffusing molecules in Section~\ref{section:magnitudes} will rely on it, we keep the simpler framework of continuously moving particles with mass evolution as c\`adl\`ag processes. 
\end{remark}

\subsection{Construction of the process $(M_t)_{t\geq 0}$}\label{sec:construction}
In this section, we suppose that Assumptions~(A0), (A1) and (A2) hold true and we construct a process that satisfies the desired dynamics. It is this particular process (appropriately rescaled) that we shall use later in our multi-scale analysis. The construction relies on the fact that the rate at which each reaction occurs is bounded from above by a polynomial in the total mass of the system, whose supremum over any fixed time horizon is a.s. finite by our additional Assumption~(A3) below. Between the occurrence times of two consecutive reactions, the finitely many particles in the system move around according to the dynamics described by $\cD$. 

More formally, let $M_0$ be a random finite counting measure, with law $\mathcal{L}(M_0)$. All the random objects used in this section are supposed to be defined on a common probability space $(\Omega,\mathcal{F},\bP)$.
\begin{itemize}
\item Write $M_0=\sum_{i\in I_0}\delta_{(x_i,y_i)}$. Set $\tau^0=0$.
\item Let $\widetilde{M}^0_t=\sum_{i\in I_0}\delta_{(x_i,Y^0_i(t))}$ denote the value at time $t$ of the (by assumption, unique) $\cM_p$-valued solution to the martingale problem associated to $(\cD,\mathcal{L}(M_0))$ (since the operator $\cD$ only makes particles move in $E$, only the second coordinate $Y^0_i(t)$ evolves for every $i$). 
\item For every reaction $r\in R$, define ${\cal E}_r^1$ as an exponential random variable with parameter~$1$, independent of all other variables, and the random time $\tau^1_r$ as (recall the definition of $\Lambda_r(M)$ given in \eqref{rate Lambda}):
\begin{align*}
\tau^1_r:= \inf\bigg\{s> 0:\, \int_0^s dt\, \Lambda_r\big(\widetilde{M}_t^0\big) \geq {\cal E}_r^1\bigg\}.
\end{align*}
Using the bound on $\lambda_r$ that we shall establish in Lemma~\ref{lem:bound lambda} together with the fact that the total mass of $\widetilde{M}^0$ is constant equal to $|I_0|$, we obtain that for each $r$, $\Lambda_r(\widetilde{M}_t^0)$ is bounded independently of $t$ and so all $\tau^1_r$ are positive a.s. We can therefore set $\tau^1:=\min_{r\in R} \tau^1_r$ and let $r^1$ be the index of the unique reaction satisfying $\tau^1_{r^1}:=\min_{r\in R} \tau^1_r$. In words, $\tau^1$ is the random time at which the first reaction occurs when we let the $|I_0|$ particles move in space and interact, and $r^1$ is the index of the reaction that takes place at time $\tau^1$. The outcome of this reaction is given by the following procedure: Sample $(\bar y^1,p^1_1,\ldots,p^1_{k_{r^1}})\in E\times \cP^{k_{r^1}}$ according to the probability measure
$$
\frac{ \lambda_{r^1}(\bar y,\widetilde{M}^0_{(\tau^1)-};p_1,\ldots,p_{k_{r^1}})}{\Lambda_{r^1}(\widetilde{M}^0_{(\tau^1)-})} \, \varrho_r(d{\bar y}) \big(\widetilde{M}^0_{(\tau^1)-}\big)^{\otimes \downarrow k_{r^1}}(dp_1,\ldots,dp_{k_{r^1}}).
$$
(Observe that the denominator is necessarily nonzero, otherwise the probability that the reaction occurring at time $\tau^1$ is the one labelled by $r^1$ would be $0$.) Define the new value of the measure describing the system just after the reaction by
$$
\widetilde{M}^1_{\tau^1}:= \sum_{i\in I_0\setminus \bar{I}^1_s} \delta_{(x_i,Y^0_i(\tau^1))}
+\sum_{i=1}^{k_{r^1}'}\delta_{(B_i^{r^1},{\bar y}^1)},
$$
where $\bar{I}^1_s$ is the index set of the particles chosen in the previous step. That is, we remove the $k_{r^1}$ source reactants, and add the $k'_{r^1}$ product reactants all at location ${\bar y}^1$. 
\item Write $I_1$ for the index set of $\widetilde{M}^1_{\tau^1}$, and (abusing notation\footnote{More rigorously, at time $\tau^1$ we paste the $\cM_p$-valued solution to the martingale problem associated to $(\cD, \mathcal{L}(\widetilde{M}^1_{\tau^1}))$ to the trajectory of $\widetilde{M}^0$ stopped at $\tau^1$.}) let the collection of particle locations $((Y^1_i(t),\, i\in I_1))_{t\geq \tau^1}$ evolve according to the random motion in $E^{I_1}$ generated by $\cD$ and started at time $\tau^1$ from the current locations $(y_i)_{i\in I_1}$ of the particles. For every time $t\geq \tau^1$, define $\widetilde{M}^1_t= \sum_{i\in I_1}\delta_{(x_i,Y^1_i(t))}$.
\end{itemize}
For every $j\geq 2$, proceed recursively following the same steps as above:
\begin{itemize}
\item For every $r\in R$, let $\mathcal{E}_r^j$ be an independent exponential r.v. with parameter~$1$. Define
\begin{equation}
\tau^j_r:= \inf\bigg\{s > \tau^{j-1}:\, \int_{\tau^{j-1}}^s dt\, \Lambda_r\big(\widetilde{M}_t^{j-1}\big) \geq {\cal E}_r^j\bigg\}.
\end{equation}
Let then $\tau^j:= \min_{r\in R}\tau^j_r$ and let us denote the index of the unique reaction that realizes the minimum at time $\tau^j_r$ by $r^j$. Sample $({\bar y}^j,p_1^j,\ldots,p_{k_{r^j}}^j)\in E\times \cP^{k_{r^j}}$ according to the probability measure
$$
\frac{ \lambda_{r^j}(\bar y,\widetilde{M}^{j-1}_{(\tau^j)-};p_1,\ldots,p_{k_{r^j}})}{\Lambda_{r^j}(\widetilde{M}^{j-1}_{(\tau^j)-})} \, \varrho_r(d{\bar y}) \big(\widetilde{M}^{j-1}_{(\tau^j)-}\big)^{\otimes \downarrow k_{r^j}}(dp_1,\ldots,dp_{k_{r^j}}).
$$
The vector $({\bar y}^j,p_1^j,\ldots,p_{k_{r^j}}^j)$ indicates the location of the $j$-th reaction and the $k_{r^j}$ particles chosen to react (which will then be removed). Next, set
$$
\widetilde{M}^j_{\tau^j}:= \sum_{i\in I_{j-1}\setminus \bar{I}^j_s} \delta_{(x_i,Y^{j-1}_i(\tau^j))}
+\sum_{i=1}^{k_{r^j}'}\delta_{(B_i^{r^j},{\bar y}^j)},
$$
where $\bar{I}^j_s$ is the index set of the particles chosen to react during step~$j$.
\item Write $I_j$ for the index set of $\widetilde{M}^j_{\tau^j}$ and, with the same abuse of notation as earlier, let $((Y^j_i(t),\,i\in I_j))_{t\geq \tau^j}$ evolve according to the random motion in $E^{I_j}$ generated by $\cD$, started at time $\tau^j$ from the collection $(y_i)_{i\in I_j}$ of particle locations at time $\tau^j$. For every $t\geq \tau^j$, define $\widetilde{M}^j_t= \sum_{i\in I_j}\delta_{(x_i,Y^j_i(t))}$.
\end{itemize}
Finally, let $\tau^\infty=\sup_{j\in \bN} \tau^j$ and define the process $(M_t)_{0\leq t<\tau^\infty}$ by
\begin{equation}\label{constructed M}
\forall t\in [0,\tau^\infty),\quad M_t=\widetilde{M}^{j(t)}_t,\quad \hbox{with }j(t) \hbox{ such that } t\in \big[\tau^{j(t)},\tau^{j(t)+1}\big).
\end{equation}

The main result of this section is the following theorem. 
\begin{theorem}\label{thm:existence}
Suppose that Assumptions (A0), (A1) and (A2) are satisfied. Suppose also that
\begin{itemize}
\item[(A3)] For every $T>0$,
$$
\sup_{t\in [0,T]}\langle M_t,1\rangle <\infty \quad \hbox{a.s., and }\sup_{t\in [0,T]}\bE\Big[\langle M_t,1\rangle^{1+\max_r k_r}\Big]<\infty,
$$
where the max in the exponent is taken over all $r\in R$. 
\end{itemize}
Then $\tau^\infty = +\infty$ a.s. and the process $(M_t)_{t\geq 0}$ is a c\`adl\`ag $\cM_p$-valued Markov process solution to $\mathrm{MP}(L)$.
\end{theorem}
The first part of Assumption~(A3) is needed to control the total reaction rate over any finite time interval, while the second part provides the integrability property required to prove that $(M_t)_{t\geq 0}$ satisfies the martingale problem MP$(L)$.

\begin{remark}\label{rmk:case by case}
Assumption~(A3) has to be checked case by case, as it may hold true for many different reasons. In Example~\ref{example1} above, assuming that $\bar{h}_1$ and $\bar{h}_2$ are uniformly bounded in both coordinates (\emph{e.g.}, the first reaction saturates when the concentration in species $S'$ is high) and that Assumption~(A2) is satisfied too, then the global creation rate of molecules of type $S$ is bounded by a constant and the rate of creation of molecules of type $S'$ is at most linear in the current number of molecules $S$. Hence, the total number of particles in the system is stochastically bounded by a binary branching process (with branching rate $\|\bar{h}_2\|_\infty$) with immigration at constant rate $\|\bar{h}_1\|_\infty$, for which it is straightforward to check that the two conditions stated in Assumption~(A3) are satisfied.

More generally, these two properties may be proved by stochastically bounding the total mass process by an appropriate birth and death process. They will also be satisfied whenever the reactions involving more than one source reactants do not make the number of particles increase (that is, $k_r'\leq k_r$) and the reaction factors $\bar{h}_r$ are all uniformly bounded in both coordinates.
\end{remark}

The proof of Theorem~\ref{thm:existence} is classical and is therefore deferred until Appendix~\ref{s: proof Th1}.

\subsection{Multi-Scale Reaction Networks and Convergence to Measure-Valued PDMP}\label{section:magnitudes}
In this section, we suppose that Assumptions~(A1), (A2) and (A3) are satisfied. We also suppose that the set $\cT$ of molecular species can be partitioned into a set $\cT_{L}$ of \emph{localized} species, for which we assume that all molecules of type $x\in \cT_{L}$ sit at some fixed point $\bar{y}_x\in E$, and a set $\cT_{NL}$ of \emph{diffusive} species. Molecules of type $x\in \cT_{NL}$ move in $E$ independently of each other and of molecules of the other types, following a diffusion with locally bounded Lipschitz drift coefficient $b_x :E\rightarrow \bR^d$ and locally bounded Lipschitz dispersion matrix $\Sigma_x : E \rightarrow \bR^{d\times d}$ such that the diffusion matrix $\Sigma_x^2:= \Sigma_x(\Sigma_x)^{\bf t}$ is uniformly elliptic, this diffusion being normally reflected at the boundary of $E$. Again, we assume that the interior of the compact set $E$ is nonempty and it has a piecewise-smooth $\cC^3$ boundary with only a finite number of convex corners. Recall the notation introduced in \eqref{derivatives}, the set of test functions introduced just below \eqref{derivatives} (that we take to be the set $\mathbf{F}$ defined in Assumption~(A0)) and the operator $\cD$ introduced in \eqref{eq:cD} (where we set $b_x\equiv 0$ and $\Sigma_x^2\equiv \mathbf{0}$ for $x\in \cT_L$). According to Remark~\ref{rmk:F sufficient}, Assumption~(A0) is satisfied and the process $(M_t)_{t\geq 0}$ constructed in Section~\ref{sec:construction} is well-defined by Theorem~\ref{thm:existence}.

For consistency, in our model we can allow localized species to be source reactants of any reaction, but we have to make the following \underline{\textbf{Assumption}}:
\begin{enumerate}
 \item[(B0)] A localized species $x$ can only be produced by a reaction $r$ that is localized and taking place at location ${\bar y}_r=\bar{y}_x$.
 \end{enumerate} 
 Indeed, otherwise molecules of type $x$ may pile up at locations different from $\bar{y}_x$, which would contradict our definition of localized species. Note however that this constraint is mainly a consequence of the way we define the mathematical model, and the biology behind localized species and reactions is obviously more complicated.

Let us assume that abundances of different species may scale differently. More precisely, let us suppose that all diffusive species are present in large numbers, of the same order of magnitude $N\in \bN$. Localized species, on the other hand, can either be abundant (with the same order of magnitude as diffusive species), or may appear in small numbers bound to remain of order $\mathcal{O}(1)$. This may happen for instance if some of the localized species are made of very big molecules compared to diffusing species. Hence, the set $\cT_L$ of localized species is further partitioned into (a disjoint union) 
\begin{equation} \label{partition TL}
\cT_L= \cT_{L,s} \cup \cT_{L,b},
\end{equation}
where $\cT_{L,s}$ (\emph{resp.,} $\cT_{L,b}$) denotes the set of localized species with $\mathcal{O}(1)$ (\emph{resp.}, $\mathcal{O}(N)$) abundance. 

Now, the philosophy is as follows: for the species with $\mathcal{O}(N)$ abundance, we scale the corresponding part of each $M_t$ by $N$ and, as we let $N$ tend to infinity, a law of large number-type of result will show the convergence of this part to a deterministic flow $\Phi$. As concerns the species with $\mathcal{O}(1)$ abundance (\emph{i.e.}, species in $\cT_{L,s}$), since they are localized at given points in space, we only have to count how many of them sit at these locations at any time and to use descending factorials of these counts to describe each reaction rate involving at least one such molecule. Furthermore, under appropriate conditions, the rate at which molecules from low abundance species are consumed or created will remain of order $\mathcal{O}(1)$ when we let $N$ tend to infinity, and therefore their dynamics will remain stochastic in the limit. This is the content of Theorem~\ref{thm:largeNbis} below. In some cases (see Proposition~\ref{prop: PDMP}), the limiting process can be shown to correspond to a measure-valued PDMP in which the spatial distribution in abundant species changes continuously according to the flow $\Phi$ (that depends on the current state of the species in $\cT_{L,s}$) and the spatial distribution in low abundance species changes only by jumps at random discrete times whose intensity depends on $\Phi$. In the particular case where $\cT_{L,s}=\emptyset$, we recover a rather classical large-population deterministic limit in which, when a density for the spatial distribution of species exists, this density satisfies a system of coupled partial integro-differential equations (for the diffusive species) and integro-differential equations (for the local amounts of localized species where they sit). This particular case will be developped in Section~\ref{section:RD scaling}.

\begin{remark}
The assumption that \emph{only} localized species may occur in small numbers is crucial to the results expounded in this section. Indeed, if some low abundance species were allowed to diffuse in space, then the reaction rates in the ``deterministic flow'' part of the limiting dynamics (see Theorem~\ref{thm:largeNbis}) would constantly change in a stochastic way and the limiting process would no longer be a PDMP. In this case, the construction and properties of the limit are more involved, and are left for future work.
\end{remark}

Let us now formalise the above intuition. Let $N\in \bN$. For every $t\geq 0$, recall the notation $\mathcal{I}_t$ for the index set of the counting measure $M_t$ and define $M_t^N\in \cM$ as follows:
\begin{equation}\label{def MMN}
M_t^N:= \frac{1}{N}\sum_{\substack{i\in \mathcal{I}_t:\\x_i\in \cT_{L,s}^c}}\delta_{(x_i,y_i)} + \sum_{\substack{i\in \mathcal{I}_t:\\x_i\in \cT_{L,s}}}\delta_{(x_i,y_i)}.
\end{equation}
Since $M_t^N=(1/N)M_t$ on $(\cT_{NL}\cup \cT_{L,b})\times E$ and $M_t^N=M_t$ on $\cT_{L,s}\times E$, it is easy to see that the process $(M_t^N)_{t\geq 0}$ is still Markovian and takes its values in $D_{\cM}[0,\infty)$. As in \eqref{def sampling}, there is a natural notion of sampling without replacement of $k$ particles from a measure $M$ of the form~\eqref{def MMN}, given by~:
\begin{align}\label{MMN w/o}
M^{\otimes \downarrow  k}(dp_1,\ldots,dp_k) :=&\ M (dp_1)\bigg(M - \frac{1}{N^{\mathfrak{s}(x_1)}}\delta_{p_1}\bigg)(dp_2) \cdots \\
& \qquad \qquad \qquad \bigg(M- \frac{1}{N^{\mathfrak{s}(x_1)}}\delta_{p_1}-\cdots - \frac{1}{N^{\mathfrak{s}(x_{k-1})}}\delta_{p_{k-1}}\bigg)(dp_k), \nonumber
\end{align}
where $\mathfrak{s}(x)$ is equal to $0$ if $x\in \cT_{L,s}$ and $1$ otherwise.

Let us suppose that the space-dependent chemical reaction factor $h_r$ depends on $N$, and so do the functions $\Psi_{r,\bar y}$ which we assume (by slight abuse of notation) can be written as
\begin{equation}\label{scaled Psi}
\Psi^N_{r,\bar y} = \frac{1}{N}\Psi_{r,\bar y}\,\1_{\cT_{L,s}^c\times E}+ \Psi_{r,\bar y}\, \1_{\cT_{L,s}\times E},
\end{equation}
where the $\Psi_{r,\bar y}$ are independent of $N$ and satisfy the properties stated in Assumption~(A1) (with $\|\Psi_r\|_\infty$ thus independent of $N$). This assumption is natural since the contribution of the abundant species should globally be of the same order $\mathcal{O}(1)$ as the contribution of the low abundance species.

Before we state the main results of this section, we define a few more pieces of notation. For every reaction $r\in R$, let us write
\begin{equation}\label{simplified exponents}
k_{r,b}:= \sum_{i=1}^{k_r}\1_{\cT_{L,s}^c}(A_i^r) \qquad \hbox{and}\qquad k'_{r,b}:= \sum_{i=1}^{k_r'}\1_{\cT_{L,s}^c}(B_i^r)
\end{equation}
for the respective numbers of source and product molecules in reaction $r$ that are of an abundant type (and the numbers of source and product molecules from a species in low abundance are therefore $k_r-k_{r,b}$ and $k_r'-k'_{r,b}$, respectively). Without loss of generality, we shall assume that the source molecules of abundant types are labelled by $\{1,\ldots,k_{r,b}\}$ and the ones in $\cT_{L,s}$ by $\{k_{r,b}+1,\ldots, k_r\}$. Likewise, the product molecules from an abundant species are labelled by $\{1,\ldots,k'_{r,b}\}$ and those from a low-abundance species by $\{k'_{r,b}+1,\ldots,k_r'\}$. Recall from \eqref{stoichiometry} that $\nu_{r,x}$ and $\nu'_{r,x}$ stand for the stoichiometric coefficients of species $x$ in reaction $r$, and define the sets 
\begin{equation}\label{other set of r}
R^{n\ell}:= \Big\{r\in R\; :\; \sum_{x\in \cT_{L,s}}\big|\nu_{r,x}-\nu'_{r,x}\big|=0\Big\}, \qquad R^\ell := R\setminus R^{n\ell}.
\end{equation}
In words, a reaction $r$ belongs to $R^{n\ell}$ if and only if it does not modify the number of molecules of any species in $\cT_{L,s}$.

For every reaction $r\in R$, let the function $\wth^N_r$ be defined by
\begin{equation}\label{normalised h^Nbis}
\wth^N_r := \left\{ \begin{array}{ll}
N^{k_{r,b}-1}h_r^N & \qquad \hbox{if }r\in R^{n\ell}, \vspace{0.2cm}\\
N^{k_{r,b}}h_r^N &  \qquad  \hbox{if }\, r\in R^{\ell}.
\end{array}\right.
\end{equation}
That is, the scaling property required for $h_r^N$ depends on whether reaction~$r$ modifies the number of species of some localized $\mathcal{O}(1)$-reactant types (the second line) or not (the first line). In the case $r\in R^\ell$, the choice of the exponent $k_{r,b}$ instead of $k_{r,b}-1$ is imposed by the fact that the net change in the abundance of the species in $\cT_{L,s}$ involved leads to a macroscopic jump for $M^N$, instead of a change of the order of $\mathcal{O}(1/N)$. In other words, $r\in R^{n\ell}$ need to be sped up by a factor of $N$ to see $\cO(1)$ change while $r\in R^{\ell}$ can stay at their regular speed. See \eqref{approx GrN} in the proof of Theorem~\ref{thm:largeNbis}. The \underline{\textbf{Assumptions}} on the regularity and boundedness of the different functions and processes that will be needed for the convergence of $(M^N)_{N\geq 1}$ are the following. 
\begin{enumerate}
 \item[(B1)] For every $r\in R$, there exists $\wth_r:E\times \bR_+\rightarrow \bR_+$ which is Lipschitz in the second coordinate, with Lipschitz constant $L_r$ independent of the first coordinate, and such that
 \begin{equation}\label{cond B1a}
{\cal S}_r:= \int_E \varrho_r(d\bar y)\, \sup_{a\geq 0}\wth_r(\bar y,a) <\infty,
 \end{equation}
 and
 \begin{equation}\label{cond B1b}
\lim_{N\rightarrow \infty} \int_E \varrho_r(d\bar y)\, \sup_{a\geq 0} \Big|\wth_r^N(\bar y,a) - \wth_r(\bar y,a)\Big| =0.
 \end{equation}

\item[(B2)] Let $k^*:= 1 + (\max_{r\in R} k_r)$. For every $T>0$, we have
 $$
 \sup_{N\in \bN} \bE\bigg[\sup_{t\in [0,T]}\langle M^N_t,1\rangle^{k^*\vee 2}\bigg] <\infty.
 $$
\end{enumerate}
Assumption~(B1) is the analogue of the first part of Assumption~(A1), while Assumption~(B2) has a role similar to that of Assumption~(A3). Both B-conditions are stronger than the corresponding A-conditions, since the latter are required for the process to be well-defined for every given $N$, while the former will guarantee the uniform integrability of the different terms appearing in the sequence of martingale problems. Note that Assumption~(B2) is not a simple consequence of our assumptions on the reaction coefficients, since nothing guarantees that the limiting dynamics obtained by replacing $\wth_r^N(\bar y,a)$ with $\wth_r(\bar y,a)$ is non-explosive. Like Assumption~(A3), it has to be proven by appropriate comparison or coupling arguments (for instance). 

\begin{remark}
 Notice that unless there are no reactions involving at least one source species, that is, unless all reactions happen from a source external to the system, we have $k^* \vee 2=k^*= 1+\max_r k_r$. The fact that the exponent should be at least $2$ is used in the control of the tail distribution of the total mass of $M^N_t$, see \eqref{condition vee}, and in the extension of the martingale problem to functions $F_f$ with $F(a)=a^2$, see Remark~\ref{rmk:F=id}. 
\end{remark}

Because we want to consider only spatial distributions of species in which molecules from a localized species $x\in \cT_L$ can only be found at location $\bar{y}_x$, we shall restrict our attention to the closed subset $\cMs$ defined by
\begin{equation}\label{def M*}
\cMs:= \bigg\{M\in \cM :\, \sum_{x\in \cT_{L}}\langle M,\1_{(x,\cdot)}-\1_{(x,\bar{y}_x)}\rangle = 0\bigg\}.
\end{equation}
Finally, recall the descending factorial notation $(n)_j=n(n-1)\cdots (n-j+1)$ (with the convention that $(n)_0=1$) and write $\cC^{2,\bot}(\cP)$ for the space of all functions on $\cP=\cT\times E$ that are measurable in the first coordinate, of class $\cC^2$ in the second and such that $f(x',y')\cdot n(y')=0$ for all $(x',y')\in \cT\times \partial E$ (where we recall that $n(y')$ is the outward normal to the boundary of $E$ at $y'\in \partial E$ -- this vector is well-defined for all but a finite number of points in $\partial E$ by assumption).

\begin{table}
\begin{tabular}{|c|l|} \hline
& {\bf Molecular species} \\ \hline 
$\cT$ & finite set of molecular species (or types) -- $\cT = \cT_{NL}\cup \cT_L$  \\
$\cT_{NL}$ & subset of non-localized species \\
$\cT_L$ & subset of localized species  -- $\cT_L=\cT_{L,b}\cup \cT_{L,s}$\\
${\bar y}_x$ & (unique) spatial location of localized species $x\in \cT_L$ \\
$\cT_{L,b}$ & subset of localized species of $\mathcal{O}(N)$ abundance \\
$\cT_{L,s}$ & subset of localized species of $\mathcal{O}(1)$ abundance \\ \hline
& {\bf Reaction types} \\ \hline
$R$ & finite set of reactions -- $R=R_{NL}\cup R_L $ and $R= R^{n\ell}\cup R^\ell$\\
$R_{NL}$ &  subset of non-localized reactions \\
$R_L$ & subset of localized reactions \\
${\bar y}_r$ & (unique) spatial location of localized reaction $r\in R_L$ \\
$R^{n\ell}$ & subset of reactions modifying the counts of none of the species in $\cT_{L,s}$ -- see~\eqref{other set of r}\\
$R^\ell$ & subset of reactions modifying the counts of at least one species in $\cT_{L,s}$ \\
$k_{r,b}$ & number of source reactants in reaction $r$ that are of an abundant type -- see \eqref{simplified exponents}\\
$k'_{r,b}$ & number of product reactants in reaction $r$ that are of an abundant type \\ \hline
\end{tabular}
\caption{\label{table:notation}Main notation for the different types of molecular species and reactions.}
\end{table}

\begin{theorem}\label{thm:largeNbis}
$(i)$ Suppose that Assumptions~(B0), (B1) and (B2) are satisfied, that for every \mbox{$N\geq 1$} we have $M^N_0\in \cMs$ a.s., and that the sequence of random variables $(M^N_0)_{N\geq 1}$ converges in distribution to a random variable $M_0\in \cMs$ as $N$ tends to infinity. Suppose also that there exists at most one solution $(M^\infty_t)_{t\geq 0}$ to the $D_{\cMs}[0,\infty)$-martingale problem: $M^\infty_0\stackrel{(d)}{=}M_0$ and for every $F\in \cC_b^1(\bR)$ with bounded first derivative and every $f\in \cC^{2,\bot}(\cP)$,

\begin{align}
&\Bigg(F(\langle M^\infty_t ,f\rangle) -  F(\langle M^\infty_0,f\rangle) \label{limiting equation 2}\\
& \qquad - \sum_{r\in R^{n\ell}}\int_0^t ds\int_{E\times \cP^{k_{r,b}}}\varrho_r(d\bar y) M_s^\infty(dp_1)\cdots M_s^\infty(dp_{k_{r,b}})\, \wth_r \big(\bar y,\langle M_s^\infty,\Psi_{r,\bar y}\rangle\big) \nonumber \\
& \qquad \qquad \qquad \times \bigg(\prod_{i=1}^{k_{r,b}}\1_{A_i^r}(x_i)\Gamma_\ep(y_i-\bar y)\bigg)\bigg(\prod_{x\in \cT_{L,s}}\Big\{(\langle M^\infty_s,\1_{(x,\cdot)}\rangle)_{\nu_{r,x}}\Gamma_\ep (\bar{y}_x-\bar y)^{\nu_{r,x}}\Big\}\bigg) \nonumber \\
&\qquad \qquad \qquad \times \bigg(\sum_{i=1}^{k_r'}f(B_i^r,\bar y)-\sum_{i=1}^{k_r}f(x_i,y_i)\bigg)F'(\langle M^\infty_s,f\rangle)\nonumber\\
& \qquad - \sum_{r\in R^{\ell}} \int_0^t ds\int_{E\times \cP^{k_{r,b}}}\varrho_r(d\bar y) M_s^\infty(dp_1)\cdots M_s^\infty(dp_{k_{r,b}})\, \wth_r \big(\bar y,\langle M_s^\infty,\Psi_{r,\bar y}\rangle\big) \nonumber \\
 & \qquad \qquad \qquad \times \bigg(\prod_{i=1}^{k_{r,b}}\1_{A_i^r}(x_i)\Gamma_\ep(y_i-\bar y)\bigg)\bigg(\prod_{x\in \cT_{L,s}}\Big\{(\langle M^\infty_s,\1_{(x,\cdot)}\rangle)_{\nu_{r,x}}\Gamma_\ep (\bar{y}_x-\bar y)^{\nu_{r,x}}\Big\}\bigg)\nonumber \\
& \qquad \qquad \qquad \times \bigg(F\Big(\langle M^\infty_s,f\rangle - \sum_{i=k_{r,b}+1}^{k_r}f(x_i,\bar{y}_{x_i})+ \sum_{i=k'_{r,b} +1}^{k_r'}f(B_i^r,\bar{y})\Big)-F\big(\langle M^\infty_s,f\rangle\big)\bigg)\nonumber\\
& \qquad - \int_0^t ds\,  \bigg\{\sum_{x\in \cT_{NL}} \langle M^\infty_s, b_x \cdot \nabla_y f+\Sigma_x^2\circ \Delta_y f\rangle  \bigg\} \, F'(\langle M^\infty_s,f\rangle)\Bigg)_{t\geq 0} \nonumber
\end{align}
is a martingale.

Then, this solution exists and as $N$ tends to infinity, $M^N$ converges in distribution to $M^\infty$ in $D_{\cMs}[0,\infty)$. In addition, $M^{\infty}$ also satisfies the martingale problem~\eqref{limiting equation 2} with $F=\mathrm{Id}$ and $F:a\mapsto a^2$.

$(ii)$ If, furthermore, for every $m_0\in \cMs$ the martingale problem \eqref{limiting equation 2} with $M_0^\infty=m_0$ admits at most one solution, then the limiting process is Markovian.
\end{theorem}

The proof of Theorem~\ref{thm:largeNbis} is given in Section~\ref{section:proofTH}. Observe that in the term of \eqref{limiting equation 2} corresponding to $r\in R^\ell$, only the changes in the low abundance species are visible. Indeed, the simultaneous changes in abundant species are of order $\mathcal{O}(1/N)$ before taking the limit, and therefore they vanish as $N\rightarrow \infty$.

The limiting measure-valued process in Theorem~\ref{thm:largeNbis} is characterised by a well-posed martingale problem, which does not provide an explicit construction of its dynamics. Based on the form of this martingale problem, a natural description of the limit should be in terms of a measure-valued piecewise deterministic Markov process with the ``flow part'' being the continuous change in abundant species due to reactions in $R^{n\ell}$ and to the diffusion of molecules, while the composition in localized low abundance species would change by discrete jumps occurring at a rate given by the terms in \eqref{limiting equation 2} corresponding to reactions in $R^{\ell}$. To our knowledge, there is no general result providing such a correspondence between solutions to martingale problems of the form~\eqref{limiting equation 2} and compositions of flow and jump dynamics for infinite-dimensional stochastic processes, and so the desired relation has to be proven by hand. Although we shall not be able to do it in the full generality of our approach, let us push this direction a bit more. To parallel the general framework for non-spatial reaction networks developed in \cite{CD12} (in which the PDMP take their values in $\bR^c\times \bN^d$), let us decompose every $M_t^\infty$ into its ``continuous'' part $M_t^{\infty,c}$ seen as an element of the set $\cM(\cT_{L,s}^c\times E)$ of all finite measures on $\cT^c_{L,s}\times E$, and its ``discrete'' part $M_t^{\infty,d}$ seen as an element of the set $\cM_p(\cT_{L,s}\times E)$ of all finite point measures on $\cT_{L,s}\times E$, in such a way that $M_t^\infty = M_t^{\infty,c}\otimes M_t^{\infty, d}$. That is, $M_t^{\infty,c}$ (\emph{resp.}, $M_t^{\infty,d}$) is the image measure of $M_t^\infty$ under the projection map to $\cT^c_{L,s}\times E$ (\emph{resp.}, $\cT_{L,s}\times E$). Let us now introduce an $\cM(\cT_{L,s}^c\times E)\times \cM_p(\cT_{L,s}\times E)$-valued process $(\cX_t)_{t\geq 0}$, explicitly constructed in terms of a deterministic flow driving the evolution of the first coordinate and a sequence of discrete jump times for the second coordinate, which we would like to prove to be identical in distribution to  $(M_t^{\infty,c},M_t^{\infty,d})_{t\geq 0}$.

To this end, let us first suppose that for every $(m^c,m^d)\in \cM(\cT_{L,s}^c\times E)\times \cM_p(\cT_{L,s}\times E)$, there exists a unique solution $(\Phi_{m^d}(t,m^c))_{t\geq 0}$ to the following deterministic system: for every $f\in \cC^{2,\bot}(\cP)$, we have for every $t\geq 0$
\begin{align}
&\langle \Phi_{m^d}(t,m^c) ,f\rangle - \langle m^c,f\rangle\nonumber\\
& = \sum_{r\in R^{n\ell}}\int_0^t ds\int_{E\times \cP^{k_{r,b}}}\varrho_r(d\bar y) \Phi_{m^d}(s,m^c)^{\otimes k_{r,b}} (dp_1,\ldots, dp_{k_{r,b}})\, \wth_r \big(\bar y,\langle \Phi_{m^d}(s,m^c)\otimes m^d,\Psi_{r,\bar y}\rangle\big) \nonumber\\
 & \qquad \qquad  \times \bigg(\prod_{i=1}^{k_{r,b}}\1_{A_i^r}(x_i)\Gamma_\ep(y_i-\bar y)\bigg)\bigg(\prod_{x\in \cT_{L,s}}\Big\{\big(m^d(\{(x,{\bar y}_x)\})\big)_{\nu_{r,x}}\Gamma_\ep (\bar{y}_x-\bar y)^{\nu_{r,x}}\Big\}\bigg)\nonumber \\
 & \qquad \qquad \times \bigg(\sum_{i=1}^{k_r'}f(B_i^r,\bar y)-\sum_{i=1}^{k_r}f(x_i,y_i)\bigg)\nonumber\\
& \qquad + \sum_{x\in \cT_{NL}}\int_0^t ds\, \langle \Phi_{m^d}(s,m^c),b_x \cdot \nabla_y f+\Sigma_x^2\circ \Delta_y f\rangle . \nonumber
\end{align}
In the above and later, integrals over $E\times (\cT_{L,s}^c\times E)^{k_{r,b}}$ are replaced by integrals over $E\times \cP^{k_{r,b}}$ to alleviate the notation. Notice that in particular, $\Phi_{m^d}(0,m^c)=m^c$. Using the regularity of the functions $\wth_r$, $\Psi_{r,\bar y}$, $\Gamma_\ep$, $b_x$ and $\Sigma_x^2$, it is then easy to show that each trajectory $(\Phi_{m^d}(t,m^c))_{t\geq 0}$ is continuous for the weak topology on $\cM(\cT_{L,s}^c\times E)$ and that the mapping $(m^c,m^d)\mapsto (\Phi_{m^d}(t,m^c))_{t\geq 0}$ is measurable.

Let us now construct $(\cX_t)_{t\geq 0}=(\cX_t^c,\cX_t^d)_{t\geq 0}$ starting at a pair $(m^c,m^d)\in  \cM(\cT_{L,s}^c\times E)\times \cM_p(\cT_{L,s}\times E)$ (which we allow to be random). For every $r\in R^{\ell}$, define ${\cal E}_r^1$ as an exponential random variable with parameter $1$ (independent of all other variables), and the random time $\tau^1_r$ as
\begin{align*}
\tau^1_r:= \inf&\bigg\{t> 0:\, \int_0^t ds\int_{E\times \cP^{k_{r,b}}} \varrho_r(d\bar y)\Phi_{m^d}(s,m^c)^{\otimes k_{r,b}}(dp_1,\ldots, dp_{k_{r,b}}) \bigg(\prod_{i=1}^{k_{r,b}}\1_{A_i^r}(x_i)\Gamma_\ep(y_i-\bar y)\bigg)\\
&\quad \wth_r \big(\bar y,\langle \Phi_{m^d}(s,m^c)\otimes m^d,\Psi_{r,\bar y}\rangle\big) \bigg(\prod_{x\in \cT_{L,s}}\Big\{\big(m^d(\{(x,{\bar y}_x)\})\big)_{\nu_{r,x}}\Gamma_\ep (\bar{y}_x-\bar y)^{\nu_{r,x}}\Big\}\bigg)\geq {\cal E}_r^1\bigg\}.
\end{align*}
Set $\tau^1:=\min_{r\in R^\ell} \tau^1_r$ and let $r^1$ be the index of the unique reaction satisfying $\tau^1_{r^1}:=\min_r \tau^1_r$. For every $t\in [0,\tau^1)$, set
$$
\cX_t:= \big(\Phi_{m^d}(t,m^c),m^d\big).
$$
In words, until the first reaction $r^1\in R^{\ell}$ occurs at time $\tau^1$, only the continuous part of $\cX$ evolves, according to the flow $\Phi_{m^d}(\cdot,m^c)$. The outcome of reaction $r^1$ is the jump of $\cX$ to
$$
\big(\cX^c_{\tau^1},\cX^d_{\tau^1}\big):= \bigg(\Phi_{m^d}(\tau^1,m^c)\, ,\, m^d-\sum_{i=k_{r^1,b}+1}^{k_{r^1}}\delta_{(A_i^{r^1},\bar{y}_{A_i^{r^1}})} + \sum_{i=k_{r^1,b}'+1}^{k_{r^1}'}\delta_{(B_i^{r^1},\bar{y}_{r^1})}\bigg).
$$

Then for every $j\geq 2$, we proceed recursively following the same ideas:
\begin{itemize}
\item For every $r\in R^{\ell}$, define ${\cal E}_r^j$ as an exponential random variable with parameter $1$ (independent of all other variables), and the random time $\tau^j_r$ as
\begin{align}
\tau^j_r:= \inf\bigg\{t>&\, \tau^{j-1}:\, \int_{\tau^{j-1}}^t ds\int_{E\times \cP^{k_{r,b}}} \varrho_r(d\bar y)\Phi_{\cX_{\tau^{j-1}}^d}(s-\tau^{j-1},\cX_{\tau^{j-1}}^c)^{\otimes k_{r,b}}(dp_1,\ldots,dp_{k_{r,b}}) \nonumber\\
&\bigg(\prod_{i=1}^{k_{r,b}}\1_{A_i^r}(x_i)\Gamma_\ep(y_i-\bar y)\bigg) \wth_r \big(\bar y,\langle \Phi_{\cX_{\tau^{j-1}}^d}(s-\tau^{j-1},\cX_{\tau^{j-1}}^c)\otimes \cX_{\tau^{j-1}}^d,\Psi_{r,\bar y}\rangle\big)\nonumber \\ & \bigg(\prod_{x\in \cT_{L,s}}\Big\{\big(\cX_{\tau^{j-1}}^d(\{(x,{\bar y}_x)\})\big)_{\nu_{r,x}}\Gamma_\ep (\bar{y}_x-\bar y)^{\nu_{r,x}}\Big\}\bigg)\geq {\cal E}_r^j\bigg\} \label{def tau}.
\end{align}
Set $\tau^j:=\min_{r\in R^\ell} \tau^j_r$ and let $r^j$ be the index of the unique reaction satisfying $\tau^j_{r^j}:=\min_r \tau^j_r$. For every $t\in [\tau^{j-1},\tau^j)$, set
\begin{equation}\label{def X}
\cX_t:= \big(\Phi_{\cX_{\tau^{j-1}}^d}(t-\tau^{j-1},\cX_{\tau^{j-1}}^c),\cX_{\tau^{j-1}}^d\big).
\end{equation}
Again, what is encoded here is the fact that after the jump of $\cX$ at time $\tau^{j-1}$ and until the next jump time $\tau^j$, the discrete part of $\cX$ remains constant equal to its value $\cX_{\tau^{j-1}}^d$ at time $\tau^{j-1}$ while its continuous part evolves in a deterministic way according to the flow dictated by the value of the discrete part and starting at $\cX_{\tau^{j-1}}^c$.
\item At time $\tau^j$, a jump due to reaction $r^j$ occurs to the discrete part of $\cX$:
\begin{equation}\label{new value X}
\big(\cX^c_{\tau^j},\cX^d_{\tau^j}\big):= \bigg(\Phi_{\cX_{\tau^{j-1}}^d}(\tau^j-\tau^{j-1},\cX_{\tau^{j-1}}^c)\, ,\, \cX_{\tau^{j-1}}^d-\sum_{i=k_{r^j,b}+1}^{k_{r^j}}\delta_{(A_i^{r^j},\bar{y}_{A_i^{r^j}})} + \sum_{i=k_{r^j,b}'+1}^{k_{r^j}'}\delta_{(B_i^{r^j},\bar{y}_{r^j})}\bigg).
\end{equation}
\end{itemize}

It is difficult to give general conditions under which
$$
\lim_{j\rightarrow \infty}\tau^j = +\infty \qquad \hbox{a.s.}
$$
and $(\cX_t)_{t\geq 0}$ is a well-defined Markov process. Indeed, using Assumption~(A2) and the first part of Assumption~(B1), similar bounds as the ones used in the proof of Theorem~\ref{thm:existence} in Appendix~\ref{s: proof Th1} show that the instantaneous jump rate of the second coordinate of $\cX$ is bounded by a constant times the total mass of $\cX_t^c\otimes \cX_t^d$ to the power (at most) $1+\max_r k_r$. As long as we have not proved that $(\cX_t^c\otimes \cX_t^d)_{t\geq 0}$ and $M^\infty$ are equal in law, we cannot use Assumption~(B2) to control the moments of the total mass of $\cX_t^c\otimes \cX_t^d$ and therefore we have no a priori control over the jump times $\tau^j$, which may accumulate in finite time (of course, the existence of $(M_t^\infty)_{t\geq 0}$ is a good indication that such an accumulation should not happen). Hence, our final result is the following.
\begin{proposition}\label{prop: PDMP}
Suppose that the assumptions of Theorem~\ref{thm:largeNbis}$(i)$ are satisfied, and write $(M_0^c,M_0^d)\in \cM(\cT_{L,s}^c\times E)\times \cM_p(\cT_{L,s}\times E)$ for the decomposition of the initial value $M_0$ of the limiting process $(M^\infty_t)_{t\geq 0}$ into its ``continuous'' and its ``discrete'' parts. Suppose also that $(\cX_t)_{t\geq 0}$ introduced above is a well-defined Markov process with $\cX_0=(M_0^c,M_0^d)$ and such that for every $T>0$,
\begin{equation}\label{integrability cond}
\sup_{t\in [0,T]}\bE\Big[\big(\langle \cX_t^c \otimes \cX_t^d,1\rangle\big)^{1+\max_{r\in R}k_r}\Big]<\infty.
\end{equation}
Then
\begin{equation}\label{equality in law}
(\cX_t)_{t\geq 0} \stackrel{(d)}{=}\big(M_t^{\infty,c},M_t^{\infty,d}\big)_{t\geq 0}.
\end{equation}
\end{proposition}
The result stated in Proposition~\ref{prop: PDMP} relies on the facts that the solution to the martingale problem~\eqref{limiting equation 2} is supposed to be unique and that the decomposition of $M_t^\infty$ into its continuous and discrete parts is unique (so that it suffices to show that $(\cX_t^c\otimes \cX_t^d)_{t\geq 0}$ satisfies the martingale problem characterising the law of $(M_t^\infty)_{t\geq 0}$ to conclude). Since its proof is identical to that of Theorem~\ref{thm:existence}, we omit it. 

\begin{example}\label{ex:PDElimit}{\bf (Diffusive species only).}\\
Consider the simple reaction network given by the set of four reactions~\eqref{eq:exampleloc}. Suppose both molecular types $S,S'$ diffuse in space with constant coefficients $b_S\equiv 0,\Sigma^2_S =\sigma^2_S \mathrm{Id}$ and $b_{S'}\equiv 0, \Sigma^2_{S'}=\sigma^2_{S'}\mathrm{Id}$ respectively (where $\sigma_S^2$ and $\sigma_{S'}^2$ are positive constants), and that both have abundances of order $\mathcal{O}(N)$. Suppose the reaction factors for all non-localized (linear) reactions $h_r^N,r=2,3,4$ are constant in space and have size of order $\mathcal{O}(1)$. That is, $h^N_r(\bar y)\equiv h_r > 0$. Recall that in Example~\ref{example1},  in the unregulated case the reaction factor $h_1(0,a)=h_1(0)$ of the localized reaction $1$ was simply a constant, while in the self-regulated case we used $h_1(0,a)$ with $a=\langle M,\Psi_{S',\ep}\rangle$ (where $\Psi_{S',\ep}$ is a continuous approximation to $\1_{\{S'\}\times B(0,\ep)}$), which depended on the local mass $a$ of $S'$ in the form $h_1(0, a)=c_1/(1+(c_2a)^k)$ for self-repression, and $h_1(0, a)=c_1a^k/(c_2^k+a^k)$ for self-activation. To match the rate of the other reactions (which happen at a rate proportional to the current counts $\mathcal{O}(N)$ of $S$ or $S'$ molecules), suppose that the mass function $\Psi^N_{1,0}$ is equal to $(1/N)\Psi_{S',\ep}$ and that the reaction factor $h_1^N$ has size of order $\mathcal{O}(N)$ : $h_1^N(0,a)=Nh_1(0,a)$ with $h_1$ given above.

Under these specifications, all reaction factors $\tilde h_r=N^{k_r-1}h^N_r(\bar y, a)$ are uniformly bounded in $\bar y$ and for $r=1$, it is uniformly bounded in $a$ as well. Hence the Assumptions of Theorem~\ref{thm:largeNbis}, as well as those of Lemma~\ref{lem:uniqueness of limit} and Proposition~\ref{thm:densityN} stated in Section~\ref{section:RD scaling} and dealing with the existence of a density for the limiting process, can be shown to hold using a multi-type birth and death process as an upper envelope (see Remark~\ref{rmk:case by case}). The density $\mu_t$ of the limiting process solves the following set of deterministic equations: for all $x\in\{S,S'\}$ and $y\in \partial E$, $\nabla_y\mu_t(x,y)\cdot n(y)=0$, and for every $t>0$ and $y\in E$,
 \begin{align}\label{eq:examplepde}
 \partial_t \mu_t(S,y)&=\sigma^2_S\Delta_y\mu_t(S,y) +h_1\bigg(0,\int_{E}\ell_E(d\bar y)\Psi_{S',\ep}(S',\bar y)\mu_t(S',\bar y)\bigg)\1_{\{y=0\}}\nonumber \\
 & \qquad - h_4\bigg(\int_E \ell_E(d\bar y)\Gamma_\ep (y-\bar y)\bigg)\mu_t(S,y),\nonumber \\
 \partial_t \mu_t(S',y)&=\sigma^2_{S'}\Delta_y\mu_t(S',y)+ h_2\int_{E}\ell_E(d\bar y)\Gamma_\ep(\bar y-y)\mu_t(S,\bar y) \nonumber \\
 & \qquad - h_3\bigg(\int_E \ell_E(d\bar y)\Gamma_\ep (y-\bar y)\bigg)\mu_t(S',y),
 \end{align}
with for instance $\mu_0(S,y)=\mu^S\1_{\{y=0\}}$ and $\mu_0(S',y)\equiv 0$ (recall that $\ell_E$ stands for Lebesgue measure on $E$).
 
Note that in the unregulated case $h_1(0,a)=h_1(0)$, the partial differential equation for $\mu_t(S,y)$ is autonomous (with the integral multiplying $h_4$ being a function of $y$ in the $\ep$-boundary of $E$ and constant equal to $1$ elsewhere), while in the self-regulated case it is fully coupled with the equation for $\mu_t(S',y)$. When it exists, the steady state solution $y\mapsto (\pi(S,y), \pi(S',y))$ is solution to the system \eqref{eq:examplepde} with the left-hand side set to 0.
\end{example}

\begin{example} \label{ex:PDMPlimit}{\bf (Diffusive and localized low abundance species).}\\
Using again the reaction network given by~\eqref{eq:exampleloc}, suppose now that only molecules of species $S'$ diffuse in space (in a symmetric and homogeneous way, \emph{i.e.}, $b_{S'}\equiv 0$ and $\Sigma_{S'}^2=\sigma^2_{S'}\mathrm{Id}$ with $\sigma^2_{S'}>0$) while species $S$ is localized at $\bar y=0$. For consistency,  assume that reaction~$4$ is now localized at $\bar y=0$ too. Suppose also that the abundance of species $S'$ is of order $\mathcal{O}(N)$, while the abundance of species $S$ is of order $\mathcal{O}(1)$. This is a simplification of the mechanism where only few molecules of mRNA are transcribed and they stay close to the nucleus, while a larger number of proteins are generated and they diffuse throughout the cell. By definition (see Table~\ref{table:notation}), we thus have $k_{1,b}=0=k_{2,b}=k_{4,b}$ and $k_{3,b}=1$ and reactions~$2$ and $3$ belong to $R^{n\ell}$ (\emph{i.e.}, do not modify the counts of the low abundance species $S$) while reactions~$1$ and $4$ belong to $R^\ell$. Coming back to \eqref{normalised h^Nbis}, we obtain 
$$
k_{1,b}-\1_{\{1\in R^{n\ell}\}} = 0\ ;\  k_{2,b}-\1_{\{2\in R^{n\ell}\}} = -1\ ;\ k_{3,b}-\1_{\{3\in R^{n\ell}\}} = 0\ ;\ k_{4,b}-\1_{\{4\in R^{n\ell}\}} = 0. 
$$ 
Hence, let us suppose that the reaction factor $h_2^N$ is constant through space and of order $\cO(N)$ (that is, $h^N_2(\bar y)\equiv N h_2$ with $h_2>0$), that the reaction factor $h_3^N$ is also constant but of order $\mathcal{O}(1)$ (that is, $h^N_3(\bar y)\equiv h_3 >0$), that the reaction factor $h_4^N$ for the localized reaction~4 is equal to some constant $h_4>0$ independent of $N$, that the reaction factor $h_1^N$ for the localized reaction~1 is of order $\mathcal{O}(1)$ too (that is, $h_1^N(0,a)=h_1(0,a)$ with $h_1$ defined as in Example~\ref{ex:PDElimit}) and finally that the mass function $\Psi^N_{1,0}$ at the location of reaction~1 is again equal to $(1/N)\Psi_{S',\ep}$.

Under these scaling specifications, the assumptions of Theorem~\ref{thm:largeNbis} and of Proposition~\ref{prop: PDMP} can be shown to hold, and the limiting process is a measure-valued piecewise deterministic Markov process described as follows.

$(i)$ The mass of the discrete coordinate $M_t^{\infty,d}$ (number of $S$ molecules at $\bar y=0$) is a Markov jump process started at $M_0(\{(S,0)\})=m_0\in \bN$ with birth rate $h_1(0,\cdot)$ and linear death rate $h_4\Gamma_{\ep}(0)M^{\infty}_t(\{(S,0)\})$; as before, in the unregulated case $h_1(0,a)=h_1(0)$, and  in the self-regulated case $h_1(0, a)=c_1/(1+(c_2a)^k)$ or $h_1(0, a)=c_1a^k/(c_2^k+a^k)$, where $a=\langle M^{\infty}_t, \Psi_{S',\ep}\rangle$.

$(ii)$ The continuous coordinate $M_t^{\infty,c}$ (concentration of $S'$ molecules in $E$) evolves in a deterministic way between the jump times $(\tau^j, j\ge 1)$ of $M^{\infty}_t(\{(S,0)\})$. In each random time interval $[\tau^j,\tau^{j+1})$, when $M_t^{\infty}(\{(S,0)\})$ is constant, the density $\mu_t(S',y)=dM^{\infty}_t(\{(S',y)\})/dy$ satisfies the equation: $\forall y\in E,  \forall t\in [\tau^j,\tau^{j+1})$,
\[\partial_t \mu_t(S',y)=\sigma^2_{S'}\Delta_y\mu_t(S',y)+ h_2\Gamma_\ep(-y)M^{\infty}_t(\{(S,0)\})- h_3\bigg(\int_E \ell_E(d\bar y)\Gamma_\ep (y-\bar y)\bigg)\mu_t(S',y)\]
with (for instance) $\mu_0(S',y)\equiv 0$ and with subsequent initial values given by $\mu_{\tau^j}(S',y)=\lim_{t\uparrow \tau^j}\mu_t(S',y)$.

In the unregulated case, the autonomous dynamics for $M(\{(S,0)\})$ has 
$$
\pi(S,0)\sim \mathrm{Poisson}\bigg(\frac{h_1(0)}{h_4 \Gamma_\ep(0)}\bigg)
$$
as its stationary distribution; and conditional on the value of $\pi(S,0)$, the steady state for $\mu(S',y)$ is determined by solving $$0=\sigma^2_{S'}\Delta_y\pi(S',y)+ h_2\Gamma_\ep(-y)\pi(S,0)- h_3\bigg(\int_E \ell_E(d\bar y)\Gamma_\ep (y-\bar y)\bigg)\pi(S',y), \; \forall y\in E.$$
\end{example}

\section{Proof of the multi-scale limit (Theorem~\ref{thm:largeNbis})} \label{section:proofTH}
We proceed as usual, by first showing that any limit of a converging subsequence of $(M^N)_{N\geq 1}$ necessarily satisfies the martingale problem stated in Theorem~\ref{thm:largeNbis} (but with trajectories that are \emph{a priori} $\cM$-valued). This is done in Section~\ref{ss:limit MP}, where we also argue that this limiting martingale problem holds with $F=\mathrm{Id}$ and $F:a\mapsto a^2$ under our assumptions (see Remark~\ref{rmk:F=id}$(a)$), and that indeed all limiting trajectories a.s. take their values in the subset $\cMs$ (see Remark~\ref{rmk:F=id}$(b)$). In Section~\ref{subsec:tightness} we show that the sequence $(M^N)_{N\geq 1}$ is tight in $D_{\cM}[0,\infty)$. We put all these bricks together and conclude in Section~\ref{ss:conclusion}.

\subsection{Limiting martingale problem}\label{ss:limit MP}
For $f\in \cC(\cP)$, let us set 
\begin{equation}\label{scaling function}
f^N:= \frac{1}{N}f\,\1_{\cT_{L,s}^c\times E} + f\, \1_{\cT_{L,s}\times E}, \qquad \hbox{so that}\quad \langle M_t,f^N\rangle = \langle M^N_t,f\rangle, \ \forall t\geq 0.
\end{equation}

Writing the operator $G_r$ in (\ref{eq:Gr1}) for functions of the form $F(\langle \cdot,f^N\rangle)$ (with $f^N$ defined as above), we obtain that the operator describing the dynamics of $(M^N_t)_{t\geq 0}$ due to reaction $r$, applied to functions of the form $F_f:M\mapsto F(\langle M,f\rangle)$ with $F\in {\cal C}_b(\bR)$ and $f\in {\cal C}(\cP)$, is
\begin{align}
&G_r^NF_f\big(M^N\big) \label{eq:GrN1bis} \\
 &  := \int_{E\times \cP^{k_r}} \varrho_r(d\bar y) N^{k_{r,b}}(M^N)^{\otimes \downarrow k_r}(dp_{1},\cdots dp_{k_r})  h^N_r(\bar y, \langle M^N,\Psi_{r,\bar y}\rangle) \bigg(\prod_{i=1}^{k_r}\big(\1_{A^r_i}(x_i)\Gamma_\ep(y_i-\bar y)\big)\bigg)  \nonumber\\
&  \qquad \times \Bigg[F\bigg(\langle M^N,f\rangle-\frac1N\sum_{i=1}^{k_{r,b}}f(x_i,y_i) - \sum_{i=k_{r,b}+1}^{k_r}f(x_i,y_i)+\frac1N\sum_{i=1}^{k'_{r,b}}f(B^r_i,\bar y) +\sum_{i=k'_{r,b}+1}^{k_r'}f(B^r_i,\bar y)\bigg) \nonumber \\
& \qquad \qquad \qquad \qquad  -F_f\big(M^N\big)\Bigg],  \nonumber
\end{align}
where we recall that the measure $(M^N)^{\otimes \downarrow k}$ corresponding to sampling $k$ particles from $M^N$ without replacement was defined in \eqref{MMN w/o}. 

Likewise, recalling that only the molecules from abundant species diffuse in space, the operator describing the (independent) motion all non-localized molecules can be derived from \eqref{eq:cD} in the same way as above and is given by
\begin{align}
\cD^NF_f(M^N):=&\,  F'(\langle M^N,f\rangle)\sum_{x\in \cT_{NL}} \langle M^N,  b_x\cdot \nabla_yf + \Sigma_x^2 \circ \Delta_y f\rangle \label{scaled cDxbis}\\
& + \ \frac{1}{N}F''(\langle M^N,f\rangle)\sum_{x\in \cT_{NL}} \langle M^N, \Sigma_x^2\circ \big((\nabla_y f)(\nabla_y f)^{\mathbf{t}}\big)\rangle  \nonumber
\end{align}
for every $F\in {\cal C}_b^2(\bR)$, and every $f\in \cC^{2,\bot}(\cP)$. We then have that for each given $N$, and every function of the form $F_f$ satisfying the above conditions, 
\begin{equation}\label{N MP}
\bigg(F_f\big(M^N_t\big) - F_f\big(M^N_0\big) - \int_0^t ds \,\bigg\{ \sum_{r\in R}G_r^NF_f\big(M^N_s\big) + \cD^NF_f(M^N_s) \bigg\}  \bigg)_{t\geq 0}
\end{equation}
is a martingale.

Let us first consider the terms in the martingale problem corresponding to reactions $r\in R$. We shall treat the cases $r\in R^{n\ell}$ and $R^\ell$ separately, since the former requires a Taylor expansion to identify the leading term while the latter does not. 

Let thus $r\in R^{n\ell}$. In this case, the sums over $i\in \{k_{r,b}+1,\cdots, k_r\}$ and $i\in \{k'_{r,b}+1,\cdots, k'_r\}$ in \eqref{eq:GrN1bis} cancel out since reaction $r$ does not modify the counts of species in $\cT_{L,s}$ by assumption, meaning that each species $x\in \cT_{L,s}$ is neither source nor product of the reaction, or it is both source and product reactant, with $\nu_{r,x}=\nu'_{r,x}$ and (by Assumption~(B0)) reaction $r$ is necessarily localized at $\bar y={\bar y}_x$. Hence, only the sums over abundant reactants remain. The key idea below is that since the number of molecules of each abundant species tends to infinity when we let $N$ tend to infinity, in a region where the appropriate molecular types are present, sampling without replacement from abundant species is nearly the same as sampling with replacement, up to an error of order $\cO(1/N)$ that will vanish as $N\rightarrow \infty$. 

Using Assumption~(A2) on the uniform boundedness of $\Gamma_\ep$ and a Taylor expansion of $F$, we have for every $M\in \cM$ of the form~\eqref{def MMN}
\begin{align}
G_r^NF_f(M)& = \int_{E\times \cP^{k_r}} \varrho_r(d\bar y) N^{k_{r,b}}M^{\otimes \downarrow k_r}(dp_{1},\cdots dp_{k_r})h^N_r(\bar y, \langle M,\Psi_{r,\bar y}\rangle)  \label{Taylor expansion}\\
& \qquad \times\bigg(\prod_{i=1}^{k_r}\big(\1_{A^r_i}(x_i)\Gamma_\ep(y_i-\bar y)\big)\bigg) \frac{1}{N}F'\big(\langle M,f\rangle\big)\bigg(\sum_{i=1}^{k_r'}f(B^r_i,\bar y)-\sum_{i=1}^{k_r}f(x_i,y_i)\bigg)\nonumber \\
& \quad+\vep_{r,1}^N(M),\nonumber
\end{align}
where the error term $\vep_{r,1}^N(M)$ satisfies
\begin{equation} \label{error1}
|\vep_{r,1}^N(M)| \leq  \frac{\|F''\|_\infty[(k_r+k_r')\|f\|_\infty]^2}{2N} \|\Gamma_\ep\|_\infty^{k_r} 
\bigg(\int_E \varrho_r(d\bar y)\, \wth_r^N(\bar y,\langle M,\Psi_{r,\bar y}\rangle)\bigg) \langle M,1\rangle^{k_r}
\end{equation}
and $\wth^N_r$ was defined in \eqref{normalised h^Nbis}. Next, if $k_{r,b}<k_r$, for convenience let us suppose that the low-abundance (localized) reactants $A_{k_{r,b}+1}^r,\ldots,A_{k_r}^r$ are listed in such a way that molecules of the same species are grouped together and can be written $a_1^r,\ldots, a_1^r, a_2^r, \ldots, a_2^r$, \emph{etc.}, where $a_{\iota}^r$ appears with multiplicity $\nu_{r,a_{\iota}^r}$. When $k_{r,b}>1$ (otherwise sampling of abundant species with or without replacement is the same), we can then replace $M^{\otimes \downarrow k_r}$ by 
\begin{equation}\label{semi-replacement}
M^{\otimes k_{r,b}} \otimes \bigg(  \bigotimes_{\iota} \big(M (\{(a_{\iota}^r,{\bar y}_{a_{\iota}^r})\})\big)_{\nu_{r,a_{\iota}^r}}\delta_{(a_{\iota}^r,{\bar y}_{a_{\iota}^r})}^{\otimes \nu_{r,a_{\iota}^r}} \bigg)
\end{equation}
in the integral on the r.h.s. of (\ref{Taylor expansion}), up to a combinatorial error term
\begin{align}
&\vep_{r,2}^N(M) \label{error 2} \\
& :=  \int_{E\times \cP^{k_r}} \varrho_r(d\bar y) \big(M^{\otimes \downarrow k_{r,b}}-M^{\otimes k_{r,b}}\big)(dp_{1},\ldots,dp_{k_{r,b}}) M^{\otimes \downarrow (k_r-k_{r,b})}(dp_{k_{r,b}+1},\ldots, dp_{k_r}) \nonumber\\
 & \quad \times \bigg(\prod_{i=1}^{k_r}\big(\1_{A^r_i}(x_i)\Gamma_\ep(y_i-\bar y)\big)\bigg)
 \bigg(\sum_{i=1}^{k_r'}f(B^r_i,\bar y) - \sum_{i=1}^{k_r}f(x_i,y_i)\bigg)  \wth^N_r(\bar y, \langle M,\Psi_{r,\bar y}\rangle)F'(\langle M,f\rangle). \nonumber
\end{align}
Indeed, the second part of the measure in \eqref{semi-replacement} encodes sampling without replacement of the low-abundance localized species involved in reaction $r$, and so is identical to the measure $M^{\otimes\downarrow (k_r-k_{r,b})}\1_{A^r_{k_{r,b}+1}}\cdots \1_{A^r_{k_r}}$ appearing on the r.h.s. of \eqref{error 2}. By the relation between $M^{\otimes \downarrow k_{r,b}}$ and $M^{\otimes k_{r,b}}$, this error term satisfies~:
\begin{align}
|\vep_{r,2}^N(M)|& \leq (k_r+k_r')\|f\|_\infty \|F'\|_\infty \|\Gamma_\ep\|_\infty^{k_r}\bigg(\int_E \varrho_r(d\bar y)\, \wth^N_r(\bar y,\langle M,\Psi_{r,\bar y}\rangle)\bigg) \prod_{\iota}  M(\{(a_{\iota}^r,y_{a_{\iota}^r})\})^{\nu_{r,a_{\iota}^r}}  \nonumber\\
& \qquad\times \big\langle M^{\otimes k_{r,b}}-M^{\otimes \downarrow k_{r,b}}, \1_{(\cT_{L,s}^c\times E)^{k_{r,b}}}\big\rangle.\label{error2a}
\end{align}
Now, 
\begin{align*}
& \big\langle M^{\otimes k_{r,b}}-M^{\otimes \downarrow k_{r,b}}, \1_{(\cT_{L,s}^c\times E)^{k_{r,b}}}\big\rangle \\
&=  \langle M,\1_{\cT_{L,s}^c\times E} \rangle^{k_{r,b}} - \frac{1}{N^{k_{r,b}}}\mathrm{Card}\Big(\big\{(i_1,i_2,\ldots,i_{k_{r,b}})\in \{1,\ldots, N\langle M,\1_{\cT_{L,s}^c\times E}\rangle\}^{k_{r,b}}:\,
i_k\neq i_l \ \forall k,l\big\}\Big) \\
& = \langle M,\1_{\cT_{L,s}^c\times E}\rangle^{k_{r,b}}\bigg(1 - \prod_{j=0}^{k_{r,b}-1}\bigg(1-\frac{j}{N\langle M,\1_{\cT_{L,s}^c\times E}\rangle}\bigg)\bigg),
\end{align*}
with the convention that the product on the r.h.s. is $0$ if $N\langle M,\1_{\cT_{L,s}^c\times E}\rangle<k_{r,b}$. When $k_{r,b}>1$, this quantity is of order $\cO(\langle M,\1_{\cT_{L,s}^c\times E}\rangle^{k_{r,b}-1}/N)$ for a fixed $M$ of total mass of order $\cO(1)$, which, together with \eqref{error2a}, gives us~:
\begin{equation}\label{error2b}
|\vep_{r,2}^N(M)| = \mathcal{O}\bigg(\frac{\langle M,\1_{\cT_{L,s}^c\times E}\rangle^{k_{r,b}-1}}{N}\langle M,1\rangle^{k_r-k_{r,b}}\bigg) \bigg(\int_E \varrho_r(d\bar y) \wth^N_r(\bar y,\langle M,\Psi_{r,\bar y}\rangle)\bigg).
\end{equation}
Assumption~(B1) ensures that the integral over $\bar y$ in the above is uniformly bounded in $N$ and $M$.

Let us now consider $r\in R^{\ell}$. For every $M\in \cM$ of the form \eqref{def MMN}, we can first write that
\begin{align}
 & F\bigg(\langle M,f\rangle-\frac1N\sum_{i=1}^{k_{r,b}}f(x_i,y_i) - \sum_{i=k_{r,b}+1}^{k_r}f(x_i,y_i)+\frac1N\sum_{i=1}^{k'_{r,b}}f(B^r_i,\bar y)+ \sum_{i=k'_{r,b}+1}^{k_r'}f(B^r_i,\bar y)\bigg) \label{partial Taylor expansion} \\
 & \qquad \qquad   -F\big(\langle M,f\rangle\big)\nonumber\\
& = F\bigg(\langle M,f\rangle  - \sum_{i=k_{r,b}+1}^{k_r}f(x_i,y_i) + \sum_{i=k'_{r,b}+1}^{k_r'}f(B^r_i,\bar y) \bigg) - F\big(\langle M,f\rangle\big) +\eta_{r}^N \nonumber\\
& \quad + \frac{1}{N}\, F'\bigg(\langle M,f\rangle  - \sum_{i=k_{r,b}+1}^{k_r}f(x_i,y_i) + \sum_{i=k'_{r,b}+1}^{k_r'}f(B^r_i,\bar y)\bigg)\bigg(\sum_{i=1}^{k'_{r,b}}f(B^r_i,\bar y)-\sum_{i=1}^{k_{r,b}}f(x_i,y_i)\bigg) ,  \nonumber
\end{align}
where $|\eta_{r}^N|\leq (2N^2)^{-1}\|F''\|_\infty(k_r+k_r')^2\|f\|_\infty^2$. Plugging this equation in \eqref{eq:GrN1bis}, we obtain that
\begin{align}
G_r^NF_f(M) = &\int_{E\times{\cP}^{k_r}}\varrho_r(d\bar y) N^{k_{r,b}}M^{\otimes \downarrow k_r}(dp_{1},\cdots dp_{k_r})  h^N_r(\bar y, \langle M,\Psi_{r,\bar y}\rangle) \bigg(\prod_{i=1}^{k_r}\big(\1_{A^r_i}(x_i)\Gamma_\ep(y_i-\bar y)\big)\bigg) \nonumber\\
& \times \Bigg[F\bigg(\langle M,f\rangle - \sum_{i=k_{r,b}+1}^{k_r}f(x_i,y_i) +\sum_{i=k'_{r,b}+1}^{k_r'}f(B^r_i,\bar y) \bigg) - F\big(\langle M,f\rangle\big)\bigg] + \hat{\varepsilon}_{r,1}^N(M),\label{approx GrN}
\end{align}
where
\begin{align}
|\hat{\varepsilon}_{r,1}^N(M)|\leq \|\Gamma_\ep\|_\infty^{k_r}\bigg(& \int_E \varrho_r(d\bar y) \wth_r^N(\bar y,\langle M,\Psi_{r,\bar y}\rangle)\bigg)\langle M,1\rangle^{k_r}\,\bigg(\frac{\|F'\|(k_r+k_r')\|f\|_\infty}{N} \nonumber\\
&\quad  + \frac{\|F''\|_\infty(k_r+k_r')^2\|f\|_\infty^2}{2N^2}\bigg).\label{bound e1}
\end{align}
Next, using the same arguments as in the case $r\in R^{n\ell}$, we can replace sampling without replacement of the abundant species in \eqref{approx GrN} by sampling with replacement up to a combinatorial error $\hat{\varepsilon}_{r,2}^N(M)$ of a similar form as \eqref{error 2}. 

Combining the above, we obtain that for $M\in \cM$ of the form \eqref{def MMN},
\begin{align}
\sum_{r\in R}G^N_rF_f(M) & = F'\big(\langle M,f\rangle\big) \sum_{r\in R^{n\ell}}  \int_{E\times \cP^{k_{r,b}}} \varrho_r(d\bar y) M^{\otimes k_{r,b}}(dp_{1},\cdots dp_{k_{r,b}})\wth^N_r(\bar y, \langle M,\Psi_{r,\bar y}\rangle)  \nonumber\\
& \qquad \times\bigg(\prod_{i=1}^{k_{r,b}}\big(\1_{A^r_i}(x_i)\Gamma_\ep(y_i-\bar y)\big)\bigg) \bigg(\prod_{x\in \cT_{L,s}}\Big\{\big(M(\{(x,{\bar y}_x)\})\big)_{\nu_{r,x}} \Gamma_\ep({\bar y}_x -\bar y)^{\nu_{r,x}}\Big\}\bigg) \nonumber \\
& \qquad \times \bigg(\sum_{i=1}^{k_r'}f(B^r_i,\bar y)-\sum_{i=1}^{k_r}f(x_i,y_i)\bigg)\nonumber \\
&\quad + \sum_{r\in R^\ell}  \int_{E\times \cP^{k_{r,b}}} \varrho_r(d\bar y) M^{\otimes k_{r,b}}(dp_{1},\cdots dp_{k_{r,b}})\wth^N_r(\bar y, \langle M,\Psi_{r,\bar y}\rangle)  \nonumber\\
& \qquad \times\bigg(\prod_{i=1}^{k_{r,b}}\big(\1_{A^r_i}(x_i)\Gamma_\ep(y_i-\bar y)\big)\bigg) \bigg(\prod_{x\in \cT_{L,s}}\Big\{\big(M(\{(x,{\bar y}_x)\})\big)_{\nu_{r,x}} \Gamma_\ep({\bar y}_x -\bar y)^{\nu_{r,x}}\Big\}\bigg) \nonumber \\
& \qquad \times \bigg(F\bigg(\langle M,f\rangle +\sum_{i=k'_{r,b}+1}^{k_r'}f(B^r_i,\bar y)-\sum_{i=k_{r,b}+1}^{k_r}f(x_i,y_i)\bigg)- F_f(M)\bigg)\nonumber \\
& \quad+\sum_{r\in R^{n\ell}}\big(\vep_{r,1}^N(M)+\vep_{r,2}^N(M)\big) + \sum_{r\in R^\ell}\big(\hat{\vep}_{r,1}^N(M)+\hat{\vep}_{r,2}^N(M)\big).  \label{approx GN}
\end{align}

Let us use these results to show that any limit point of a subsequence of $(M^N)_{N\geq 1}$ satisfies the $D_{\cM}[0,\infty)$-martingale problem stated in Theorem~\ref{thm:largeNbis}, for test functions of the form \eqref{eq:Ff} with $F\in \cC_b^2(\bR)$ and $f\in \cC^{2,\bot}(\cP)$. The fact that this property holds true when $F\in \cC_b^1(\bR)$ can then be obtained by a simple density argument, and we shall argue that the trajectories of the limit point(s) of $(M^N)_{N\geq 1}$ remain in $D_{\cMs}[0,\infty)$ in Remark~\ref{rmk:F=id}$(b)$ at the end of this Section.

Let thus $F\in \cC_b^2(\bR)$ and $f\in  \cC^{2,\bot}(\cP)$. By Theorem~\ref{thm:existence} (see also \eqref{N MP}), we know that for any given $N\in \bN$,
$$
\bigg(F_f\big(M^N_t\big)-F_f\big(M^N_0\big)- \int_0^t ds \,\bigg\{\sum_{r\in R} G^N_r F_f\big(M^N_s\big) + \cD^N F_f\big(M^N_s\big)\bigg\}  \bigg)_{t\geq 0}
$$
is a martingale.  Let us prove that for every $t,t'\geq 0$, $k\in \bN$, $0\leq t_1<\ldots<t_k\leq t$ and $\beta_1,\ldots,\beta_k\in \cC_b(\cM)$, we have along any subsequence (which we also denote by $(M^N)_{N\geq 1}$ for simplicity) converging to a limit $(\bar M_t)_{t\geq 0}$
\begin{align}
\lim_{N\rightarrow \infty} &\bE\Bigg[\bigg(\prod_{j=1}^k \beta_j\big(M^N_{t_j}\big)\bigg)\bigg(F_f\big(M^N_{t+t'}\big)-F_f\big(M^N_t\big)  \label{convergence MP}\\
& \qquad \quad \qquad \qquad \qquad \qquad - \int_t^{t+t'} ds \,\bigg\{\sum_{r\in R} G^N_r F_f\big(M^N_s\big) + \cD^N F_f\big(M^N_s\big)\bigg\}\bigg)\Bigg]\nonumber \\
& = \bE\Bigg[\bigg(\prod_{j=1}^k \beta_j\big(\bar M_{t_j}\big)\bigg)\bigg(F_f\big(\bar M_{t+t'}\big)-F_f\big(\bar M_t\big)\nonumber \\
& \qquad \quad \qquad \qquad \qquad \qquad -\int_t^{t+t'} ds \,\bigg\{\sum_{r\in R} G^\infty_r F_f\big(\bar M_s\big) + \cD^\infty F_f\big(\bar M_s\big)\bigg\}\bigg)\Bigg], \nonumber
\end{align}
where
\begin{align}
G^\infty_r F_f(M) &:= F'\big(\langle M,f\rangle\big) \int_{E\times \cP^{k_{r,b}}} \varrho_r(d\bar y) M^{\otimes k_{r,b}}(dp_{1},\cdots dp_{k_{r,b}})\wth_r(\bar y, \langle M,\Psi_{r,\bar y}\rangle)  \label{G infty 1}\\
& \qquad \times\bigg(\prod_{i=1}^{k_{r,b}}\big(\1_{A^r_i}(x_i)\Gamma_\ep(y_i-\bar y)\big)\bigg) \bigg(\prod_{x\in \cT_{L,s}}\Big\{\big(M(\{(x,{\bar y}_x)\})\big)_{\nu_{r,x}} \Gamma_\ep({\bar y}_x -\bar y)^{\nu_{r,x}}\Big\}\bigg) \nonumber \\
& \qquad \times \bigg(\sum_{i=1}^{k_r'}f(B^r_i,\bar y)-\sum_{i=1}^{k_r}f(x_i,y_i)\bigg)\nonumber 
\end{align}
for $r\in R^{n\ell}$,
\begin{align}
G^\infty_r F_f(M) &:=\int_{E\times \cP^{k_{r,b}}} \varrho_r(d\bar y) M^{\otimes k_{r,b}}(dp_{1},\cdots dp_{k_{r,b}})\wth_r(\bar y, \langle M,\Psi_{r,\bar y}\rangle)  \label{G infty 2}\\
& \qquad \times\bigg(\prod_{i=1}^{k_{r,b}}\big(\1_{A^r_i}(x_i)\Gamma_\ep(y_i-\bar y)\big)\bigg) \bigg(\prod_{x\in \cT_{L,s}}\Big\{\big(M(\{(x,{\bar y}_x)\})\big)_{\nu_{r,x}} \Gamma_\ep({\bar y}_x -\bar y)^{\nu_{r,x}}\Big\}\bigg) \nonumber \\
& \qquad \times \bigg(F\bigg(\langle M,f\rangle +\sum_{i=k'_{r,b}+1}^{k_r'}f(B^r_i,\bar y)-\sum_{i=k_{r,b}+1}^{k_r}f(x_i,y_i)\bigg)- F_f(M)\bigg)\nonumber 
\end{align}
for $r\in R^\ell$, and
\begin{equation}
\cD^\infty F_f(M) = F'(\langle M,f\rangle)\sum_{x\in \cT_{NL}} \langle M,  b_x\cdot \nabla_yf + \Sigma_x^2 \circ \Delta_y f\rangle  \label{D infty}.
\end{equation}
Since the expectation on the l.h.s. of \eqref{convergence MP} is zero for every $N$, \eqref{convergence MP} will show that the limit $\bar M$ solves the desired martingale problem. 

We proceed in three steps.

\smallskip
\noindent {\bf Step 1.} By assumption, $(M^N_{t_1},\ldots, M^N_{t_k},M^N_{t},M^N_{t+t'})$ converges in distribution to $(\bar M_{t_1},\ldots, \bar M_{t+t'})$ as $N\rightarrow \infty$ (along the subsequence considered). Since $F_f, \beta_1,\ldots,\beta_k$ are bounded continuous functions on $\cM$, we obtain that
$$
\lim_{N\rightarrow \infty} \bE\Bigg[\bigg(\prod_{j=1}^k \beta_j\big(M^N_{t_j}\big)\bigg)\Big(F_f\big(M^N_{t+t'}\big)-F_f\big(M^N_t\big)\Big)\Bigg]
 = \bE\Bigg[\bigg(\prod_{j=1}^k \beta_j\big(\bar M_{t_j}\big)\bigg)\Big(F_f\big(\bar M_{t+t'}\big)-F_f\big(\bar M_t\big)\Big)\Bigg]. 
$$

\smallskip
\noindent{\bf Step 2.} Let us show that
\begin{align}
\lim_{N\rightarrow \infty}& \bE\left[\bigg(\prod_{j=1}^k\beta_j\big(M^N_{t_j}\big)\bigg)\int_t^{t+t'}ds\, \left|G^N_rF_f\big(M^N_s\big) - G^\infty_rF_f\big(M^N_s\big)\right|\right] = 0 \quad \forall r\in R, \hbox{ and} \label{conv r}\\
\lim_{N\rightarrow \infty}& \bE\left[\bigg(\prod_{j=1}^k\beta_j\big(M^N_{t_j}\big)\bigg)\int_t^{t+t'}ds\, \left|\cD^N F_f\big(M^N_s\big) - \cD^\infty F_f\big(M^N_s\big)\right|\right] = 0. \label{conv x}
\end{align}
Starting with \eqref{conv r} for $r\in R^{n\ell}$ and using \eqref{approx GN} and then \eqref{error1}, we can write that for every $M\in \cM$ of the form~\eqref{def MMN},
\begin{align*}
 \big|& G_r^NF_f(M)- G_r^\infty F_f(M)\big| \\
& \leq (k_r+k_r')\|F'\|_{\infty}\|f\|_\infty \|\Gamma_\ep\|_\infty^{k_r}\langle M,1\rangle^{k_r}\int_E \varrho_r(d\bar y) \Big|\wth_r^N\big(\bar y,\langle M,\Psi_{r,\bar y}\rangle \big)- \wth_r\big(\bar y,\langle M,\Psi_{r,\bar y}\rangle\big) \Big| \\
 & \qquad \qquad + |\varepsilon_{r,1}^N(M)| + |\varepsilon_{r,2}^N(M)|\\
& \leq C_1 \langle M,1\rangle^{k_r}\int_E \varrho_r(d\bar y) \Big|\wth_r^N\big(\bar y,\langle M,\Psi_{r,\bar y}\rangle \big)- \wth_r\big(\bar y,\langle M,\Psi_{r,\bar y}\rangle\big) \Big| \\
& \qquad \qquad + \frac{C_2}{N}\langle M,1\rangle^{k_r} \int_E \varrho_r(d\bar y)\, \wth_r^N(\bar y,\langle M,\Psi_{r,\bar y}\rangle) + |\varepsilon_{r,2}^N(M)|.
\end{align*}
Let us focus on the first two terms on the r.h.s. We have
\begin{align*}
& \bE\bigg[\bigg(\prod_{j=1}^k\beta_j\big(M^N_{t_j}\big)\bigg)\int_t^{t+t'}ds\, \langle M^N_s,1\rangle^{k_r}\bigg\{C_1\int_E \varrho_r(d\bar y) \Big|\wth_r^N\big(\bar y,\langle M^N_s,\Psi_{r,\bar y}\rangle \big)- \wth_r\big(\bar y,\langle M^N_s,\Psi_{r,\bar y}\rangle\big) \Big|\\
& \qquad \qquad \qquad \qquad \qquad \qquad \qquad   + \frac{C_2}{N} \int_E \varrho_r(d\bar y) \wth_r^N(\bar y,\langle M,\Psi_{r,\bar y}\rangle)\bigg\}\bigg] \\
& \leq \bigg(\prod_{j=1}^k \|\beta_j\|_{\infty}\bigg)\bigg(C_1\int_E \varrho_r(d\bar y) \sup_{a\geq 0}\Big|\wth_r^N(\bar y,a)- \wth_r (\bar y,a) \Big| + \frac{C_2}{N}\int_E \varrho_r(d\bar y) \sup_{a\geq 0}\wth_r^N(\bar y,a)\bigg)\\
& \qquad \qquad \times \int_t^{t+t'}ds\, \bE\Big[\langle M^N_s,1\rangle^{k_r}\Big].
\end{align*}
Assumption~(B1) guarantees that the second bracketed expression goes to $0$ as $N\rightarrow \infty$. In addition, Assumption~(B2) implies that the last integral is bounded by some constant depending on $T:=t+t'$ but not on $N$, and so the whole quantity on the above r.h.s. converges to $0$ as $N\rightarrow\infty$.

Next, let us control the term involving $\varepsilon_{r,2}^N$, which is nonzero only if $k_{r,b}\geq 2$. Recall \eqref{error2b}, where the $\cO(1/N)$ asymptotics are valid only if $\langle M,\1_{\cT_{L,s}^c\times E}\rangle$ is not too close to $0$. Hence, let us fix $\delta>0$ and let us decompose the bound on $\varepsilon_{r,2}^N(M^N_s)$ according to whether $\langle M^N_s,\1_{\cT_{L,s}^c\times E}\rangle <\delta$ or not. We obtain, writing $C_\delta>0$ for the constant independent of $N$ entering the first estimate and using (in this case only) that $\langle M^N_s,\1_{\cT_{L,s}^c\times E}\rangle^{k_{r,b}-1} \leq \langle M^N_s,1\rangle^{k_{r,b}-1}$,
\begin{align*}
&\bE\bigg[\bigg(\prod_{j=1}^k\beta_j\big(M^N_{t_j}\big)\bigg)\int_t^{t+t'}ds\, \big|\varepsilon_{r,2}^N\big(M^N_s\big)\big|\bigg] \\
& \leq C_r\bigg(\prod_{j=1}^k \|\beta_j\|_{\infty}\bigg)\bigg(\int_E \varrho_r(d\bar y) \,\sup_{a\geq 0} \wth^N_r(\bar y,a)\bigg) \int_t^{t+t'} ds\, \bE\bigg[\1_{\{\langle M^N_s,\1_{\cT_{L,s}^c\times E}\rangle \geq \delta\}}\frac{C_\delta\langle M^N_s, 1 \rangle^{k_r-1}}{N}  \\
& \qquad \qquad \qquad \qquad \qquad \qquad \qquad \qquad   + \1_{\{\langle M^N_s,\1_{\cT_{L,s}^c\times E}\rangle < \delta\}}\langle M^N_s,\1_{\cT_{L,s}^c\times E}\rangle^{k_{r,b}}\langle M^N_s,1\rangle^{k_r-k_{r,b}}\bigg].
\end{align*}
By Assumption~(B2), $\bE[\langle M^N_s,1 \rangle^{k_r-1}]$ is bounded uniformly in $N\geq 1$ and $s\in [0,t+t']$, and so the first term in the expectation is of the order of $\cO(1/N)$. Furthermore, the second term is bounded by $\delta^{k_{r,b}}\bE[\langle M^N_s,1\rangle^{k_r-k_{r,b}}]$. Consequently, using the uniform bound on the integral of $\wth^N_r$ derived from Assumption~(B1), the fact that $k_r-k_{r,b}< k_r-1$ and Assumption~(B2), and letting $N$ tend to infinity, we obtain that
$$
\limsup_{N\rightarrow \infty} \bE\bigg[\bigg(\prod_{j=1}^k\beta_j\big(M^N_{t_j}\big)\bigg)\int_t^{t+t'}ds\, \big|\varepsilon_2\big(M^N_s\big)\big|\bigg] \leq Ct'\delta^{k_{r,b}}
$$
for a constant $C$ independent of $\delta$. Since this is true for any $\delta>0$ and since we consider only the case $k_{r,b}\geq 2$, we can conclude that this limit is actually $0$. Combining the above, we obtain \eqref{conv r} for $r\in R^{n\ell}$. Exactly the same reasoning applies and yields \eqref{conv r} for $r\in R^\ell$.

The convergence result \eqref{conv x} is obtained by noticing that
$$
\bigg|\frac{1}{N}\, F''(\langle M,f\rangle)  \langle M, \Sigma_x^2\circ \big((\nabla_y f)(\nabla_y f)^{\mathbf{t}}\big)\rangle \bigg| \leq \frac{1}{N} \|F''\|_{\infty} \|\Sigma_x^2\circ \big((\nabla_y f)(\nabla_y f)^{\mathbf{t}}\big)\|_{\infty} \, \langle M,1\rangle,
$$
and using the same arguments as before.

\smallskip
\noindent{\bf Step 3.} Finally, let us show that
\begin{align}
&\lim_{N\rightarrow \infty} \bE\Bigg[\bigg(\prod_{j=1}^k \beta_j\big(M^N_{t_j}\big)\bigg)\bigg(\int_t^{t+t'} ds \,\bigg\{\sum_{r\in R} G^\infty_r F_f\big(M^N_s\big) + \cD^\infty F_f\big(M^N_s\big)\bigg\}\bigg)\Bigg]\nonumber\\
& = \bE\Bigg[\bigg(\prod_{j=1}^k \beta_j\big(\bar M_{t_j}\big)\bigg)\bigg(\int_t^{t+t'} ds \,\bigg\{\sum_{r\in R} G^\infty_r F_f\big(\bar M_s\big) +  \cD^\infty F_f\big(\bar M_s\big)\bigg\}\bigg)\Bigg].\label{L3}
\end{align}

By construction and Assumption~(B1), we have for every $M\in \cM$
\begin{equation} \label{bounds Grinfty}
\big|G_r^\infty F_f(M)\big| \leq \langle M,1\rangle^{k_r} {\cal S}_r \|\Gamma_\ep\|^{k_r}_\infty \|F'\|_\infty  (k_r+k_r') \|f\|_\infty
\end{equation}
when $r\in R^{n\ell}$,
\begin{equation} \label{bounds Grinfty 2}
\big|G_r^\infty F_f(M)\big| \leq \langle M,1\rangle^{k_r} {\cal S}_r  \|\Gamma_\ep\|^{k_r}_\infty 2 \|F\|_\infty
\end{equation}
when $r\in R^\ell$, and 
\begin{equation}\label{bounds Dinfty}
\big|\cD^\infty F_f(M)\big| \leq \langle M,1\rangle \|F'\|_\infty \sum_{x\in \cT_{NL}} \|b_x\cdot \nabla_yf + \Sigma_x^2 \circ \Delta_y f\|_\infty.
\end{equation}
Thanks to these bounds, together with Assumption~(B2) which ensures that each of these quantities is integrable, we can use Fubini's theorem to exchange order between integration and summation in the above and treat each term separately.

Starting with the terms corresponding to reactions $r\in R$, if $k_r=0$, a simple dominated convergence argument together with the continuity of $G_r^\infty F_f$ suffice to conclude. If $k_r\geq 1$, we use \eqref{bounds Grinfty} to observe that the function $G_r^\infty F_f$ is bounded over any subset of $\cM$ of measures with total mass less than a given quantity. Hence, let $\delta>0$ be small, $g$ be a continuous function with values in $[0,1]$, such that $g(a)=1$ if $a\leq -\delta$ and $g(a)=0$ if $a>\delta$, and let $A>0$ be a large constant. We use $g(\cdot -A)$ as a continuous approximation to $\1_{\{\cdot \leq A\}}$. We have
\begin{align}
\int_t^{t+t'}&ds\, \bE\Bigg[\bigg(\prod_{j=1}^k \beta_j\big(M^N_{t_j}\big)\bigg) G_r^\infty F_f\big(M^N_s\big)\Bigg] \nonumber \\
& = \int_t^{t+t'}ds\, \bE\Bigg[\bigg(\prod_{j=1}^k \beta_j\big(M^N_{t_j}\big)\bigg) G_r^\infty F_f\big(M^N_s\big)g\big(\langle M^N_s,1\rangle-A\big)\Bigg] \nonumber \\
& \quad + \int_t^{t+t'}ds\, \bE\Bigg[\bigg(\prod_{j=1}^k \beta_j\big(M^N_{t_j}\big)\bigg) G_r^\infty F_f\big(M^N_s\big)\big(1-g\big(\langle M^N_s,1\rangle-A\big)\big)\Bigg].\label{decomp gA}
\end{align}
Since $G_r^\infty F_f(\cdot) g\big(\langle \cdot,1\rangle-A\big)$ is bounded and continuous (as well as each $\beta_j$), the dominated convergence theorem guarantees that the first term on the r.h.s. of \eqref{decomp gA} converges to
\begin{equation}\label{limit with A}
\int_t^{t+t'}ds\, \bE\Bigg[\bigg(\prod_{j=1}^k \beta_j\big(\bar M_{t_j}\big)\bigg) G_r^\infty F_f\big(\bar M_s\big)g\big(\langle \bar M_s,1\rangle-A\big)\Bigg]
\end{equation}
as $N\rightarrow \infty$ (along the converging subsequence). On the other hand, the second term on the r.h.s. of \eqref{decomp gA} is bounded by
\begin{align}
& C \int_t^{t+t'} ds\, \bE\Big[\langle M^N_s,1 \rangle^{k_r}\big(1-g\big(\langle M^N_s,1\rangle-A\big)\big) \Big] \nonumber \\
& \leq C \int_t^{t+t'} ds\, \bE\Big[\langle M^N_s,1 \rangle^{pk_r}\Big]^{1/p} \bE\Big[\big(1-g\big(\langle M^N_s,1\rangle-A\big)\big)^q \Big]^{1/q} \label{condition vee}
\end{align}
for some constant $C$ independent of $N,A$, and any pair $(p,q)$ such that $p>1$, $1/p+1/q=1$ and $pk_r\leq k^*\vee 2$ (so that we may use H\"older's inequality to pass from the first to the second line). By Assumption~(B2), the first expectation on the r.h.s. is bounded uniformly in $N$ and $s\in [0,t+t']$. Because $1-g$ is bounded, continuous and satisfies $1-g(\cdot -A) \leq \1_{\{\cdot >A-\delta\}}$, we have
\begin{align*}
\lim_{N\rightarrow \infty} \int_t^{t+t'}ds\,\bE\Big[\big(1-g\big(\langle M^N_s,1\rangle-A\big)\big)^q \Big]^{1/q} & =  \int_t^{t+t'}ds\, \bE\Big[\big(1-g\big(\langle \bar M_s,1\rangle-A\big)\big)^q \Big]^{1/q}\\
& \leq \int_t^{t+t'}ds\, \bP\big[ \langle \bar M_s ,1\rangle > A-\delta \big]^{1/q}.
\end{align*}
Consequently, we have
\begin{align*}
\limsup_{N\rightarrow \infty}
& \int_t^{t+t'}ds\, \bE\Bigg[\bigg(\prod_{j=1}^k \beta_j\big(M^N_{t_j}\big)\bigg) G_r^\infty F_f\big(M^N_s\big)\big(1-g\big(\langle M^N_s,1\rangle-A\big)\big)\Bigg]\\
& \qquad \leq C' \int_t^{t+t'}ds\, \bP\big[ \langle\bar M_s ,1\rangle > A-\delta \big]^{1/q},
\end{align*}
where the constant $C'$ is independent of $A$. Since the integrand is bounded by $1$, and tends to $0$ as $A$ tends to infinity provided that we know that $\langle \bar M_s,1\rangle <\infty$ a.s. for every $s\in [t,t+t']$ (see below for a proof), we can use the dominated convergence theorem to conclude that the quantity on the last line tends to $0$ as $A\rightarrow \infty$. It only remains to show that as $A\rightarrow \infty$, the quantity in \eqref{limit with A} tends to
$$
\int_t^{t+t'}ds\, \bE\Bigg[\bigg(\prod_{j=1}^k \beta_j\big(\bar M_{t_j}\big)\bigg) G_r^\infty F_f\big(\bar M_s\big)\Bigg].
$$
But using Fatou's Lemma and Assumption~(B2), we see that 
\begin{equation}\label{Fatou}
\sup_{s\in [0,T]}\bE\big[\langle \bar M_s,1\rangle^{k^*\vee 2}\big]<\infty, \qquad \forall T>0.
\end{equation} 
The first consequence is that, as supposed in the previous paragraph, $\langle \bar M_s,1\rangle$ is a.s. finite for every $s$. Second, by \eqref{bounds Grinfty}, $G_r^\infty F_f(\bar M_s)$ is also integrable (and so is $\cD^\infty F_f(\bar M_s)$, by the same arguments). Since $g$ is bounded by $1$ and $g(\cdot - A)$ converges pointwise to $1$ as $A$ tends to infinity, the dominated convergence theorem gives us the desired convergence.

Exactly the same chain of arguments gives us the convergence of the term in \eqref{L3} involving $\cD^{\infty}F_f$, and so we do not repeat it.

\smallskip
As desired, we have proved that any limit point of a convergent subsequence of $(M^N)_{N\in \bN}$ satisfies the martingale problem stated in Theorem~\ref{thm:largeNbis} (with trajectories that are $\cM$-valued).
\begin{remark}\label{rmk:F=id} $(a)$ By taking two sequences of functions $(F^{1,n})_{n\geq 1}$ and $(F^{2,n})_{n\geq 1}$ in $\cC_b^2(\bR)$, each with bounded first and second derivatives, and converging respectively to the identity function and to $a\mapsto a^2$ uniformly over compact intervals and such that their first derivatives converge respectively to $1$ and $a\mapsto 2a$ uniformly over compact intervals, we can show (using again Assumptions~(B1-B2) and the different bounds of $G_r^N F^{i,n}_f$ and $\cD^N F^{i,n}_f$ obtained in this section) that the pre-limiting and limiting martingale problems hold also with $F^1(a)=a$ and $F^2(a)=a^2$. For $F^2$, observe that the component $(F^2)'(\langle M,f\rangle)$ of some of the terms in the (approximations to the) operators involved adds another factor $\langle M,1\rangle$ in all the bounds. However, Assumption~(B2) controls the $(k^*\vee 2)$-th moment of the total mass of the processes, where $k^* = 1+ \max_r k_r$, which is precisely the moment appearing in those bounds. 

$(b)$ Likewise, by taking a sequence $(f_n)_{n\geq 1}$ in $\cC^{2,\bot}(\cP)$ approximating $f=\1_{\{(x,\cdot)\}}-\1_{\{(x,\bar{y}_x)\}}$ uniformly over $\cP$, for any $x\in \cT_L$, we can show that any limit point of $(M^N)_{N\geq 1}$ has trajectories taking their values in $\cMs$ a.s. (provided the sequence of initial values $(M^N_0)_{N\geq 1}$ converges in $\cMs$).
\end{remark}

\subsection{Tightness of $(M^N)_{N\geq 1}$}\label{subsec:tightness}
Using the criterion of \cite{RC86} and the fact that $f\in \cC^{2,\bot}(\cP)$ is dense in $\cC(\cP)$ in the topology of uniform convergence over the compact space $\cP$, if we can show that the \emph{compact containment condition} is satisfied, then tightness of $(M^N)_{N\geq 1}$ will be equivalent to tightness of $(\langle M^N,f\rangle)_{N\geq 1}$ for every $f\in \cC^{2,\bot}(\cP)$. Now, since $\cP$ is compact, for every $a>0$ the set
$$
E_a:=\{M\in \cM:\, \langle M,1\rangle \leq a\}
$$
is a compact subset of $\cM$. By Assumption~(B2) and the Markov inequality, for every $T>0$ and every $\eta\in (0,1)$ there exists $a_{T,\eta}>0$ such that
$$
\inf_{N\geq 1}\bP\big(M^N_t \in E_{a_{T,\eta}}, \ \forall t\in [0,T]\big) \geq 1-\eta,
$$
and so the compact containment condition holds indeed.

We now call on the criterion from Aldous~\cite{Al78} and Rebolledo~\cite{Re80} and use the semi-martingale decomposition of $(\langle M^N_t, f\rangle)_{t\geq 0}$ to show that this sequence of processes is tight. More precisely, we use the following lemma. 
\begin{lemma}\label{lem: semimartingale}
Let $f\in \cC^{2,\bot}(\cP)$ and let $N\in \bN$. For every $t\geq 0$, let us define $V^N_t$ and $Z^N_t$ by:
$$
V^N_t := \int_0^t ds\, \bigg\{\sum_{r\in R} G^N_r \mathrm{Id}_f(M^N_s) + {\cal D}^N \mathrm{Id}_f(M^N_s)\bigg\},
$$
where $\mathrm{Id}_f(M)=\langle M,f\rangle$, and
$$
Z^N_t := \langle M^N_t,f\rangle - \langle M^N_0,f\rangle - V^N_t.
$$
Then the process $(Z^N_t)_{t\geq 0}$ is a square integrable martingale, with predictable quadratic variation
\begin{align*}
\big\langle Z^N\big\rangle_t  = &\, \frac{1}{N} \int_0^t ds\, \bigg\{\sum_{r\in R^{n\ell}}\int_{E\times \cP^{k_r}}\varrho_r(d\bar y)\big(M_s^N\big)^{\otimes \downarrow k_r}(dp_1,\ldots,dp_{k_r}) \wth^N_r\big(\bar y,\langle M^N_s ,\Psi_{r,\bar y}\rangle\big)\\
&  \qquad \qquad \times \bigg(\prod_{i=1}^{k_r} \big(\1_{A_i^r}(x_i)\Gamma_\ep (y_i-\bar y)\big)\bigg)\bigg(\sum_{i=1}^{k_r'}f\big(B_i^r,\bar y\big) -\sum_{i=1}^{k_r} f(x_i,y_i)\bigg)^2\bigg\}\\
& +\int_0^t ds\, \bigg\{\sum_{r\in R^\ell}\int_{E\times \cP^{k_r}} \varrho_r(d\bar y) \big(M_s^N\big)^{\otimes \downarrow k_r}(dp_1,\ldots,dp_{k_r})  \wth^N_r\big(\bar y,\langle M^N_s ,\Psi_{r,\bar y}\rangle\big)\\
& \qquad \qquad \times \bigg(\prod_{i=1}^{k_r} \big(\1_{A_i^r}(x_i)\Gamma_\ep (y_i-\bar y)\big)\bigg)\\
& \qquad  \times \bigg(\frac{1}{N}\sum_{i=1}^{k'_{r,b}}f\big(B_i^r,\bar y\big) + \sum_{i=k'_{r,b}+1}^{k'_r}f\big(B_i^r,\bar y\big) - \frac{1}{N}\sum_{i=1}^{k_{r,b}} f(x_i,y_i) - \sum_{i=k_{r,b}+1}^{k_r}f(x_i,y_i)\bigg)^2\bigg\} \\
& +\frac{2}{N}\int_0^t ds\, \bigg\{\sum_{x\in \cT_{NL}} \big\langle M^N_s, \Sigma_x^2 \circ \big( (\nabla_yf) (\nabla_y f)^{\bf t}\big) \big\rangle \bigg\}.
\end{align*}
\end{lemma}

\begin{proof}{(Proof of Lemma~\ref{lem: semimartingale}.)} By Remark~\ref{rmk:F=id}$(a)$, the processes $(\langle M^N_t,f\rangle)_{t\geq 0}$ and $(\langle M^N_t,f\rangle^2)_{t\geq 0}$ satisfy the martingale problem \eqref{N MP} written with $F=\mathrm{Id}$ and $F:a\mapsto a^2$ respectively. Based on this, the identification of the predictable finite variation and quadratic variation terms in the semi-martingale decomposition of $\langle M^N,f\rangle$ is standard (see, \emph{e.g.}, the proof of Theorem~3.3 in \cite{ChM07}). Note that the factor $1/N$ in the first term of the expression for $\langle Z^N\rangle_t$ comes from the fact that  for $r\in R^{n\ell}$, $(1/N^2)\times N^{k_{r,b}}h^N_r(\bar y,\langle M,\Psi_{r,\bar y}\rangle)= (1/N)\wth^N_r(\bar y,\langle M,\Psi_{r,\bar y}\rangle)$ for all $\bar y, M$ by definition of $\wth^N_r$. \end{proof}

Let us now fix $f\in \cC^{2,\bot}(\cP)$. Let $T>0$ and $(\tau_N)_{N\geq 1}$ be a sequence of stopping times bounded by $T$. Using Lemma~\ref{lem: semimartingale} and (\ref{eq:GrN1bis}--\ref{scaled cDxbis}) with $F=\mathrm{Id}$, we can write for every $t\in [0,1]$
\begin{align*}
\big|V^N_{\tau^N+t}-V^N_{\tau^N}\big|& = \bigg| \int_{\tau^N}^{\tau^N+t} ds\, \bigg\{\sum_{r\in R} G^N_r \mathrm{Id}_f(M^N_s) + {\cal D}^N \mathrm{Id}_f(M^N_s)\bigg\}\bigg|\\
& \leq \int_{\tau^N}^{\tau^N+t} ds\, \bigg\{\sum_{r\in R}\big|G^N_r \mathrm{Id}_f(M^N_s)\big|+ \big|{\cal D}^N \mathrm{Id}_f(M^N_s)\big|\bigg\} \\
& \leq \sum_{r\in R}\bigg[(k_r+k'_r)\|f\|_\infty\|\Gamma_\ep\|_\infty^{k_r}\bigg(\int_E \varrho_r(d{\bar y}) \sup_{a\geq 0}\wth^N_r(\bar y,a)\bigg)\int_{\tau^N}^{\tau^N+t}ds\, \langle M^N_s,1\rangle^{k_r}\bigg]\\
& \quad +\sum_{x\in \cT_{NL}} \bigg[\| b_x\cdot \nabla_yf + \Sigma_x^2 \circ \Delta_y f \|_\infty  \int_{\tau^N}^{\tau^N+t}ds\, \langle M^N_s,1\rangle\bigg].
\end{align*}
Since $\tau^N+t\leq T+1$, Assumption~(B2) guarantees that
$$
\bE\left[\sup_{s\in [\tau^N,\tau^N+t]} \langle M^N_s,1\rangle^k \right]\leq \bE\left[\sup_{s\in [0,T+1]} \langle M^N_s,1\rangle^k \right]
$$
is bounded uniformly in $N$ for every $k\leq k^*\vee 2$. Therefore, using Fubini's theorem as well as Assumption~(B1) to control the terms involving the $\wth^N_r$, we obtain the existence of a constant $C_T>0$ independent of $N$ such that for every $t\in [0,1]$,
$$
\bE\Big[\big|V^N_{\tau^N+t}-V^N_{\tau^N}\big|\Big]\leq C_T t.
$$
Using the Markov inequality, we can therefore conclude that for any $\eta>0$, there exists $\delta=\delta(\eta,T)>0$ such that
$$
\sup_{N\geq 1}\sup_{t\in [0,\delta]}\bP\Big[\big|V^N_{\tau^N+t}-V^N_{\tau^N}\big|>\eta\Big]\leq \eta,
$$
and the first part of the Aldous-Rebolledo criterion is satisfied. Likewise, there exists $C'_{T,n\ell}, C'_{T,\ell}$ independent of $N$ such that
\begin{equation}\label{bound QV}
\bE\Big[\big|\big\langle Z^N\big\rangle_{\tau^N+t}-\big\langle Z^N\big\rangle_{\tau^N}\big|\Big]\leq \frac{C'_{T,n\ell}\mathrm{Card}(R^{n\ell})}{N} t + C'_{T,\ell}\mathrm{Card}(R^{\ell})t,
\end{equation}
and so the same conclusion as above holds for the quadratic variation process $(\langle Z^N\rangle_t)_{t\geq 0}$. The second part of the Aldous-Rebolledo criterion is satisfied too, and hence the sequence $(\langle M^N,f\rangle)_{N\geq 1}$ is tight. This completes the proof of the tightness of $(M^N)_{N\geq 1}$ in $D_{\cM}[0,\infty)$.

\subsection{Conclusion of the proof of Theorem~\ref{thm:largeNbis}}\label{ss:conclusion}
To conclude, let us write $L^\infty$ for the operator on which the martingale problem \eqref{limiting equation 2} is based~:
\begin{equation}\label{def Linfty}
L^\infty := \sum_{r\in R} G_r^\infty + \cD^\infty,
\end{equation}
where $G_r^\infty$ and $\cD^\infty$ were defined respectively in (\ref{G infty 1}--\ref{G infty 2}) and \eqref{D infty}. We use Theorem~4.8.10 in \cite{EK86}, in which the condition that the operator $L^\infty$ should take its values in $\cC_b(\cM)$ is replaced by the bounds \eqref{bounds Grinfty}, \eqref{bounds Grinfty 2} and \eqref{bounds Dinfty} which, together with \eqref{Fatou}, ensure that the limiting local martingales
$$
\Big(F_f(M_t^\infty)-F_f(M_0^\infty) -\int_0^t ds\, L^\infty F_f(M^\infty_s) \Big)_{t\geq 0}
$$
are integrable and are therefore true martingales for every limit $M^\infty$ of a converging subsequence of $(M^N)_{N\geq 1}$. By Remark~\ref{rmk:F=id}$(b)$, such a limit takes its values in $\cMs$ a.s. Equation~\eqref{convergence MP} shows that Condition~$(b')$ of Theorem~4.8.10 is satisfied and since we assumed that there was at most one solution to the $D_{\cMs}[0,\infty)$-martingale problem~\eqref{limiting equation 2}, we thus obtain that the unique solution $M^\infty$ to the $D_{\cMs}[0,\infty)$-martingale problem indeed exists and that $M^N$ converges to it in distribution as $N\rightarrow \infty$, in $D_{\cM}[0,\infty)$. Using again Remark~\ref{rmk:F=id}$(b)$ and Corollary~3.3.2 in \cite{EK86}, we can conclude that the convergence also holds in $D_{\cMs}[0,\infty)$. Remark~\ref{rmk:F=id}$(a)$ shows that the limiting martingale problem also holds for $F=\mathrm{Id}$ and $F:a\mapsto a^2$, and so Theorem~\ref{thm:largeNbis}$(i)$ is proved.

The fact that if uniqueness of the solution to the limiting martingale problem holds for every initial condition $m_0\in \cMs$, then $M^\infty$ has the Markov property, is a consequence of Theorem~4.4.2(a) in \cite{EK86}. The proof of Theorem~\ref{thm:largeNbis} is now complete.

\section[Properties of the reaction-diffusion scaling limit]{Properties of the reaction-diffusion scaling limit (abundant species only)}\label{section:RD scaling}
In this section, we further explore the properties of the limiting process obtained in Theorem~\ref{thm:largeNbis} in the particular case where all species are abundant, that is $\cT_{L,s}=\emptyset$.  We also suppose that the initial value $M_0^\infty$ is deterministic and equal to some $M_0\in \cMs$. In particular, we give a sufficient condition for the assumption on uniqueness of the solution to \eqref{limiting equation 2} to be satisfied (see Lemma~\ref{lem:uniqueness of limit}), then we give some conditions under which the measures $M_t^\infty$ admits a density (see Proposition~\ref{thm:densityN} for a more precise statement). The results obtained in this section and their proofs are largely inspired by the results and methods developed in \cite{ChM07}, in which only linear birth of particles and pairwise interactions are considered (but these ``reactions'' are allowed to have spatially inhomogeneous rates, as in our framework).

\begin{remark}
Notice that, in what follows, we allow the presence of localized species, but assume that their abundances are of order $\mathcal{O}(N)$ before taking the limit.
\end{remark}

Since $\cT_{L,s}=\emptyset$, we obviously have $\mathrm{Card}(R^\ell)=0$ and so by \eqref{bound QV} and Doob's maximal inequality, for every $f\in \cC^{2,\bot}(\cP)$, $T\geq 0$, $\eta>0$ and every $N\geq 1$ we have
$$
\bP\left(\sup_{0\leq t\leq T}\big|Z^N_t\big|\geq \eta\right) \leq \frac{1}{\eta^2}\bE\big[(Z^N_T)^2\big]= \frac{1}{\eta^2}\bE\Big[\big\langle Z^N\big\rangle_T\Big] \leq \frac{C'_TT}{\eta^2 N},
$$
where $Z^N=Z^N(f)$ is the martingale defined in Lemma~\ref{lem: semimartingale}. This probability goes to $0$ as $N\rightarrow \infty$, and so the limit $M^\infty$ of $(M^N)_{N\geq 1}$ is the deterministic solution to~: $M^\infty_0 =M_0$ and for every $f\in \cC^{2,\bot}(\cP)$,
\begin{equation}\label{deter limit}
\langle M_t^\infty,f\rangle = \langle M_0^\infty,f\rangle + \int_0^t ds\, \bigg\{\sum_{r\in R}G_r^\infty \mathrm{Id}_f(M_s^\infty) + \cD^\infty \mathrm{Id}_f(M_s^\infty)\bigg\},\qquad \forall t\geq 0.
\end{equation}

As in the definition of the pre-limiting stochastic model (see Remark~\ref{rmk:case by case}), the assumption on uniqueness of the solution $M^\infty$ has to be checked case by case, as it may hold for very different reasons. Note however the following result.
\begin{lemma}\label{lem:uniqueness of limit}
Suppose that the conditions on the functions $\Psi_{r,\bar{y}}$ stated in Assumption~(A1) are satisfied, together with Assumptions~(A2) and (B1). Suppose also that for any solution $(m_t)_{t\geq 0}$ to the set of equations~\eqref{deter limit} and any $T>0$, we have
\begin{equation}\label{cond lemma uniqueness}
m_T^*:=\sup_{t\in [0,T]} \langle m_t,1\rangle <\infty.
\end{equation}
Then for any $m\in \cMs$, there is at most one solution to \eqref{deter limit} with initial condition $m$. 
\end{lemma}

\begin{proof}{(Proof of Lemma~\ref{lem:uniqueness of limit}.)}
For each $x\in \cT_{NL}$, let $(P_x^t,\,t\ge 0)$ denote the semigroup on ${\cal C}(E)$ of the diffusion process with drift coefficient $b_x$ and dispersion matrix $\Sigma_x$, normally reflected at the boundary of $E$. By extension, for $x\in \cT_L$ we also write $(P_x^t,\,t\geq 0)$ for the semigroup of the trivial process following which particles do not move, \emph{i.e.}, $P_x^t = \mathrm{Id}$ for every $t\ge 0$. When necessary, we shall abuse notation and write $P^t_x\varphi(x,y)$ for the function $P^t_x\varphi(x,\cdot)$ applied at $y\in E$. The following is a mild formulation of the equations satisfied by a solution $(m_t)_{t\geq 0}$ to \eqref{deter limit} (the proof follows exactly the same lines as that of Lemma~4.5 in \cite{ChM07}). For each $\varphi\in \cC^{2,\bot}(\cP)$, we have
 \begin{align}
\langle m_t ,\varphi\rangle -  \langle m_0,P^t_\cdot\varphi\rangle  &\ =   \sum_{r\in R}\int_0^t ds\int_{E\times \cP^{k_r}}\varrho_r(d\bar y) m_s(dp_1)\cdots m_s(dp_{k_r})\, \wth_r \big(\bar y,\langle m_s,\Psi_{r,\bar y}\rangle\big) \label{eq:mmild} \\
 & \qquad \times \bigg(\prod_{i=1}^{k_r}\1_{A_i^r}(x_i)\Gamma_\ep(y_i-\bar y)\bigg)\bigg(\sum_{i=1}^{k_r'}P^{t-s}_{B_i^r}\varphi(B_i^r,\bar y)-\sum_{i=1}^{k_r}P^{t-s}_{x_i}\varphi(x_i,y_i)\bigg),\nonumber
 \end{align}
for every $t\geq 0$. By Remark~\ref{rmk:F sufficient} and the dominated convergence theorem, the equality \eqref{eq:mmild} also holds for any $\varphi\in \cC(\cP)$.
 
Let $m', m''$ be two solutions to \eqref{deter limit}, and let us consider $\varphi\in \cC(\cP)$ such that $\|\varphi\|_\infty \le 1$. Then
 \begin{align}
&\big|\langle m'_t-m''_t,\varphi\rangle\big|- |\langle m'_0-m''_0,P^t_\cdot\varphi\rangle| \label{bound diff}\\
 & \le\sum_{r\in R}\int_0^t ds\Bigg|\int_{E\times \cP^{k_r}} \varrho_r(d\bar y) \bigg(\prod_{i=1}^{k_r}\1_{A_i^r}(x_i)\Gamma_\ep(y_i-\bar y)\bigg) \bigg(\sum_{i=1}^{k_r'}P^{t-s}_{B_i^r}\varphi(B_i^r,\bar y)-\sum_{i=1}^{k_r}P^{t-s}_{x_i}\varphi(x_i,y_i)\bigg)  \nonumber\\
 &\qquad \times\bigg(\wth_r \big(\bar y,\langle m'_s,\Psi_{r,\bar y}\rangle\big)m'_s(dp_1)\cdots m'_s(dp_{k_r})-\wth_r \big(\bar y,\langle m''_s,\Psi_{r,\bar y}\rangle\big) m''_s(dp_1)\cdots m''_s(dp_{k_r})\bigg) \Bigg|.  \nonumber
 \end{align}
Using the facts that $\|\varphi\|_\infty \le 1$ and that the semigroups $(P^t_x)_{t\geq 0}$ preserve the supremum norm, we obtain that the r.h.s. of \eqref{bound diff} is bounded by
 \begin{align}
 & \sum_{r\in R} \|\Gamma_\ep\|^{k_r}_\infty(k'_r+k_r)\int_0^t ds\Bigg|\int_{E\times \cP^{k_r}} \varrho_r(d\bar y) \bigg(\prod_{i=1}^{k_r}\psi_i^r(x_i,y_i,\bar y)\bigg)
\Big(\wth_r \big(\bar y,\langle m'_s,\Psi_{r,\bar y}\rangle\big)m_s'^{\;\otimes k_r} \label{mild inequalities}\\
& \qquad \qquad \qquad \qquad -\wth_r \big(\bar y,\langle m''_s,\Psi_{r,\bar y}\rangle\big) m_s''^{\;\otimes k_r}\Big)(dp_1,\ldots,dp_{k_r}) \Bigg|, \nonumber 
\end{align}
where we have defined for every $r\in R$, $i\in\{1,\ldots,k_r\}$, and $x_i,y_i,\bar y\in E$,
$$
\psi_i^r(x_i,y_i,\bar y) := \1_{A_i^r}(x_i)\, \frac{\Gamma_\ep(y_i-\bar y)}{\|\Gamma_\ep\|_\infty}.
$$
Since $\|\psi_i^r\|_\infty\leq 1$, if $k_r\geq 1$ we can write that for every $\bar y\in E$ and every $s\geq 0$,
\begin{align*}
& \bigg|\int_{\cP^{k_r}}\bigg(\prod_{i=1}^{k_r} \psi_i^r(x_i,y_i,\bar y)\bigg)m_s'(dp_1)\cdots m_s'(dp_{k_r})- \int_{\cP^{k_r}}\bigg(\prod_{i=1}^{k_r} \psi_i^r(x_i,y_i,\bar y)\bigg)m_s''(dp_1)\cdots m_s''(dp_{k_r})\bigg|\\
& =\bigg|\prod_{i=1}^{k_r} \langle m_s',\psi_i^r(\cdot,\cdot,\bar y)\rangle - \prod_{i=1}^{k_r} \langle m_s'',\psi_i^r(\cdot,\cdot,\bar y)\rangle\bigg| \\
& = \bigg|\big( \langle m_s',\psi_1^r(\cdot,\cdot,\bar y)\rangle -\langle m_s'',\psi_1^r(\cdot,\cdot,\bar y)\rangle \big)\prod_{i=2}^{k_r}\langle m_s',\psi^r_i(\cdot,\cdot,\bar y)\rangle \\
& \qquad \qquad + \langle m_s'',\psi_1^r(\cdot,\cdot,\bar y)\rangle\bigg(\prod_{i=2}^{k_r}\langle m_s',\psi_i^r(\cdot,\cdot,\bar y)\rangle - \prod_{i=2}^{k_r}\langle m_s'',\psi_i^r(\cdot,\cdot,\bar y)\rangle \bigg)\bigg| \\
& \leq \langle m_s',1\rangle^{k_r-1} \big|\langle m_s',\psi_1^r(\cdot,\cdot,\bar y)\rangle -\langle m_s'',\psi_1^r(\cdot,\cdot,\bar y)\rangle\big| \\
& \qquad \qquad + \langle m_s'',1\rangle\bigg|\prod_{i=2}^{k_r}\langle m_s',\psi^r_i(\cdot,\cdot,\bar y)\rangle- \prod_{i=2}^{k_r}\langle m_s'',\psi^r_i(\cdot,\cdot,\bar y)\rangle\bigg|.
\end{align*}
By an easy recursion, we thus obtain that
$$
\int_{{\cP}^{k_r}}\bigg(\prod_{i=1}^{k_r}\psi^r_i(x_i,y_i,\bar y)\bigg)\Big(m_s'^{\;\otimes k_r}- m_s''^{\;\otimes k_r}\Big)(d\mathbf{p})
\le k_r\langle m'_s+ m''_s,1\rangle^{k_r-1}\mathop{\sup}\limits_{\psi:\|\psi\|_\infty\le 1}\big|\langle m'_s-m''_s,\psi\rangle\big|,
$$
where the supremum is taken over all $\psi\in \cC(\cP)$ such that $\|\psi\|_\infty \le 1$. In addition, by Assumption~(B1) we have for $r\in R$
\begin{align*}
\int_E \varrho_r(d\bar y)\big|\wth_r \big(\bar y,\langle m'_s,\Psi_{r,\bar y}\rangle\big) - \wth_r \big(\bar y,\langle m''_s,\Psi_{r,\bar y}\rangle\big)\big| &\le L_r\int_E\varrho_r(d\bar y) \big|\langle m'_s- m''_s,\Psi_{r,\bar y}\rangle\big| \\
& \le L_r\|\Psi_r\|_\infty (\mathrm{Vol}(E)+1)\sup_{\psi:\|\psi\|_\infty\le 1}\big|\langle m'_s- m''_s,\psi\rangle\big|.
\end{align*}
Combining the above and using the assumption on the total mass of $(m'_t)_{t\geq 0}$ and $(m''_t)_{t\geq 0}$ stated in Lemma~\ref{lem:uniqueness of limit}, we obtain that for every time horizon $T>0$ and every $s\in [0,T]$,
\begin{align}
& \bigg|\int_E \varrho_r(d\bar y)\int_{{\cP}^{k_r}}\bigg(\prod_{i=1}^{k_r}\psi^r_i(x_i,y_i,\bar y)\bigg)\Big(\wth_r(\bar y, \langle m'_s,\Psi_{r,\bar y}\rangle)m_s'^{\otimes k_r}-\wth_r(\bar y, \langle m''_s,\Psi_{r,\bar y}\rangle)m_s''^{\otimes k_r}\Big)(d\mathbf{p})\bigg| \nonumber\\
& \le \bigg|\int_E \varrho_r(d\bar y)\int_{{\cP}^{k_r}}\bigg(\prod_{i=1}^{k_r}\psi^r_i(x_i,y_i,\bar y)\bigg)\Big(\wth_r(\bar y, \langle m'_s,\Psi_{r,\bar y}\rangle)-\wth_r(\bar y, \langle m''_s,\Psi_{r,\bar y}\rangle)\Big) m_s'^{\otimes k_r}(d\mathbf{p})\bigg| \nonumber\\
&\qquad +\bigg|\int_E \varrho_r(d\bar y) \,\wth_r(\bar y, \langle m''_s,\Psi_{r,\bar y}\rangle)\int_{{\cP}^{k_r}}\bigg(\prod_{i=1}^{k_r}\psi^r_i(x_i,y_i,\bar y)\bigg)\big(m_s'^{\otimes k_r}-m_s''^{\otimes k_r}\big)(d\mathbf{p})\bigg| \nonumber\\
& \le L_r\|\Psi_r\|_\infty (\mathrm{Vol}(E)+1)\langle m_s',1\rangle^{k_r} \mathop{\sup}\limits_{\psi:\|\psi\|_\infty\le 1}\big|\langle m_s'-m_s'',\psi\rangle\big| \nonumber \\
& \qquad \qquad \qquad  + \mathcal{S}_r k_r\langle m_s'+m_s'',1\rangle^{k_r-1}\mathop{\sup}\limits_{\psi:\|\psi\|_\infty\le 1}\big|\langle m_s'- m_s'',\psi\rangle\big| \nonumber\\
& \le \Big[L_r\|\Psi_r\|_\infty (\mathrm{Vol}(E)+1)\big((m')_T^*+(m'')_T^*\big)+ \mathcal{S}_rk_r\Big] \nonumber \\
& \qquad \qquad \qquad \times \big((m')_T^*+(m'')_T^*\big)^{k_r-1} \mathop{\sup}\limits_{\psi:\|\psi\|_\infty\le 1}\big|\langle m_s'-m_s'',\psi\rangle\big|, \label{bound on difference}
\end{align}
where $\mathcal{S}_r$ was defined in \eqref{cond B1a}. Observe that for any reaction $r$ with $k_r=0$, the rate of reaction is independent of mass so its contribution to $\langle m_t'- m_t'', \varphi\rangle$ is zero. This implies that for every $T>0$, there exists a constant $C>0$ such that for every $\varphi\in \cC(\cP)$ with $\|\varphi\|_\infty \le 1$, and every $t\leq T$, we have
$$
\big|\langle m_t'-m_t'',\varphi\rangle\big|\le \big|\langle m_0'-m_0'',P^t_\cdot\varphi\rangle\big|+ C\int_0^t ds \mathop{\sup}\limits_{\psi:\|\psi\|_\infty\le 1}\big| \langle m_s'-m_s'',\psi\rangle\big|,
$$
and Gronwall's inequality yields for all $t\in [0,T]$
\begin{equation}\label{eq:m-estimate}
\mathop{\sup}\limits_{\varphi:\|\varphi\|_\infty\le 1} \big|\langle m_t'-m_t'',\varphi\rangle\big| \le \mathop{\sup}\limits_{\varphi:\|\varphi\|_\infty\le 1} \big|\langle m_0'-m_0'',P^t_\cdot\varphi\rangle\big|\,e^{Ct}.
\end{equation}
Hence, whenever $m_0'=m_0''$ we have $m_t'=m_t''$ for all $t\le T$ and uniqueness holds. \end{proof}

For our last result, Proposition~\ref{thm:densityN}, which gives conditions under which at any time the spatial distributions of non-localized species have a density with respect to Lebesgue measure on $E$ while the localized species have an evolving mass at the locations where they sit, we restrict our attention to a particular case which is already rich (and notationally heavy): for every $x\in \cT_{NL}$, we suppose that the diffusion matrix $\Sigma_x^2$ is of the form $\sigma_x^2\mathrm{Id}$, where $\sigma_x^2: E\mapsto \bR_+$ is Lipschitz.

We shall use the following set of functions $\mathbf{D}$ as the set of possible densities. Recall the notation $\ell_E$ for Lebesgue measure on $E$.
\begin{definition}\label{def density}
We say that a function $\varphi:\cP\rightarrow \bR$ is in $\mathbb{L}^1(\cP)$ if
\begin{equation}\label{def LP}
\|\varphi\|_1:=\sum_{x\in \cT_{NL}}\, \int_E \ell_E(dy) |\varphi(x,y)| + \sum_{x\in \cT_{L}}|\varphi(x,\bar{y}_x)|<\infty.
\end{equation}
We call $\mathbf{D}$ the set of nonnegative functions $\varphi$ such that $\varphi\in \mathbb{L}^1(\cP)$ and $\varphi(x,y)=0$ whenever $x\in \cT_{L}$ and $y\neq \bar{y}_x$.
\end{definition}

The additional constraints on the drift and variance coefficients for the diffusion of non-localized species are summarised in the following \underline{\bf Assumption}:
\begin{enumerate}
 \item[(B3)] There exists $\sigma^2_*>0$ such that $\sigma_x^2(y)\geq \sigma_*^2$ for all $(x,y)\in \cT_{NL}\times E$. Furthermore, for every $x\in \cT_{NL}$, $\sigma_x^2$ is of class $C^2$ on $E$ and its second derivatives with respect to the $d$ spatial coordinates $y_1,\ldots,y_d$ are $\alpha$-H\"olderian for some $\alpha>0$. 
 
In addition, for every $x\in \cT_{NL}$, the drift coefficient $b_x$ is of class $C^1$ and its derivatives with respect to the $d$ spatial coordinates are $\alpha$-H\"olderian for some $\alpha>0$.
\end{enumerate}
Assumption~(B3) is analogous to Assumption~(H2) in \cite{ChM07}, where it is used to guarantee the existence of a spatial density for the semigroups corresponding to the motions of particles (which are all assumed to diffuse in \cite{ChM07}). We shall also need a last assumption, that can only be stated in a rigourous way once we have introduced the appropriate sequence $\{\mu^n,\, n\geq 0\}$ of approximations to the density. This assumption is used in the proof of Proposition~\ref{thm:densityN} to have a uniform control on the total mass of the approximate densities over any compact time interval.

\begin{proposition}\label{thm:densityN}
Suppose that the regularity and boundedness assumptions on each $\wth_r$ stated in Assumption~(B1) are satisfied, together with Assumption~(B3) and Assumption~(B4) (see the proof of the Proposition for a statement). Let $M^\infty$ be the solution to~\eqref{deter limit}, and suppose that there exists $\mu_0^\infty\in \mathbf{D}$ such that
$$
M^\infty_0= \sum_{x\in \cT_{NL}}\mu_0^\infty(x,y)\delta_{x}\otimes \ell_E + \sum_{x\in \cT_{L}} \mu_0^\infty(x,\bar{y}_x)\delta_x\otimes \delta_{\bar{y}_x}.
$$
That is, for every $f\in \cC(\cP)$,
$$
\langle M^\infty_0 ,f\rangle = \sum_{x\in \cT_{NL}}\int_E \ell_E(dy) \, \mu_0^\infty(x,y)f(x,y) + \sum_{x\in \cT_{L}} \mu_0^\infty(x,\bar{y}_x)f(x,\bar{y}_x).
$$
Then for every $t\geq 0$, there exists a function $\mu_t \in \mathbf{D}$ such that
\begin{equation}\label{decomp Mt}
M^\infty_t= \sum_{x\in \cT_{NL}}\mu_t(x,y) \delta_{x}\otimes \ell_E + \sum_{x\in \cT_{L}} \mu_t(x,\bar{y}_x)\delta_x\otimes \delta_{\bar{y}_x}.
\end{equation}
Moreover, for every $T>0$
\begin{equation}
\sup_{t\in [0,T]}\|\mu_t\|_1 <\infty
\end{equation}
and $(\mu_t)_{t\geq 0}$ is a weak solution to the following system of integro-differential and partial integro-differential equations: for every $x\in \cT_{NL}$, and for every $y\in E$
\begin{align}
&\partial_t \mu_t(x,y)=\Delta_y\big(\sigma_x^2(y)\mu_t(x,y)\big)-\nabla_y\cdot \big(b_x(y)\mu_t(x,y)\big) \nonumber\\
&\qquad -\sum_{r\in R} \sum_{i=1}^{k_r}\1_{A^r_i}(x)\mu_t(x,y)\int_E \varrho_r(d\bar y) \bigg(\wth_r(\bar y, \mu_t)\Gamma_\ep(y-\bar y)\prod_{\substack{j=1\\j\neq i,}}^{k_r}\bigg\{\1_{\{A_j^r\in \cT_L\}} \Gamma_\ep({\bar y}_{A_j^r}-\bar y)\mu_t({A_j^r},{\bar y}_{A_j^r})\nonumber \\
& \qquad \qquad \qquad + \1_{\{A_j^r\in \cT_{NL}\}}\int_E \ell_E(dy_j)\Gamma_\ep(y_j-\bar y)\mu_t({A_j^r},y_j)   \bigg\}\bigg)\nonumber\\
& \qquad+\sum_{r\in R_{NL}}\bigg(\sum_{i=1}^{k'_r}\1_{B^r_i}(x)\bigg)\wth_r(y, \mu_t) \prod_{j=1}^{k_r}\bigg\{\1_{\{A_j^r\in \cT_L\}} \Gamma_\ep({\bar y}_{A_j^r}-y)\mu_t({A_j^r},{\bar y}_{A_j^r})\nonumber \\
& \qquad \qquad \qquad + \1_{\{A_j^r\in \cT_{NL}\}}\int_E \ell_E(dy_j)\Gamma_\ep(y_j- y)\mu_t({A_j^r},y_j)   \bigg\}\nonumber\\
& \qquad +\sum_{r\in R_{L}}\1_{\{y=\bar y_r\}}\bigg(\sum_{i=1}^{k'_r}\1_{B^r_i}(x)\bigg)\wth_r(\bar y_r, \mu_t) \prod_{j=1}^{k_r}\bigg\{\1_{\{A_j^r\in \cT_L\}} \Gamma_\ep({\bar y}_{A_j^r}-{\bar y}_r)\mu_t({A_j^r},{\bar y}_{A_j^r})\nonumber \\
& \qquad \qquad \qquad + \1_{\{A_j^r\in \cT_{NL}\}}\int_E \ell_E(dy_j)\Gamma_\ep(y_j- {\bar y}_r)\mu_t({A_j^r},y_j)   \bigg\}\ ;\nonumber\\
& \mu_0(x,\cdot)=\mu_0^\infty(x,\cdot); \;\; \nabla_y\mu_t(x,y')\cdot n(y')=0\ \hbox{for all }t>0 \hbox{ and } y'\in\partial E\ ; \label{pde density}
\end{align}
and for every $x\in \cT_{L}$,
\begin{align}
&\partial_t \mu_t(x,\bar{y}_x) \nonumber\\
&= -\sum_{r\in R} \sum_{i=1}^{k_r}\1_{A^r_i}(x)\mu_t(x,\bar{y}_x)\int_E \varrho_r(d\bar y) \bigg(\wth_r(\bar y, \mu_t)\Gamma_\ep(\bar{y}_x-\bar y) \prod_{\substack{j=1\\j\neq i}}^{k_r}\bigg\{\1_{\{A_j^r\in \cT_L\}} \Gamma_\ep({\bar y}_{A_j^r}-\bar y)\mu_t({A_j^r},{\bar y}_{A_j^r})\nonumber \\
& \qquad \qquad \qquad + \1_{\{A_j^r\in \cT_{NL}\}}\int_E \ell_E(dy_j)\Gamma_\ep(y_j-\bar y)\mu_t({A_j^r},y_j)   \bigg\}\bigg)\nonumber \\
&\quad +\sum_{r\in R_{NL}}\bigg(\sum_{i=1}^{k'_r}\1_{B^r_i}(x)\bigg)\wth_r(\bar{y}_x, \mu_t)\prod_{j=1}^{k_r}
\bigg\{\1_{\{A_j^r\in \cT_L\}} \Gamma_\ep({\bar y}_{A_j^r}-{\bar y}_x)\mu_t({A_j^r},{\bar y}_{A_j^r})\nonumber \\
& \qquad \qquad \qquad + \1_{\{A_j^r\in \cT_{NL}\}}\int_E \ell_E(dy_j)\Gamma_\ep(y_j-{\bar y}_x)\mu_t({A_j^r},y_j)   \bigg\}\nonumber \\
&\quad + \sum_{r\in R_{L}}\1_{\{\bar{y}_x=\bar{y}_r\}}\bigg(\sum_{i=1}^{k'_r}\1_{B^r_i}(x)\bigg)\wth_r(\bar{y}_r, \mu_t)\prod_{j=1}^{k_r}\bigg\{\1_{\{A_j^r\in \cT_L\}} \Gamma_\ep({\bar y}_{A_j^r}-{\bar y}_x)\mu_t({A_j^r},{\bar y}_{A_j^r})\nonumber \\
& \qquad \qquad \qquad + \1_{\{A_j^r\in \cT_{NL}\}}\int_E \ell_E(dy_j)\Gamma_\ep(y_j-{\bar y}_x)\mu_t({A_j^r},y_j)   \bigg\} \ ; \nonumber\\
&\mu_0(x,\bar{y}_x)=\mu_0^\infty(x,\bar{y}_x).\label{de density}
\end{align}
In the above, we have abused notation and written $\wth_r(\cdot ,\mu_t)$ to mean $\wth_r(\cdot,\langle M_t,\Psi_{r,\cdot}\rangle)$, where $M_t$ is the measure built out of the function $\mu_t$.
\end{proposition}

\begin{proof}{(Proof of Proposition~\ref{thm:densityN}.)} 
We follow the lines of the proof of Theorem~4.6 in \cite{ChM07} closely, mainly using the same arguments as in our proof of Lemma~\ref{lem:uniqueness of limit} (which therefore we do not repeat entirely). In all that follows, we shall suppose that for each $n$, $\mu^n_t(x,\cdot)$ is defined for all $x\in\cT$ and $t\geq 0$ with the understanding that $\mu^n_t(x,y)=\mu^n_t(x,\bar{y}_x)\1_{\bar{y}_x}(y)$ when $x\in\cT_{L}$ (that is, $\mu^n\in \mathbf{D}$). To ease the notation, as in \eqref{def varrho} we also define
\begin{equation}\label{def varrhox}
\varrho^x(d\bar y)= \ell_E(d\bar{y}) \quad \hbox{if }x\in \cT_{NL} \qquad \hbox{and} \qquad \varrho^x(d\bar y)= \delta_{{\bar y}_x}(d\bar{y}) \quad \hbox{if }x\in \cT_L.
\end{equation}
Consider the following collection of inductively constructed functions $\{\mu^n, n\ge0\}$. First, we set $\mu^0_t=\mu_0^\infty$ for all $t\geq 0$. Suppose that for some $n\geq 0$, $(\mu^n_t)_{t\geq 0}$ is well-defined and takes its values in $\mathbf{D}$. For every ${x\in\cT_{NL}}$, define $(\mu^{n+1}_t(x,\cdot))_{t\geq 0}$ as the weak solution to the partial differential equation
\begin{align}
&\partial_t \mu^{n+1}_t(x,y)=\Delta_y\big(\sigma_x^2(y)\mu^{n+1}_t(x,y)\big)-\nabla_y\cdot\big(b_x(y)\mu^{n+1}_t(x,y)\big) \nonumber\\
&\qquad -\sum_{r\in R} \sum_{i=1}^{k_r}\1_{A^r_i}(x)\mu^{n+1}_t(x,y)\int_E \varrho_r(d\bar y) \bigg(\wth_r(\bar y, \mu^n_t)\Gamma_\ep(y-\bar y)\prod_{\substack{j=1\\j\neq i}}^{k_r}\bigg\{\int_E \varrho^{A^r_j}(dy_j)\Gamma_\ep(y_j-\bar y)\mu^n_t({A_j^r},y_j)\bigg\}\bigg)\nonumber\\
& \qquad+\sum_{r\in R_{NL}}\bigg(\sum_{i=1}^{k'_r}\1_{B^r_i}(x)\bigg)\wth_r(y, \mu^n_t) \prod_{j=1}^{k_r}\bigg\{\int_E \varrho^{A^r_j}(dy_j)\Gamma_\ep(y_j-y)\mu^n_t(A^r_j,y_j)\bigg\}\nonumber\\
& \qquad +\sum_{r\in R_{L}}\1_{\{y=\bar y_r\}}\bigg(\sum_{i=1}^{k'_r}\1_{B^r_i}(x)\bigg)\wth_r(\bar y_r, \mu^n_t) \prod_{j=1}^{k_r}\bigg\{\int_E \varrho^{A^r_j}(dy_j)\Gamma_\ep(y_j-\bar y_r)\mu^n_t(A^r_j,y_j)\bigg\}\ ;\nonumber\\
& \mu^{n+1}_0(x,\cdot)=\mu_0^\infty(x,\cdot); \;\; \nabla_y\mu^{n+1}_t(x,y')\cdot n(y')=0\ \hbox{for all }t>0 \hbox{ and } y'\in\partial E. \label{eq:4.19NL}
\end{align}
For every $x\in\cT_{L}$, define $(\mu^{n+1}_t(x,\bar{y}_x))_{t\geq 0}$ as the solution to the ordinary differential equation:
\begin{align}
&\partial_t \mu^{n+1}_t(x,\bar{y}_x) \nonumber\\
&= -\sum_{r\in R} \sum_{i=1}^{k_r}\1_{A^r_i}(x)\mu^{n+1}_t(x,\bar{y}_x)\int_E \varrho_r(d\bar y) \bigg(\wth_r(\bar y, \mu^n_t)\Gamma_\ep(\bar{y}_x-\bar y) \prod_{\substack{j=1\\j\neq i}}^{k_r}\bigg\{\int_E \varrho^{A_j^r}(dy_j)\Gamma_\ep(y_j-\bar y)\mu^n_t({A_j^r},y_j)\bigg\}\bigg)\nonumber \\
&\quad +\sum_{r\in R_{NL}}\bigg(\sum_{i=1}^{k'_r}\1_{B^r_i}(x)\bigg)\wth_r(\bar{y}_x, \mu^n_t)\prod_{j=1}^{k_r}\bigg\{\int_E \varrho^{A^r_j}(dy_j)\Gamma_\ep(y_j-\bar{y}_x)\mu^n_t(A^r_j,y_j)\bigg\}\nonumber\\
&\quad + \sum_{r\in R_{L}}\1_{\{\bar{y}_x=\bar{y}_r\}}\bigg(\sum_{i=1}^{k'_r}\1_{B^r_i}(x)\bigg)\wth_r(\bar{y}_r, \mu^n_t)\prod_{j=1}^{k_r}\bigg\{\int_E \varrho^{A^r_j}(dy_j)\Gamma_\ep(y_j-\bar{y}_r)\mu^n_t(A^r_j,y_j)\bigg\}\ ; \nonumber\\
&\mu^{n+1}_0(x,\bar{y}_x)=\mu_0^\infty(x,\bar{y}_x).\label{eq:4.19L}
\end{align}
Existence of these functions follows from standard results on linear parabolic equations, see \emph{e.g.} Theorem~7.3 in \cite{EV98}, since the lower bound on $\sigma^2$ required in Assumption~(B3) yields the uniform ellipticity of the diffusion operator for the non-localized species. Nonnegativity of each $\mu^{n+1}(x,\cdot)$ follows from the nonnegativity of $\mu_0^\infty$, $\mu^n$, $\Gamma_\ep$ and all $\wth_r$, from the rate of removal of mass being proportional to mass itself and from standard maximal inequality arguments. Recall that for definiteness, we also set $\mu_t^{n+1}(x,y)=0$ for all $x\in \cT_{L}$ and $y\neq \bar{y}_x$.

Now that we have defined the sequence $\mu^n$, let us state the mysterious \underline{\bf Assumption} appearing in Proposition~\ref{thm:densityN}. Recall the definition of $\|\varphi\|_1$ given in \eqref{def LP}.
\begin{enumerate}
\item [(B4)] For every $T>0$, there exists $C_T<\infty$ such that
$$
\sup_{n\ge 0} \sup_{0\leq t\le T} \| \mu^n_t\|_1\le C_T.
$$
\end{enumerate}
Just like Assumptions~(A3) and (B2), because of the generality of our formulation, we believe that this condition can only be checked case by case.

Let us first focus on \eqref{eq:4.19NL}. By construction, for every $\varphi_x\in \cC^{2}(E)$ with $\nabla_y\varphi_x(y)\cdot n(y)=0$ for all $y\in \partial E$, we have
 \begin{align}
\int_E& \ell_E(dy) \varphi_x(y)\mu^{n+1}_t(x,y) \label{eq:4.20}\\
 =& \int_E \ell_E(dy) \varphi_x(y)\mu^{\infty}_0(x,y) + \int_0^t\int_{E}ds\ell_E(dy)\big(\sigma_x^2(y)\Delta_y\varphi_x(y)-b_x(y)\cdot\nabla_y\varphi_x(y)\big)\mu^{n+1}_s(x,y) \nonumber\\
&-\int_0^t ds\sum_{r\in R} \sum_{i=1}^{k_r}\1_{A^r_i}(x)\int_{E}\ell_E(dy)\varphi_x(y)\mu^{n+1}_s(x,y) \nonumber\\
& \qquad \qquad \qquad \qquad
\times \int_{E}\varrho_r(d\bar y)  \wth_r(\bar y,\mu^{n}_s)\Gamma_\ep(y-\bar y)\prod_{\substack{j=1\\j\neq i}}^{k_r}\bigg\{\int_E \varrho^{A_j^r}(dy_j)\Gamma_\ep(y_j-\bar y)\mu^{n}_s({A_j^r},y_j)\bigg\}\nonumber\\
&+\int_0^t ds\sum_{r\in R_{NL}}\bigg(\sum_{i=1}^{k'_r}\1_{B^r_i}(x)\bigg)\int_{E}\ell_E(dy)\varphi_x(y)\wth_r(y,\mu^{n}_s) \prod_{j=1}^{k_r}\bigg\{\int_E \varrho^{A^r_j}(dy_j)\Gamma_\ep(y_j-y)\mu^{n}_s(A^r_j,y_j)\bigg\} \nonumber\\
&+\int_0^t ds\sum_{r\in R_{L}}\bigg(\sum_{i=1}^{k'_r}\1_{B^r_i}(x)\bigg) \varphi_x(\bar{y}_r)\wth_r(\bar{y}_r,\mu^{n}_s) \prod_{j=1}^{k_r}\bigg\{\int_E \varrho^{A^r_j}(dy_j)\Gamma_\ep(y_j-\bar{y}_r)\mu^{n}_s(A^r_j,y_j)\bigg\}. \nonumber
\end{align}
As in the proof of Lemma~\ref{lem:uniqueness of limit}, the above equation has a mild form which can be written using the density $p^t_x(\cdot, \cdot)$ of $P^t_x\varphi_x(y)=\int_E\ell_E(dy') p^t_x(y,y')\varphi_x(y')$. Indeed, the regularity of the diffusion and drift coefficients stated in Assumption~(B3), together with boundedness of $E$ and smoothness of its boundary $\partial E$ imply that a unique such density $p^t_x(\cdot,\cdot)$ exists which is continuous in $t,y,y'$ for all $x\in\cT_{NL}$ (see Lemma~4.5 in \cite{ChM07}, and note that the original Sato-Ueno result allows $\partial E$ to have finitely many piecewise smooth components). Hence, for  all $x\in\cT_{NL}$ and each continuous function $\varphi_x$ on $E$, we have
\begin{align*}
\int_E &\ell_E(dy)\, \varphi_x(y) \mu^{n+1}_t(x,y) =  \int_{E}\ell_E(dy)\int_{E}\ell_E(dy')p_x^{t}(y,y')\varphi_x(y')\mu_0^\infty(x,y)\\
 &-\int_0^t ds\sum_{r\in R} \sum_{i=1}^{k_r} \1_{A^r_i}(x) \int_{E}\ell_E(dy) \int_E \ell_E(dy')p_x^{t-s}(y,y') \varphi_x(y')\mu^{n+1}_s(x,y)\\&
 \qquad\qquad \qquad \qquad \times \int_E \varrho_r(d\bar y) \wth_r(\bar y, \mu^n_s)\Gamma_\ep(y-\bar y) \prod_{\substack{j=1\\j\neq i}}^{k_r}\bigg\{\int_E \varrho^{A_j^r}(dy_j)\Gamma_\ep(y_j-\bar y)\mu^{n}_s({A_j^r},y_j)\bigg\}\\
&+\int_0^t ds\sum_{r\in  R_{NL}} \bigg(\sum_{i=1}^{k'_r}\1_{B^r_i}(x)\bigg)\int_{E}\ell_E(dy)\int_E \ell_E(dy') p_x^{t-s}(y,y')\varphi_x(y')\wth_r(y,\mu^n_s) \\
& \qquad \qquad \qquad \qquad \times \prod_{j=1}^{k_r}\bigg\{\int_E \varrho^{A^r_j}(dy_j)\Gamma_\ep(y_j-y)\mu^{n}_s(A^r_j,y_j)\bigg\}\\
&+\int_0^t ds \sum_{r\in R_{L}} \bigg(\sum_{i=1}^{k'_r}\1_{B^r_i}(x)\bigg)\int_{E}\ell_E(dy') p_x^{t-s}(\bar{y}_r,y')\varphi_x(y')\wth_r(\bar{y}_r,\mu^n_s) \\
& \qquad \qquad \qquad \qquad \times \prod_{j=1}^{k_r}\bigg\{\int_E \varrho^{A_j^r}(dy_j)\Gamma_\ep(y_j-\bar{y}_r)\mu^{n}_s(A^r_j,y_j)\bigg\}.
 \end{align*}
Fubini's theorem then implies that for all $y'\in E$,
\begin{align}
 & \mu^{n+1}_t(x,y') = \int_E \ell_E(dy) p_x^{t}(y,y')\mu_0^\infty(x,y)dy \label{mild density}\\
  &-\int_0^t ds\sum_{r\in R}\sum_{i=1}^{k_r}\1_{A^r_i}(x)\int_{E}\ell_E(dy) p_x^{t-s}(y,y')\mu^{n+1}_s(x,y)\int_E \varrho_r(d\bar y) \wth_r(\bar y,\mu^n_s) \Gamma_\ep(y-\bar y)\nonumber \\
  & \qquad \qquad \qquad \qquad \times \prod_{\substack{j=1\\j\neq i}}^{k_r}\bigg\{\int_E \varrho^{A_j^r}(dy_j) \Gamma_\ep(y_j-\bar y)\mu^{n}_s({A_j^r},y_j)\bigg\} \nonumber \\
& +\int_0^t ds\sum_{r\in R_{NL}} \bigg(\sum_{i=1}^{k'_r}\1_{B^r_i}(x)\bigg)\int_{E}\ell_E(dy) p_x^{t-s}(y,y')\wth_r(y,\mu^n_s) \prod_{j=1}^{k_r}\bigg\{\int_E \varrho^{A_j^r}(dy_j) \Gamma_\ep(y_j- y)\mu^{n}_s(A^r_j,y_j)\bigg\}\nonumber\\
& +\int_0^t ds\sum_{r\in R_{L}} \bigg(\sum_{i=1}^{k'_r}\1_{B^r_i}(x)\bigg)p_x^{t-s}(\bar{y}_r,y')\wth_r(\bar{y}_r,\mu^n_s) \prod_{j=1}^{k_r} \bigg\{\int_E \varrho^{A_j^r}(dy_j)\Gamma_\ep(y_j-\bar{y}_r)\mu^{n}_s(A^r_j,y_j)\bigg\}. \nonumber
\end{align}
Integrating \eqref{eq:4.19L} with respect to time, for every $x\in \cT_{L}$ we directly obtain an analogue of \eqref{mild density} where the measure $\ell_E(dy)\, p_x^{t-s}(y,y')$ is replaced by a Dirac mass at $\bar{y}_x$. Let us now show that the sequence $\{\mu^{n},\, n\geq 0\}$ converges as $n\rightarrow \infty$ and that the limit satisfies \eqref{pde density} and \eqref{de density}.

Let $T>0$ and recall from Assumption~(B4) that we assume that we can prove the existence of $C_T<\infty$ such that
$$
\sup_{n\geq 0}\sup_{t\in [0,T]} \|\mu^n_t\|_1 \leq C_T.
$$
For $x\in \cT_{NL}$ and $y'\in E$, we can write
\begin{align*}
 & |\mu^{n+1}_t(x,y')-\mu^{n}_t(x,y')|  \\
  &=\Bigg|\int_0^t ds\sum_{r\in R}\sum_{i=1}^{k_r}\1_{A^r_i}(x)\int_{E}\ell_E(dy) p_x^{t-s}(y,y')\\
  &\qquad \bigg(\mu^{n+1}_s(x,y)
  \int_{E}\varrho_r(d\bar y)\wth_r(\bar y,\mu^n_s) \Gamma_\ep(y-\bar y)\prod_{\substack{j=1\\j\neq i}}^{k_r}\bigg\{\int_E \varrho^{A_j^r}(dy_j) \Gamma_\ep(y_j-\bar y)\mu^{n}_s({A_j^r},y_j)\bigg\}\\
  &\qquad-\mu^{n}_s(x,y)
  \int_E \varrho_r(d\bar y) \wth_r(\bar y,\mu^{n-1}_s) \Gamma_\ep(y-\bar y) \prod_{\substack{j=1\\j\neq i}}^{k_r} \bigg\{\int_E \varrho^{A_j^r}(dy_j) \Gamma_\ep(y_j-\bar y)\mu^{n-1}_s({A_j^r},y_j)\bigg\}\bigg)\Bigg|\\
& +\Bigg|\int_0^t ds\sum_{r\in R_{NL}}\bigg(\sum_{i=1}^{k'_r}\1_{B^r_i}(x)\bigg)\int_{E}\ell_E(dy) p_x^{t-s}(y,y') \bigg(\wth_r(y,\mu^n_s)\prod_{j=1}^{k_r}\bigg\{\int_E \varrho^{A_j^r}(dy_j) \Gamma_\ep(y_j- y)\mu^{n}_s(A^r_j,y_j)\bigg\} \\
& \qquad \qquad \quad -\wth_r(y,\mu^{n-1}_s)\prod_{j=1}^{k_r}\bigg\{\int_E \varrho^{A_j^r}(dy_j) \Gamma_\ep(y_j- y)\mu^{n-1}_s(A^r_j,y_j)\bigg\}\bigg)\Bigg|\\
& +\Bigg|\int_0^t ds\sum_{r\in R_{L}}\bigg(\sum_{i=1}^{k'_r}\1_{B^r_i}(x)\bigg)p_x^{t-s}(\bar{y}_r,y') \bigg(\wth_r(\bar{y}_r,\mu^n_s)\prod_{j=1}^{k_r}\bigg\{\int_E \varrho^{A_j^r}(dy_j) \Gamma_\ep(y_j- \bar{y}_r)\mu^{n}_s(A^r_j,y_j)\bigg\}\\
& \qquad \qquad \quad-\wth_r(\bar{y}_r,\mu^{n-1}_s) \prod_{j=1}^{k_r}\bigg\{\int_E \varrho^{A_j^r}(dy_j) \Gamma_\ep(y_j- \bar{y}_r)\mu^{n-1}_s(A^r_j,y_j)\bigg\}\bigg)\Bigg|.
\end{align*}
Analogous expressions trivially hold for $|\mu^{n+1}_t(x,\bar{y}_x)-\mu^n_t(x,\bar{y}_x)|$ when $x\in \cT_{L}$. Considering each molecular species (non-localized and localized) separately and, for a given species, each term in the above sums over $R$, $R_{NL}$ and $R_{L}$ one by one, we may then proceed exactly as in the proof of Lemma~\ref{lem:uniqueness of limit} and obtain the existence of a finite constant $C_T'$, which depends on the finitely many parameters $k_r$, $k_r'$, $L_r$, $\|\Psi_r\|_\infty$, $\mathcal{S}_r$, $C_T$ and $\mathrm{Vol}(E)$ but not on $n$, and such that for all $t\in [0,T]$ and all $n\geq 0$,
\begin{equation}
\sup_{s\in [0,t]}\|\mu_s^{n+1} - \mu_s^n\|_1 \leq C_T'\int_0^t ds\, \bigg(\sup_{u\in [0,s]}\|\mu_u^{n+1} - \mu_u^n\|_1 + \sup_{u\in [0,s]}\|\mu_u^{n} - \mu_u^{n-1}\|_1\bigg).
\end{equation}
Using Gronwall's lemma, we find that there exists another constant $C_T''<\infty$ such that for every $t\in [0,T]$ and all $n\geq 1$,
$$
\sup_{s\in [0,t]}\|\mu_t^{n+1} - \mu_t^n\|_1 \leq C_T''\int_0^t ds\, \sup_{u\in [0,s]}\|\mu_u^{n} - \mu_u^{n-1}\|_1.
$$
Picard's iteration proof gives us that
$$
\sum_{n\ge 0} \sup_{t\in [0,T]}\|\mu^{n+1}_t -\mu^n_t\|_1<\infty,
$$
which implies that the sequence $\{\mu_t^n,\, n\ge 0\}$ converges uniformly over $t\in [0,T]$ to a function $\mu_t$ on $\cP$ which furthermore satisfies
$$
\sup_{t\in [0,T]}\|\mu_t\|_1 \le C_T.
$$
In particular, each $\mu_t$ belongs to the set $\mathbf{D}$ of densities. Using the uniform convergence and passing to the limit in \eqref{eq:4.20} (and in the analogous equations for localized species), we obtain that $(\mu_t)_{t\geq 0}$ satisfies the partial and ordinary integro-differential equations \eqref{eq:4.19NL} and \eqref{eq:4.19L}. Since by assumption (see Theorem~\ref{thm:largeNbis}) the measure-valued solution to \eqref{deter limit} is unique, we can finally conclude that for every $t\geq 0$, $\mu_t$ is a density for $M_t^\infty$ in the sense of \eqref{decomp Mt}, and Proposition~\ref{thm:densityN} holds true.
\end{proof}

\begin{remark}
In the more general case where species of low abundances are present, the characterization of the limiting process in Theorem~\ref{thm:largeNbis} as a measure-valued PDMP suggests that under conditions similar to the assumptions of Proposition~\ref{thm:densityN}, the continuous part $M_t^{\infty,c}$ may be written as
\begin{equation}\label{density PDMP}
M_t^{\infty,c}=\sum_{x\in \cT_{NL}}\mu_t(x,y)\delta_x\otimes \ell_E + \sum_{x\in \cT_{L,b}} \mu_t(x,\bar{y}_x)\delta_x \otimes \delta_{\bar{y}_x},
\end{equation}
where $(\mu_t)_{t\geq 0}$ satisfies a system of partial integro-differential equations of the same form as (\ref{pde density}--\ref{de density}), whose parameters may change in a stochastic way when the discrete part $(M_t^{\infty,d})_{t\geq 0}$ jumps (recall that the terms $\wth_r(\bar y,\langle M_t^\infty,\Psi_{r,\bar y}\rangle)$ a priori depend on both the continuous and the discrete part of $M_t^\infty$).
We leave this exercise to the reader as it is notationally very heavy, but we expect that most of the proof of Proposition~\ref{thm:densityN} can be reused in a very straightforward way to prove such a result (at least in particular cases), since the effect of the discrete component on the continuous one is only through $\tilde h_r$.
\end{remark}

\section*{Acknowledgements} This research was supported by NSERC (Natural Sciences and Engineering Research Council of Canada) and the CRM (Centre de Recherches Math\'ematiques) -UMI travel allowance. AV was also supported in part by the \emph{chaire Mod\'elisation Math\'ematique et Biodiversit\'e} of Veolia Environnement-\'Ecole Polytechnique-Museum National d'Histoire Naturelle-Fondation X. The authors are grateful to the two reviewers for their particularly useful and constructive comments on previous versions of the manuscript, which helped to improve readability and to simplify some technical aspects of the construction of the process.

\appendix
\section{Proof of Theorem~\ref{thm:existence}}\label{s: proof Th1}
We first state and prove two lemmas which will be used later.
\begin{lemma}\label{lem:bound lambda}
Under Assumptions (A1) and (A2), for every $r\in R$, $M\in \cM$, $(p_1,\ldots,p_{k_r})\in \cP^{k_r}$ and $\bar y\in E$, we have
$$
\big|\lambda_r(\bar y,M;p_1,\ldots,p_{k_r})\big| \leq \|h_r\|_{\infty,\ell(r,M)}\times \|\Gamma_\ep\|_{\infty}^{k_r},
$$
where $\lambda_r(\bar y,M;p_1,\ldots,p_{k_r})$ is the local reaction rate at $\bar y$ of particles $p_1,\ldots,p_{k_r}$ defined in (\ref{eq:lambda}) and
$$
\ell(r,M)=\|\Psi_r\|_{\infty}\langle M,1\rangle.
$$
In addition, the function $M\mapsto \lambda_r(\bar y,M;p_1,\ldots,p_{k_r})$ is continuous on $\cM$.
\end{lemma}

\begin{proof}{(Proof of Lemma~\ref{lem:bound lambda}.)} The bound is a straightforward consequence of the Assumptions and of the fact that $|\langle M,\Psi_{r,\bar y}\rangle|\leq \|\Psi_r\|_{\infty}\langle M,1\rangle$. The continuity of the mapping $M\mapsto \lambda_r(\bar y,M;p_1,\ldots,p_{k_r})$ comes from the facts that, according to Assumption~(A1), $h_r$ is Lipschitz in its second coordinate and $\Psi_{r,y}$ is continuous (and thus bounded) on $\cP$.
\end{proof}

\begin{lemma}\label{lem: Lambda}
Suppose Assumptions~(A1) and (A2) are satisfied. Then the function $\Lambda_r$ defined in~\eqref{rate Lambda} is continuous on $\cM_p$.
\end{lemma}

\begin{proof} {(Proof of Lemma~\ref{lem: Lambda}.)} Let $r\in R$.

Let us first derive a bound on $|\Lambda_r(M_1)-\Lambda_r(M_2)|$ for every $M_1,M_2\in \cM_p$. We shall then use this bound to show that if $(M_n)_{n\geq 0}$ converges weakly to $M$ in $\cM_p$, then $\Lambda_r(M_n)\rightarrow \Lambda_r(M)$ as $n\rightarrow \infty$.

Let thus $M_1,M_2\in \cM_p$. Recalling the definition (\ref{eq:lambda}) of $\lambda_r$ and our notation $p_i=(x_i,y_i)$, and leaving aside the integral over $E$ for a moment, we have
\begin{align}
& \int_{\cP^{k_r}}M_1^{\otimes \downarrow k_r}(dp_1,\ldots, dp_{k_r}) \, \Bigg(\prod_{i=1}^{k_r}\1_{A_i^r}(x_i)\Gamma_\ep(y_i-\bar y)\Bigg){\bar h}_r(\bar y, M_1)\nonumber\\
& =  \int_{\cP^{k_r}}M_1^{\otimes \downarrow k_r}(dp_1,\ldots, dp_{k_r})\, \Bigg(\prod_{i=1}^{k_r}\1_{A_i^r}(x_i)\Gamma_\ep(y_i-\bar y)\Bigg)
\big({\bar h}_r(\bar y, M_1)-{\bar h}_r(\bar y, M_2)\big)\nonumber\\
& \qquad + \int_{\cP^{k_r}}M_1^{\otimes \downarrow k_r}(dp_1,\ldots,dp_{k_r})\, \Bigg(\prod_{i=1}^{k_r}\1_{A_i^r}(x_i)\Gamma_\ep(y_i-\bar y)\Bigg){\bar h}_r(\bar y, M_2). \label{diff Lambda}
\end{align}
By Assumption~(A1), the reaction factor $h_r$ is Lipschitz in its second coordinate, with Lipschitz constant $L_r$ independent of the first coordinate, and so the absolute value of the first term on the r.h.s. of (\ref{diff Lambda}) is bounded by
\begin{align}
& \|\Gamma_\ep\|_{\infty}^{k_r}\, \langle M_1,1\rangle^{k_r}\, \big|h_r\big(\bar y,\langle M_1,\Psi_{r,\bar y}\rangle\big)-h_r\big(\bar y,\langle M_2,\Psi_{r,\bar y}\rangle\big)\big| \nonumber \\
& \leq \|\Gamma_\ep\|_{\infty}^{k_r}\langle M_1,1\rangle^{k_r} \, L_r\big|\langle M_1,\Psi_{r,\bar y}\rangle-\langle M_2,\Psi_{r,\bar y}\rangle\big|. \label{approx 1}
\end{align}
Integrating now with respect to $\varrho_r(d\bar y)$, we can thus write
\begin{align}
\big|\Lambda_r(M_1)-\Lambda_r(M_2)\big| \leq & \|\Gamma_\ep\|_{\infty}^{k_r}\langle M_1,1\rangle^{k_r} \, L_r\int_{E}\varrho_r(d{\bar y})\, \big|\langle M_1,\Psi_{r,\bar y}\rangle-\langle M_2,\Psi_{r,\bar y}\rangle\big| \label{approx 2}\\
& \ + \int_E \varrho_r(d\bar y)\,{\bar h}_r({\bar y},M_2)\bigg|\Big\langle M_1^{\otimes\downarrow k_r}-M_2^{\otimes \downarrow k_r}, \prod_{i=1}^{k_r}\1_{A_i^r}(x_i)\Gamma_\ep(y_i-{\bar y})\Big\rangle \bigg|. \nonumber
\end{align}
Now suppose that $(M_n)_{n\geq 1}$ is a sequence of measures in $\cM_p$ converging weakly to $M\in \cM_p$, and consider the expression in (\ref{approx 2}) with $M_1=M_n$ and $M_2=M$. First, $\langle M_n,1\rangle\rightarrow \langle M,1\rangle$ and so there exists $n_0$ such that for every $n\geq n_0$, we have $\langle M_n,1\rangle\leq 2\langle M,1\rangle$. Thanks to the continuity and uniform boundedness of $\Psi_{r,\bar y}$ stated in Assumption~(A1) and the fact that $E$ has finite volume, we can use the dominated convergence theorem to conclude that
$$
\lim_{n\rightarrow \infty} \|\Gamma_\ep\|_{\infty}^{k_r}\langle M_n,1\rangle^{k_r} \, L_r\int_{E}\varrho_r(d{\bar y})\, \big|\langle M_n,\Psi_{r,\bar y}\rangle-\langle M,\Psi_{r,\bar y}\rangle\big|=0.
$$
For the second term on the r.h.s. of (\ref{approx 2}), first observe that by Assumption~(A1),
$$
\forall {\bar y}\in E, \qquad \big|{\bar h}_r({\bar y},M)\big| \leq \|h_r\|_{\infty,\|\Psi_r\|_\infty\langle M,1\rangle}<\infty.
$$
Together with the continuity of $\Gamma_\ep$, the fact that the set $\cT$ of molecular types is finite and the weak convergence of the counting measure $M_n$ to the counting measure $M$, this allows us to conclude that the second term on the r.h.s. of (\ref{approx 2}) (again, with $M_1=M_n$ and $M_2=M$) converges to $0$ as $n\rightarrow \infty$ and consequently
$$
\lim_{n\rightarrow \infty}\big|\Lambda_r(M_n)-\Lambda_r(M)\big| = 0.
$$
This concludes the proof of the continuity of $\Lambda_r$ on $\cM_p$.
\end{proof}

We can now proceed to the proof of Theorem~\ref{thm:existence}.
\begin{proof}{(Proof of Theorem~\ref{thm:existence}.)} By construction, $(M_t)_{0\leq t<\tau^\infty}$ is an $\cM_p$-valued c\`adl\`ag process a.s. and is Markovian. Indeed, by Assumption~(A0) the movement of particles between two reaction times is given by a c\`adl\`ag Markov process. In addition, in the time interval $[\tau^j,\tau^{j+1})$ we have $\widetilde{M}^j_t = M_t$, and so the integrals appearing in the definition of the random times $\tau^{j+1}_r$ are adapted to the natural filtration $(\mathcal{F}_t)_{t\geq 0}$ of $(M_t)_{t\geq 0}$ (for the filtration to be well-defined for all times, one may add a cemetery state $\partial$ and declare that $M_t=\partial$ for all $t\geq \tau^\infty$). This tells us that each $\tau_r^{j+1}$ (for $r\in R$), and hence $\tau^{j+1}$, is a stopping time for $(\mathcal{F}_t)_{t\geq 0}$, for every $j\geq 0$; for every $t\geq 0$ and $k\in \bN$, the event $\{j(t)=k\}=\{\tau^k\leq t\}\setminus \{\tau^{k+1}\leq t\}$ is thus $\mathcal{F}_t$-measurable. Using the lack of memory property of the exponential r.v., on the event $\{j(t)=k\}$ the distribution of $\tau^{k+1}$ conditionally on ${\cal F}_t$ is equal to its distribution conditionally on $M_t$. Finally, since the sampling rule and updating of the measure at time $\tau^{k+1}$ depends only on the state of the process at time $(\tau^{k+1})-$, we can conclude that $(M_t)_{0\leq t <\tau^\infty}$ has the Markov property.

The fact that $\tau^\infty=+\infty$ a.s. is an easy consequence of the first part of Assumption~(A3), namely that
\begin{equation}\label{eq:sup}
S_T:= \sup_{t\in [0,T]}\langle M_t,1\rangle <\infty \qquad \hbox{a.s.},
\end{equation}
for any fixed $T>0$. Indeed, let us fix $T>0$ and show that $\tau^\infty >T$ with probability one. Recalling the bound on $\lambda_r$ stated in Lemma~\ref{lem:bound lambda}, we simply have to observe that for every $r\in R$ and $t\in [0,T]$, we have
\begin{align*}
\int_{E\times \cP^{k_r}}& \varrho_r(d\bar y) M_t^{\otimes \downarrow k_r}(dp_1,\ldots,dp_{k_r})\, \lambda_r\big(\bar y,M_t;p_1,\ldots,p_{k_r}\big)\\
& \leq (1+\mathrm{Vol}(E))S_T^{k_r}\|\Gamma_\ep\|_{\infty}^{k_r}\|h_r\|_{\infty,\|\Psi_r\|_\infty S_T}<\infty \quad \mathrm{a.s.}
\end{align*}
Consequently, each $\tau^j$ (until time $T$) is stochastically bounded from below by the minimum of $|R|$ independent exponentially distributed random variables whose parameters are independent of $j$. The number of $j$'s such that $\tau^j\leq T$ is thus a.s. finite, which is equivalent to the property $\tau^\infty>T$ a.s. that we were seeking. The $\cM_p$-valued process $(M_t)_{t\geq 0}$ is therefore defined for all times. 

Finally, let us show that $(M_t)_{t\geq 0}$ satisfies the martingale problem MP$(L)$ with initial distribution $\mathcal{L}(M_0)$, in the particular case where $\mathcal{L}(M_0)=\delta_{m_0}$ for some fixed $m_0\in \cM_p$. The general case can be obtained by integrating with respect to the law of $M_0$. Using Lemma~\ref{lem:bound lambda} and the fact that until $\tau^1$, the (finite) total mass of the system remains unchanged, we have for every $r\in R$ and every $s< \tau^1$:
\begin{align*}
\Lambda_r\big(\widetilde{M}_s^0\big)&= \int_{E\times \cP^{k_r}} \varrho_r(d{\bar y}) \Big(\widetilde{M}_s^0\Big)^{\otimes \downarrow k_r}(dp_1,\ldots,dp_{k_r})\, \lambda_r(\bar y,\widetilde{M}^0_s;p_1,\ldots,p_{k_r})\\
&  \leq (1+\mathrm{Vol}(E))\langle M_0,1\rangle^{k_r}\|\Gamma_\ep\|_{\infty}^{k_r}\|h_r\|_{\infty,\ell(r,M_0)}<\infty \qquad \hbox{a.s.},
\end{align*}
where we recall that $\ell(r,M_0)=\|\Psi_r\|_{\infty}\langle M_0,1\rangle< \infty$. Consequently, for every $t\geq 0$ we can write
\begin{align}
\bP\big(\tau^1>t\big)& =\bE\Big[\bP\big(\tau^1>t\,\big|\, \big(\widetilde{M}^0_s\big)_{s\geq 0}\big)\Big]= \bE\bigg[\prod_{r\in R}\bP\big(\tau_r^1>t\,\big|\, \big(\widetilde{M}^0_s\big)_{s\geq 0}\big)\bigg]\nonumber\\
& = \bE\bigg[\exp\bigg\{-\sum_{r\in R}\int_0^tds\, \Lambda_r\big(\widetilde{M}^0_s\big)\bigg\}\bigg] \nonumber\\
& = 1 - \sum_{r\in R}\bE\bigg[\int_0^tds\, \Lambda_r\big(\widetilde{M}^0_s\big) \bigg] + \mathcal{O}\big(t^2\big). \label{DL proba}
\end{align}
The same reasoning applies to $\tau^2-\tau^1$, since the total mass of the system between times $\tau^1$ and $\tau^2$ is bounded by $\langle M_0,1\rangle + \sup_{r\in R}(k_r'-k_r)<\infty$. Let us now write for every function $F_f$ of the form (\ref{eq:Ff}) belonging to the set $\mathbf{F}$ of Assumption~(A0): 
\begin{align}
\frac{1}{t}\,\bE\big[F_f(M_t)-F_f(M_0)\big] = & \ \frac{1}{t}\,\bE\big[\big(F_f(M_t)-F_f(M_0)\big)\1_{\{\tau^1>t\}}\big] \label{generator}\\
&  + \frac{1}{t}\,\bE\big[\big(F_f(M_t)-F_f(M_0)\big)\1_{\{\tau^1\leq t\}}\big]. \nonumber
\end{align}
Since $\1_{\{\tau^1>t\}}\rightarrow 1$ a.s. as $t\rightarrow 0$, $M_t=\widetilde{M}_t^0$ evolves only through the spatial movement of the particles initially described by $M_0$ as long as $t<\tau^1$ and the function $F_f$ belongs to the domain of $\cD$, we have that
\begin{equation}\label{term 1}
\lim_{t\rightarrow 0}\frac{1}{t}\,\bE\big[\big(F_f(M_t)-F_f(M_0)\big)\1_{\{\tau^1>t\}}\big] = \cD F_f(M_0).
\end{equation}
Next, we have
\begin{equation}\label{term 2}
\frac{1}{t}\,\bE\big[\big(F_f(M_t)-F_f(M_0)\big)\1_{\{\tau^1\leq t\}}\big]= \bE\big[\big(F_f(M_t)-F_f(M_0)\big)\,\big|\, \tau^1\leq t\big]\frac{\bP(\tau^1\leq t)}{t}.
\end{equation}
Using (\ref{DL proba}) and the continuity of $\Lambda_r$ shown in Lemma~\ref{lem: Lambda}, we obtain that
$$
\lim_{t\rightarrow 0} \frac{\bP(\tau^1\leq t)}{t} = \sum_{r\in R}\Lambda_r(M_0).
$$
Furthermore, conditionally on $\tau^1$, the lapse of time $\tau^2-\tau^1$ between the first and second reaction satisfies the same type of asymptotics as in (\ref{DL proba}) and so the probability that two reactions occur before time $t$ (\emph{i.e.}, that $\tau^2\leq t$) is of the order of $\mathcal{O}(t^2)$. As a consequence, integrating over all possible values of $M_{\tau^1}$ and disregarding the unlikely event that two reactions occur before time $t$ yields
\begin{align*}
& \lim_{t\rightarrow 0}\frac{1}{t} \,\bE\big[\big(F_f(M_t)-F_f(M_0)\big)\1_{\{\tau^1\leq t\}}\big]\\
& = \bigg(\sum_{r'\in R}\Lambda_{r'}(M_0)\bigg)\Bigg\{\sum_{r\in R}\frac{\Lambda_r(M_0)}{\sum_{r'\in R}\Lambda_{r'}(M_0)}\int_{E\times \cP^{k_r}}\varrho_r(d{\bar y})M_0^{\otimes \downarrow k_r}(dp_1,\ldots,dp_{k_r})\\
& \qquad \frac{\lambda_r(\bar y,M_0;p_1,\ldots,p_{k_r})}{\Lambda_r(M_0)}\, \bigg[F\bigg(\bigg\langle M_0-\sum_{i=1}^{k_r}\delta_{p_i}+\sum_{i=1}^{k_r'}\delta_{(B_i^r,\bar y)},f\bigg\rangle\bigg)
-F_f(M_0)\bigg]\\
& = \sum_{r\in R} G_rF_f(M_0),
\end{align*}
where in the first equality, each summand is the product of the probability that reaction $r$ is the first to occur and of the integral describing the sampling of a location and a set of source reactants as specified in our construction of $(M_t)_{t\geq 0}$. Note that the continuity of $\lambda_r$ and $\Lambda_r$ on $\cM_p$, stated respectively in Lemmas~\ref{lem:bound lambda} and~\ref{lem: Lambda}, are used here to obtain that the sampling of the location and particles at time $0+$ is made according to $\lambda_r(\bar y,M_0;p_1,\ldots,p_{k_r})/\Lambda_r(M_0)$. Together with (\ref{generator}) and (\ref{term 1}), this gives us that
$$
\lim_{t\rightarrow 0} \frac{1}{t}\,\bE\big[F_f(M_t)-F_f(M_0)\big]= LF_f(M_0).
$$
Using the bound on $\cD F_f$ stated in Assumption~(A0) and the bound on the $(1+ \max_{r\in R}k_r)$-th moment of the total mass stated in Assumption~(A3) together with the inequalities
\begin{align*}
|G_rF_f(M_s)|& \leq 2\|F\|_\infty \|\Gamma_\ep\|^{k_r}_\infty\int_{E \times \cP^{k_r}}\varrho_r(d\bar y) M_s^{\otimes \downarrow k_r}(dp_1,\ldots,dp_{k_r})\, h_r(\bar y,\langle M_s,\Psi_{r,\bar y}\rangle)\\
& \leq 2\|F\|_\infty \|\Gamma_\ep\|^{k_r}_\infty (1+\mathrm{Vol}(E))\langle M_s,1\rangle^{k_r} \big(h_r(\bar y,0) + L_r\langle M_s,\Psi_{r,\bar y}\rangle\big)\\
& \leq 2\|F\|_\infty \|\Gamma_\ep\|^{k_r}_\infty (1+\mathrm{Vol}(E))\langle M_s,1\rangle^{k_r} \big(h_r(\bar y,0) + L_r\|\Psi_r\|_\infty \langle M_s,1\rangle\big),
\end{align*}
where the second inequality uses the fact that $h_r$ is Lipschitz in its second coordinate by Assumption~(A1),  we can then write that for every $t>0$ and $F_f$ as above,
$$
F_f(M_t)-\int_0^t ds\, LF_f(M_s)
$$
is integrable. Extending the above computations to $F_f(M_{t_2})-F_f(M_{t_1})$ for any $0 \leq t_1<t_2$ and using the Markov property of $(M_t)_{t\geq 0}$ allows us to conclude that $(M_t)_{t\geq 0}$ is a solution to MP$(L)$. Theorem~\ref{thm:existence} is proved. \end{proof}

\begin{remark}\label{rmk:justif K}
If instead of assuming that $|\cD F_f(M)|\leq c_{F,f}\langle M,1\rangle$ as in the statement of Assumption~(A0), we suppose that there exists $K\in \bN$ and $c_{F,f}>0$ such that
$$
|\cD F_f(M)|\leq c_{F,f}\langle M,1\rangle^K \qquad \hbox{for every }M\in \cM_p,
$$
we see from the end of the proof of Theorem~\ref{thm:existence} that the integrability of 
$$
F_f(M_t)-\int_0^t ds\, LF_f(M_s)
$$
will be a guaranteed by the stronger condition:
\begin{itemize}
\item[(A3')] For every $T>0$,
$$
\sup_{t\in [0,T]}\langle M_t,1\rangle <\infty \quad \hbox{a.s., and }\sup_{t\in [0,T]}\bE\Big[\langle M_t,1\rangle^{K\vee (1+\max_r k_r)}\Big]<\infty.
$$
\end{itemize}
This justifies the claim in Remark~\ref{rmk: moment bound}$(a)$. 
\end{remark}

\end{document}